\documentclass[11pt]{amsart}

\usepackage{amsmath,amssymb,amsthm,verbatim,amscd}
\usepackage{graphicx}
\usepackage{mathrsfs}
\usepackage{array}
\usepackage{multirow}
\usepackage{stmaryrd}
\usepackage[colorlinks=true,urlcolor=blue]{hyperref}
\usepackage[all]{xy}

\usepackage{fullpage}

\theoremstyle{plain}

\newtheorem{thm}{Theorem}[section]
\newtheorem{definition}[thm]{Definition}
\newtheorem{prop}[thm]{Proposition}
\newtheorem{lemma}[thm]{Lemma}
\newtheorem{cor}[thm]{Corollary}

\theoremstyle{definition}
\newtheorem{remark}[thm]{Remark}

\newtheorem{example}[thm]{Example}

\numberwithin{equation}{section}

\newcommand{\A}{{\mathcal A}}
\newcommand{\B}{{\mathcal B}}
\newcommand{\bA}{{\bf A}}
\newcommand{\D}{{\mathcal D}}
\newcommand{\bD}{{\bf D}}

\newcommand{\F}{{\mathcal F}}

\newcommand{\T}{{\mathbb T}}
\newcommand{\K}{{\mathcal K}}
\renewcommand{\L}{{\mathcal L}}
\renewcommand{\O}{{\mathcal O}}
\newcommand{\p}{{\mathfrak p}}
\renewcommand{\P}{{\mathcal P}}
\newcommand{\M}{{\mathcal M}}
\newcommand{\m}{{\mathfrak m}}

\newcommand{\OO}{\text{O}}

\newcommand{\W}{{\mathcal W}}
\newcommand{\Z}{{\mathbb Z}}
\newcommand{\Q}{{\mathbb Q}}

\newcommand{\C}{{\mathbb C}}
\newcommand{\Fp}{{\mathbb F}_p}

\newcommand{\Zp}{\Z_p}
\newcommand{\Zpx}{\Z_p^\times}

\newcommand{\Cpx}{\C_p^\times}
\newcommand{\Qp}{\Q_p}
\newcommand{\Cp}{\C_p}

\newcommand{\lra}{\longrightarrow}

\newcommand{\mat}{\begin{pmatrix} a & b \\ c & d \end{pmatrix}}
\newcommand{\smallmat}{\bigl( \begin{smallmatrix} a & b \\ c & d \end{smallmatrix} \bigr)}
\newcommand{\sigop}{\Sigma_0(p)}

\newcommand{\psmallmat}[4]{\left( \begin{smallmatrix} #1 & #2 \\ #3 & #4 \end{smallmatrix} \right)}
\newcommand{\smatrix}[4]{\left( \begin{smallmatrix} #1 & #2 \\ #3 & #4 \end{smallmatrix} \right)}

\newcommand{\wt}{\widetilde}
\newcommand{\ol}{\overline}
\renewcommand{\binom}[2]{\genfrac{(}{)}{0pt}{}{#1}{#2}}

\newcommand{\spec}{{\text{sp}}}

\DeclareMathOperator{\chr}{char}
\DeclareMathOperator{\rank}{rank}
\DeclareMathOperator{\disc}{disc}
\renewcommand{\mod}[1]{\text{ }(\operatorname{mod}\text{ }#1)}

\newcommand{\Dvrig}[1]{\bD_{#1}}

\newcommand{\Dkrig}{\Dvrig{k}}

\newcommand{\Doc}{\D^\dag}
\newcommand{\Aoc}{\A^\dag}

\newcommand{\Phis}{\widetilde{\Phi}}

\newcommand{\nus}{\widetilde{\nu}}
\newcommand{\mus}{\widetilde{\mu}}

\newcommand{\ds}{\displaystyle}

\DeclareMathOperator{\Div}{Div} 
\DeclareMathOperator{\Hom}{Hom}
\DeclareMathOperator{\ord}{ord} 

\DeclareMathOperator{\GL}{GL} 
\DeclareMathOperator{\SL}{SL}
 
\DeclareMathOperator{\Symb}{Symb}

\DeclareMathOperator{\Sym}{Sym} 
 
\DeclareMathOperator{\im}{im}
\DeclareMathOperator{\Fil}{Fil}

\DeclareMathOperator{\nil}{nil}

\DeclareMathOperator{\cont}{cont}

\DeclareMathOperator{\Spec}{Spec}

\newcommand{\MS}[1]{\Symb_{\Gamma}(#1)}
\newcommand{\MSo}[1]{\Symb_{\Gamma_0}(#1)}

\newcommand{\tate}[1]{\langle \langle #1 \rangle \rangle}

\newcommand{\bAr}{R}
\newcommand{\bArc}{\Lambda}  
\newcommand{\Wr}{W_m}
\newcommand{\cWr}{\W_m}
\newcommand{\ot}{\hat{\otimes}}
\newcommand{\otz}{\hat{\otimes}}
\newcommand{\bAs}{\bA \ot \bAr}
\newcommand{\bAsbig}{\bA[r] \ot \bAr}
\newcommand{\bDs}{\bD \ot \bAr}
\newcommand{\bDsc}{\bD^0 \ot \bArc}
\newcommand{\bDsbig}{\bD[r] \ot \bAr}
\newcommand{\bDo}{\bD^0}
\newcommand{\bAo}{\bA^{\!0}}
\newcommand{\bAro}{\bAr^0}

\newcommand{\bDso}{\bDo \otz \bAro}

\newcommand{\Docs}{\D^\dag(\bAr)}

\newcommand{\MSkord}{X_k^{\ord}}
\newcommand{\MSord}{X^{\ord}}

\newcommand{\aut}{K_{a,c,m}}
\newcommand{\autv}[3]{K_{#1,#2,#3}}

\newcommand{\tLambda}{\wt{\Lambda}}
\newcommand{\Gr}{\text{Gr}}
\DeclareMathOperator{\ad}{ad}

\begin{document}

\title[Hida families via overconvergent modular symbols]{Explicit computations of Hida families via overconvergent modular symbols}

\author[Dummit]{Evan P.\ Dummit}
\address[Dummit]{
Department of Mathematics\\
915 Hylan Building\\
University of Rochester\\
Rochester, NY 14627\\
USA
}
\email{edummit@ur.rochester.edu}

\author[Hablicsek]{M\'{a}rton Hablicsek}
\address[Hablicsek]{
Department of Mathematics\\
David Rittenhouse Lab\\
University of Pennsylvania\\
Philadelphia, PA 19104\\
USA
}
\email{mhabli@math.upenn.edu}

\author[Harron]{Robert Harron}
\address[Harron]{
Department of Mathematics\\
Keller Hall\\
University of Hawai`i at M\={a}noa\\
Honolulu, HI~96822\\
USA
}
\email{rharron@math.hawaii.edu}

\author[Jain]{Lalit Jain}
\address[Jain]{
Department of Mathematics\\
Van Vleck Hall\\
University of Wisconsin--Madison\\
Madison, WI 53706\\
USA
}
\email{jain@math.wisc.edu}

\author[Pollack]{Robert Pollack}
\address[Pollack]{
Department of Mathematics and Statistics\\
111 Cummington Mall\\
Boston University\\
Boston, MA 02215\\
USA
}
\email{rpollack@math.bu.edu}

\author[Ross]{Daniel Ross}
\address[Ross]{
Department of Mathematics\\
Van Vleck Hall\\
University of Wisconsin--Madison\\
Madison, WI 53706\\
USA
}
\email{ross@math.wisc.edu}

\thanks{Robert Harron was supported by NSA grant \#H98230-13-1-0223 during part of this project. Robert Pollack was supported by NSF grant DMS-1303302.}

\date{\today}

\begin{abstract}
In \cite{PS1}, efficient algorithms are given to compute with overconvergent modular symbols.  These algorithms then allow for the fast computation of  $p$-adic $L$-functions and have further been applied to compute rational points on elliptic curves ({\it e.g.}\ \cite{DarmonPollack,Mak}).  In this paper, we generalize these algorithms to the case of {\it families} of overconvergent modular symbols.  As a consequence, we can compute $p$-adic families of Hecke-eigenvalues, two-variable \mbox{$p$-adic} $L$-functions, $L$-invariants, as well as the shape and structure of ordinary Hida--Hecke algebras.
\end{abstract}

\maketitle

\tableofcontents

\section{Introduction}
In a seminal work from the mid-80s, Hida \cite{HidaGalRep,HidaIwasawaModules} introduced a theory of $p$-adic families of ordinary Hecke-eigenforms.    This work was generalized by Coleman \cite{Coleman} in the mid-90s to include non-ordinary forms which ultimately led to the Coleman--Mazur \cite{ColemanMazur} construction of the eigencurve---a rigid analytic space which parametrizes all finite slope Hecke-eigenforms.  Over the following decade, the theory of $p$-adic variation of automorphic forms blossomed with multiple constructions of eigenvarieties over a  wide class of reductive groups \cite{AshStevens,Emerton,UrbanEigenvarieties,AIS,AIP}.
Moreover, the consequences to number theory of the existence of $p$-adic families of automorphic forms have been profound with the proofs of the Mazur--Tate--Teitelbaum conjecture, the Main Conjecture (for class groups or modular forms!), and the Fontaine--Mazur conjecture (just to name a few) all heavily reliant upon the theory of $p$-adic variation.


With that said, our current state of understanding of the shape and structure of these eigenvarieties is still quite limited.  Taking the simplest example of ordinary forms on $\GL_2/\Q$ ({\it e.g}.\ the setting of Hida's original work), we do not have a good understanding of a single example of a Hecke--Hida algebra which is not simply a union of open discs.

This paper will attempt to rectify this situation at least in the case of classical Hida theory by introducing methods for computing with families of overconvergent modular symbols (which form the basis of Stevens'  construction of the eigencurve).  These methods generalize the constructions of \cite{PS1} where overconvergent modular symbols of a fixed weight were studied.  As a consequence, we can compute $q$-expansions of Hida families, two-variable $p$-adic $L$-functions, $L$-invariants of modular forms and their symmetric squares, and, moreover,  we can get our hands on the geometry of Hida families in several non-trivial situations.

As an example of some of the invariants we can compute, take $p=11$ and consider the 11-adic Hida family passing through Ramanujan's discriminant form $\Delta$.  In this case, the Hida family passing through $\Delta$ is parametrized by a single open disc.  Thus, for any prime $\ell$, the Hecke-eigenvalue of $T_\ell$ acting on the Hida family through Ramanujan's $\Delta$ is a power series in a weight variable $k$\footnote{Here, and throughout the paper, we are normalizing the weight variable $k$ to correspond to forms in $M_{k+2}(\Gamma)$.}  We compute for example
$$
a_{2}(k) = -2 - 41118k - 22748k^{2} + 37268k^{3} + 43923k^{4} + \OO(11^5,k^5).
$$
Note that we have that
 $$a_2(10) = a_2(\Delta) = -24$$ while $$a_2(0) = a_2(X_0(11)) = -2$$ as this Hida family specializes in weight 2 to the modular form associated with the elliptic curve $X_0(11)$.  
 The above approximation to $a_2(k)$ does indeed give these values modulo $11^5$.
 For other positive integer values of $k$, the above formula gives \mbox{$11$-adic} approximations to the coefficient of $q^2$ in the unique normalized $11$-ordinary form of weight $k$ and level 1. A higher precision approximation to $a_2(k)$ (and $a_\ell(k)$ for $\ell\leq11$) is given in Example~\ref{ex:X011_qexp}.

For another example, consider $p=3$ and fix a tame level $N=11$.  There are exactly two $3$-ordinary forms of weight 2 and level 33, and moreover, these forms are congruent modulo 3.  In particular, the Hida family attached to these forms cannot simply be the union of two open discs (because of the congruence between the two forms).  The possibilities for the geometry of this Hida family include two discs glued together at some collection of points or a double cover of weight space ramified at several points.  We note this example has already appeared in several places including \cite{GS2,PS1,Buzzardnote}.  Using the methods of this paper, we were able to determine that this Hida family is a double-cover of weight space ramified at a single (non-classical) weight, and moreover, this weight is congruent to $30060 \mod{3^{11}}$.

To discuss the methods of the paper, we introduce some notation.  Let $\bA$ denote the space of convergent power series on the closed unit disc, and let $\bD$ denote the space of distributions equal to the continuous $\Qp$-dual of $\bA$.  Let $\sigop$ denote the semigroup of matrices $\smallmat \in \text{M}_2(\Zp)$ with $a \in \Zpx$, $c \in p\Zp$, and with non-zero determinant.  For each weight $k$, one can endow $\bD$ with a weight $k$ action by $\sigop$, and we write $\bD_k$ for this space of distributions.  The space
$$\MSo{\bD_k}$$
is the collection of overconvergent modular symbols of level $\Gamma_0 = \Gamma_0(Np)$.  The systems of Hecke-eigenvalues occurring in this space are essentially the same as the systems which occur in  the space of finite slope overconvergent  modular forms of weight $k+2$ and level $\Gamma_0$ (see \cite[Theorem 7.1]{PS2}).  

As with overconvergent modular forms, these spaces can be $p$-adically interpolated over weight space.  To this end, set $D$ equal to a closed disc in weight space of radius $1/p$ about any tame character, and set  $R$ equal to the collection of convergent power series on $D$.  Then $\bDs$ sits inside of the space of $R$-valued distributions.  In particular, by evaluating at a weight $k$ in $D$, we obtain a specialization map $\bD \ot R \to \bD_k$, and we refer to elements of $\bDs$ as {\it families of distributions} on $D$.

Moreover, one can equip $\bDs$ with a $\sigop$-action which is simultaneously compatible with all of the specialization maps on $D$.\footnote{To work with other discs in weight space, one needs to replace $\bD$ with smaller spaces of distributions such as $\bD[r]$ with $r<1$.}  We thus interpret
$$
\MSo{\bDs}
$$
as the space of {\it families of overconvergent modular symbols} on $D$ of level $\Gamma_0$.   This space admits a Hecke-action, and we define the ordinary subspace $\MSo{\bDs}^{\ord}$ as the intersection of the images of all powers of $U_p$.  All of the information of $p$-adic Hida families of tame level $N$ is contained within this ordinary subspace (as Hida families extend to all of weight space).

In this paper, we introduce methods for explicitly computing approximations to elements of $\MSo{\bDs}$.  In particular, we are able to compute approximations to the characteristic polynomial of any Hecke operator acting on $\MSo{\bDs}^{\ord}$.  From these computations, one can then compute $q$-expansions of Hida families of eigenforms.  From this, one can compute $L$-invariants via the formulae of 
\cite{Hida-Linv, Harron-thesis,samit}.  Moreover, computing two-variable $p$-adic $L$-functions is immediate once one has a family of overconvergent eigensymbols in hand as in \cite{GS}.  Lastly, these computations also allow us to gain some fine control over the geometry of these Hida families in a wide variety of examples. We include several of the examples we computed below. In future work, we will implement non-trivial nebentypus (thus allowing for odd weights), coefficients in extensions of $\Zp$, and computations for the prime $p=2$. The primary bottleneck in speed is computing with $p$-adic polynomials in Sage. After speeding this up, our future work will also aim to compute examples more systematically.

The algorithms developed in this paper have been implemented in Sage \cite{sage} and continue to be developed on the SageMathCloud. Once sufficiently polished, the code will be submitted for inclusion into Sage.


\subsection{Outline}  In the following section, we introduce the relevant distribution spaces leading to the definition of the space of families of distributions, $\bDs$.  In the third section, we introduce methods of working in the space $\MSo{\bDs}$ including producing explicit elements in this space, forming a basis of $\MSo{\bDs}^{\ord}$, and computing characteristic power series of Hecke operators on this ordinary subspace.  In the fourth section, we explain how to carry out these computations in practice by giving a systematic method of approximating families of overconvergent modular symbols.  Lastly, in the fifth section, we close with several examples which we computed via these methods.

\subsection*{Acknowledgments}
We would like to thank the Southwest Center for Arithmetic Geometry for organizing the 2011 Arizona Winter School where the work on this article began as a student project. We would also like to thank the participants of Sage Days 44 for their work in porting the original Sage scripts into a full blown Sage package. We would like to thank Sage, as well as the SageMath Cloud, where we developed the algorithms and computed the examples in this article. We would like to thank Glenn Stevens for his support of this project and Frank Calegari for some very helpful conversations. Finally, our thanks go to the referee for some comments and suggestions that improved the clarity of this article.

\section{Distribution modules in families}

In this section, we introduce the relevant distribution spaces which will ultimately be the coefficients of our spaces of modular symbols.

\subsection{Distribution spaces}
\label{sec:dist}

\label{sec:dists}

Let $\bA$ denote the Tate algebra in a single variable $z$.  That is, $\bA= \Qp\langle z\rangle$, the collection  of power series with coefficients in $\Qp$ which converge on the unit disc of $\Cp$:
$$
\bA = \{ f(z) \in \Qp\llbracket z\rrbracket ~:~ f(z) = \sum_{n\geq0} a_n z^n \mbox{~and~} |a_n| \to 0 \text{~as~} n \to \infty \}.
$$
Note that $\bA$ is a Banach space under the norm
$$
|| f|| = \max_n |a_n|
$$
where $f(z) = \sum_n a_n z^n$.  We then define our space of distributions $\bD$ by
$$
\bD = \Hom_{\cont}(\bA,\Qp).
$$
Note that $\bD$ is a Banach space under the operator norm
$$
||\mu|| = \sup_{\genfrac{}{}{0pt}{}{f \in \bA}{f \neq 0}}
\frac{|\mu(f)|}{||f||}.
$$
An element $\mu \in \bD$ is uniquely determined by its values on all monomials $z^j$ since the latter have dense span in $\bD$.  We will refer to the sequence $\{\mu(z^j)\}_{j=0}^\infty$ as the {\it moments} of $\mu$.  We have that
$$
\mu \in \bD \text{~if~and~only~if~}  \{|\mu(z^j)|\} \text{~is~a~bounded~sequence~in~} \Qp.
$$
Indeed, for each $f(z) = \sum_j c_j z^j \in \bA$, we need that $\sum_j c_j \mu(z^j)$ converges.  But since $|c_j| \to 0$, this only forces $\mu(z^j)$ to be bounded (and any bounded sequence defines a distribution).

We will write $\bDo$ (resp.\ $\bAo$) for the unit ball of $\bD$ (resp.\ $\bA$).  Note that $\mu \in \bDo$ if and only if $\mu(z^j) \in \Zp$ for all $j \geq 0$.

The space $\bD$ is our basic distribution space which we will ultimately study in families.  But we will need to make use of some slightly fancier distribution spaces which we introduce now.

For $r \geq 1$, let $\bA[r]$ denote the collection of power series over $\Qp$ which converge on the disc in $\Cp$ of radius $r$ around 0, i.e.\
\[ \bA[r]=\left\{\sum_{n\geq0}a_nz^n\in\Qp\llbracket z\rrbracket:|a_n|r^n\rightarrow0\text{ as }n\rightarrow\infty\right\}.
\]
Then $\bA[r]$ is a Banach space under the sup norm, and we define $\bD[r] = \Hom_{\cont}(\bA[r],\Qp)$ as the dual Banach space.

Thus, $\bD[1]$ is nothing other than $\bD$ introduced above.  However, if $r>1$, then $\bD[r]$ is a larger space than $\bD$ and contains distributions whose moments are not bounded.  Indeed, for $f = \sum_j c_j z^j \in \bA[r]$ to converge on the disc of radius $r$, the sequence $\{c _j \}$ must converge rapidly to 0, thus allowing the sequence $\{ \mu(z^j) \}$ to have some non-trivial growth.  Explicitly,
$$
\mu \in \bD[r] \text{~if~and~only~if~}  \{|\mu(z^j)|\} \text{~is~} O(r^j) \text{~as~} j \to \infty.
$$

Finally, we set $\displaystyle \Aoc = \varinjlim_{r\rightarrow1^+}  \bA[r]$, {\it i.e.}\ the collection of power series which converge on {\it some} disc of radius strictly greater than 1.  This space is endowed with the inductive limit topology. We define $\Doc = \Hom_{\cont}(\Aoc,\Qp)$.   Equivalently, $\displaystyle \Doc = \varprojlim_{r\rightarrow1^+} \bD[r]$; or more simply, $\Doc$ is the intersection over all $r>1$ of $\bD[r]$.  Thus,
$$
\mu \in \Doc \text{~if~and~only~if~} \{\mu(z^j)\} \text{~is~} O(r^j) \text{~as~} j \to \infty \text{~for~all~} r>1.
$$

\subsection{Weight space}

For the remainder of the paper, let $p$ denote an odd prime.  Let $\W = \Hom(\Zpx,\Cpx)$ denote the collection of continuous
characters from $\Zpx$ to $\Cpx$. We will refer to this as weight
space.  There is an injective map from $\Z \to \W$ sending $k$ to the
``raising to the $k$-th power" character.

Since $\Zpx \cong (\Z/p\Z)^\times \times(1+p\Zp)$, a character in $\W$ is
uniquely determined by its restriction to $(\Z/p\Z)^\times$ and by its value on
a topological generator $\gamma$ of $1+p\Zp$.  Moreover, if $\kappa
\in \W$, then $|\kappa(\gamma) - 1| < 1$.

Let $D(0,1)$ be the open unit disc of $\Cp$ about 0.  The map
\begin{align*}
\Hom(1+p\Zp,\Cpx) &\to D(0,1) \\
\kappa &\mapsto \kappa(\gamma)-1
\end{align*}
is a bijection.  In particular, $\W$ can be identified with $p-1$
copies of the open unit disc.  Let $\omega : (\Z/p\Z)^\times \to \Zp^\times$ denote the Teichm\"{u}ller character and, for $0\leq m\leq p-2$, let $\cWr$ denote the subspace of $\W$ consisting of characters whose restriction to $(\Z/p\Z)^\times$ equals $\omega^m$.

\subsection{The weight $\kappa$ action}

Let $\sigop \subseteq \text{M}_2(\Zp)$ denote the semigroup of matrices $\smallmat$ of non-zero determinant with $a \in \Zpx$ and $c \in p\Zp$. For each $\kappa\in\W$, we wish to define a ``weight $\kappa$ action'' of $\sigop$ on the above spaces of power series and distributions. This will allow us to eventually define Hecke actions on spaces of overconvergent modular symbols. 

First, for $k$ an integer, we can define the weight $k$ action of $\sigop$ on the spaces defined in section \ref{sec:dists} as follows. For $f$ in $\bA[r]$ with $r<p$ and $\gamma \in \sigop$, we define
$$
(\gamma \cdot_k f)(z) = (a+cz)^k f \left( \frac{b+dz}{a+cz} \right)
$$
which endows $\bA[r]$ with a left $\sigop$-action.  Dually, for $\mu \in \bD[r]$, we define
$$
(\mu |_k \gamma)(f) = \mu(\gamma \cdot_k f)
$$
which endows $\bD[r]$ with a right $\sigop$-action.  Furthermore, this endows $\Doc$ (resp.\ $\Aoc$) with a right (resp.\ left) $\sigop$-action.  

Now we consider the case of $p$-adic weights. Let $\Wr$ denote the subspace of characters $\kappa$ in $\W_m$ that satisfy $|\kappa(\gamma) - 1| \leq 1/p$ for some (and hence every)
topological generator $\gamma$ of $1+p\Zp$. Note that the classical weights---the ``raising to the $k$-th power'' characters, for $k\in\Z$---are all in $\Wr$, for some $m$. We can (and will) identify
$\Wr$ with the closed disc of radius $1/p$ around 0.  For $\kappa \in \Wr$, we will define weight $\kappa$ actions on our spaces of distributions.   The key to doing this is to make sense of $\kappa(a+cz)$ as a power series in $z$ (see Definition \ref{def:Fkappa} and Lemma \ref{lemma:rigan} below).

We begin with some lemmas.

\begin{lemma}
\label{lemma:logiso}
If $\ord_p(x) > \frac{1}{p-1}$, then $|\log(1+x) | = |x|$.
\end{lemma}

\begin{proof}
The condition that $\ord_p(x) > \frac{1}{p-1}$ forces the first term to dominate in the power series expansion of $\log(1+x)$.
\end{proof}

\begin{lemma}
\label{lemma:factorial}
For $n \geq 1$, 
$$
\ord_p \left( p^n / n! \right) \geq n \cdot \left(1 - \frac{1}{p-1} \right).
$$
\end{lemma}

\begin{proof}
We have
$$
\ord_p(p^n/n!) \geq n - \left( n/p + n/p^2 + \dots \right) = n - n/(p-1).
$$
\end{proof}

Now for $\kappa \in W_0$, define
\begin{equation}
F_\kappa(z) := \sum_{n=0}^\infty \binom{\frac{\log(1+z)}{\log \gamma}}{n} (\kappa(\gamma)-1)^n
\end{equation}
which the following lemma shows is a power series expansion for the character $\kappa$.

\begin{lemma}
\label{lemma:Fk}
Fix $\kappa \in W_0$.

\begin{enumerate}
\item  $F_\kappa(x)$ converges for $x$ such that $\ord_p(x) > \frac{1}{p-1}$, i.e.\ $F_\kappa(z) \in \bA[r]$ for any $r < p^{-1/(p-1)}$.
\item For $x$ with $\ord_p(x) > \frac{1}{p-1}$, we have $|F_\kappa(x)| \leq 1$.
\item For $x \in 1+p\Zp$, 
$$
F_\kappa(x-1) = \kappa(x).
$$
\end{enumerate}
\end{lemma}

\begin{proof}
For the first part, since $\kappa \in W_0$, we have $|\kappa(\gamma)-1| \leq 1/p$.  Furthermore,  if $\ord_p(x) > \frac{1}{p-1}$, then by Lemma \ref{lemma:logiso}, we have $\ord_p( \log(1+x) ) = \ord_p(x)$.  If $$L := \frac{\log(1+x)}{\log \gamma},$$ then $\ord_p(L) = \ord_p(x) - 1$.

Let's further assume that $\ord_p(x) < 1$, so that $\ord_p(L) < 0$.
Then,
\begin{align*}
\ord_p  \binom{L}{n} 
&= \ord_p \left( \frac{L(L-1) \dots (L-n+1)}{n!}  \right) \\
&= \ord_p (L(L-1) \dots (L-n+1)) - \ord_p(n!)   \\
&= n \ord_p (L) - \ord_p(n!)   \\
&= n (\ord_p (x)-1) - \ord_p(n!)   \\
&\geq n (\ord_p (x)-1) - \frac{n}{p-1}  \\
&\geq n \left(\ord_p (x)-1-\frac{1}{p-1}\right),
\end{align*}
and thus
$$
\ord_p  \binom{L}{n} (\kappa(\gamma) - 1)^n \geq 
 n \left(\ord_p (x)-\frac{1}{p-1}\right).
$$
Since this term goes to infinity as $n \to \infty$, we have that $F_\kappa(x)$ converges.  Furthermore, since this is true for any $x$ with $1 > \ord_p(x) > \frac{1}{p-1}$, we must have that $F_\kappa(x)$ converges for all $x$ with $\ord_p(x) > \frac{1}{p-1}$.

For the second part, note that every term in the power series which defines $F_\kappa(x)$ has valuation at least 0 for $x$ with $1 > \ord_p(x) > \frac{1}{p-1}$.  Thus, $|F_\kappa(x)| \leq 1$ for all $x$ with $\ord_p(x)> \frac{1}{p-1}$ (by the Maximum Modulus Principle applied to any closed disc of radius between $1/p$ and $p^{-1/(p-1)}$, see \cite[Proposition 3 of \S5.1.4]{BGR}).

For the third part, write $x = \gamma^a$. Then, we have
$$
F_\kappa(x-1) 
= \sum_{n=0}^\infty \binom{\frac{\log(\gamma^a)}{\log \gamma}}{n} (\kappa(\gamma)-1)^n 
= \sum_{n=0}^\infty \binom{a}{n} (\kappa(\gamma)-1)^n = \kappa(\gamma)^a = \kappa(x).
$$
\end{proof}

\begin{definition}\label{def:Fkappa}
Fix $\kappa \in \Wr$ and write $\kappa = \omega^m \cdot \kappa_0$ with $\kappa_0 \in W_0$.  Let $a \in \Zpx$ and $c \in p \Zp$.  Define
$$
F_{\kappa,a,c}(z) := \omega(a)^m \cdot F_{\kappa_0}\left( \frac{a+cz}{\omega(a)} -1 \right).
$$
\end{definition}

\begin{lemma}
\label{lemma:rigan}
For $\kappa \in \Wr$, $a \in \Zpx$, and $c \in p\Zp$,  we have
\begin{enumerate}
\item  $F_{\kappa,a,c}(x)$ converges for $x$ such that $\ord_p(x) > \frac{1}{p-1} - 1$, i.e.\ $F_{\kappa,a,c}(z)$ is in $\bA[p^h]$ for any $h< c_p  := 1 - \frac{1}{p-1}=\frac{p-2}{p-1}$,
\item $|F_{\kappa,a,c}(x)| \leq 1$ for $x$ with $\ord_p(x) > \frac{1}{p-1} - 1$, 
\item for $x \in \Zp$, 
$$
F_{\kappa,a,c}(x) = \kappa(a+cx).
$$
\end{enumerate}
\end{lemma}

\begin{proof}
The first and second parts follow from the previous lemma since ${\ord_p \left( \frac{a+cx}{\omega(a)}-1 \right) > \frac{1}{p-1}}$.  
  The third part also follows from the previous lemma.  Indeed, as
  \[
  	\frac{a+cx}{\omega(a)} \equiv 1\mod{p},
  \]
  we have
$$
F_{\kappa_0}\left( \frac{a+cx}{\omega(a)} -1 \right) 
= \kappa_0\left( \frac{a+cx}{\omega(a)}  \right)
= \kappa_0\left( \frac{a+cx}{\omega(a+cx)}  \right)
= \kappa_0\left( a+cx  \right).
$$
Thus, 
$$
F_{\kappa,a,c}(x) = \omega(a)^m \cdot \kappa_0\left( a+cx  \right) = \kappa(a+cx).
$$
\end{proof}

We can now define the weight $\kappa$ action for $\kappa \in \Wr$ just as before.  Indeed, for $f$ in $\bA[r]$ (with $1\leq r < r_p := p^{c_p}$) and $\gamma \in \sigop$, we define
$$
(\gamma \cdot_\kappa f)(z) = F_{\kappa,a,c}(z)\cdot f\!\left( \frac{b+dz}{a+cz} \right)
$$
which by Lemma \ref{lemma:rigan} is again in $\bA[r]$.  Thus, we have endowed $\bA[r]$  with a left $\sigop$-action.  Further, for $\mu \in \bD[r]$, we define
$$
(\mu |_k \gamma)(f) = \mu(\gamma \cdot_\kappa f)
$$
which endows $\bD[r]$ with a right $\sigop$-action.  Again, this automatically endows $\Doc$ (resp.\ $\Aoc$) with a right (resp.\ left) $\sigop$-action.

\subsection{Power series in families over weight space}

Let $\bAr := \bA(\Wr)$ denote the space of convergent power series on the
closed disc $\Wr$, say in a variable $W$. We then have 
$$
\bAr = \left\{ \sum_n a_n W^n ~\bigg|~ a_n \in \Qp \text{~and~} |
p^n a_n | \to 0 \right\}.
$$
If we set $\ds w : = \frac{W}{p}$, then $\bAr$ is simply the Tate algebra $\Qp \tate{w}$ in the variable $w$.   
Set $\bAro$ equal to the unit ball of $\bAr$ under the sup norm which is simply the integral Tate algebra $\Zp \tate{w}$.

Consider the space $\bAs$.  We can
think of elements of this space as families of elements of $\bA$
over $\Wr$.  Indeed for each $\kappa \in \Wr$ with values in $\Qp$, we
have a map
$$
\kappa: \bAs \to \bA
$$ given by evaluating elements of $\bAr$ at $\kappa$.  Thus, for
a fixed element of $F$ of $\bAs$, we get a
family of elements $\kappa(F) \in \bA$ for each $\kappa \in \Wr$
having values in $\Qp$.

More explicitly, we have
$$
\bAs \cong \Qp\tate{z} \ot \Qp\tate{w} \cong \Qp\tate{z,w},
$$
and thus elements of $\bAs$ are formal power series in $z$ and $w$ which converge for all $|z| \leq 1$, $|w| \leq 1$.   Evaluating at $\kappa$ simply means evaluating $w$ at $(\kappa(\gamma)-1)/p$. Thus, as we $p$-adically vary $\kappa$ over $\Wr$, we get a $p$-adic family of elements of $\bA$.

We now seek to give $\bAs$ the structure
of a $\sigop$-module in such a way that the above map (``evaluation at
$\kappa$") is equivariant with respect to this action on the source
and the weight $\kappa$ action on the target. We do this by constructing a two-variable power series that interpolates $F_{\kappa,a,c}(z)$ as $\kappa$ varies.

For $a \in \Zpx$ and $c \in p\Zp$, define
\begin{equation}\label{eqn:Kacm_defn}
\aut(z,w) 
= \omega(a)^m \cdot 
\sum_{n=0}^\infty \binom{\log_\gamma\!\left(\frac{a+cz}{\omega(a)}\right)}{n} (pw)^n
= \omega(a)^m \cdot (1+pw)^{\log_\gamma\left(\frac{a+cz}{\omega(a)}\right)}
\end{equation}
where $\log_\gamma(z) := \log z /\log \gamma$.

\begin{lemma}
\label{lemma:Kconverge}
For $a \in \Zpx$ and $c \in p\Zp$, we have 
\begin{enumerate}
\item $\aut(z,w)$ converges for $z$ and $w$ such that $|z| < p^{c_p}$ and $|w| \leq 1$.  That is, $\aut(z,w) \in \bA[p^h] \ot \bAr$ for $h < c_p$,
\item  $|\aut(z,w)| \leq 1$ for all such $z,w$,
\item for $\kappa \in \Wr$, we have
$$
\kappa(\aut(z,w)) = F_{\kappa,a,c}(z).
$$
\end{enumerate}
\end{lemma}

\begin{proof}
The third part follows immediately from the definitions as
$$
\kappa(\aut(z,w)) = \aut(z,w)|_{w = (\kappa(\gamma) -1)/p}  = F_{\kappa,a,c}(z).
$$
But then the first and second parts follow from this equality and from Lemma \ref{lemma:rigan}.

\end{proof}

With this lemma in hand, we can thus define a $\sigop$-action on
$\bAsbig$ for $r < p^{c_p}$.  For $f \in \bA[r]$, set
\begin{equation}
\label{eqn:act}
\gamma \cdot (f(z) \otimes 1) = \aut(z,w) \cdot f \left(
\frac{b+dz}{a+cz} \right),
\end{equation}
and extend this action $\bAr$-linearly to all of $\bAsbig$. 

\begin{lemma}
For $\kappa \in \Wr$ and $r<p^{c_p}$, we have
$$
\kappa : \bAsbig \to \bA[r]
$$
is $\sigop$-equivariant where the source is endowed with the action in \eqref{eqn:act} and the target is endowed with the weight $\kappa$ action.
\end{lemma}

\begin{proof}
This lemma follows immediately from the definition of both actions.
\end{proof}

We mention here two basic properties of the automorphy factor $\aut(z,w)$, both of which follow directly from the definition and which will be useful later.

\begin{lemma}
\label{lemma:Kprops}
We have
\begin{enumerate}
\item $\autv{1}{0}{m}(z,w) = 1$,
\item \label{lemma:Kpropspart2}$\aut(z,w)\big|_{w = 0} = \omega(a)^m$.
\end{enumerate}
\end{lemma}

\subsection{Distributions in families over weight space}
\label{sec:distfam}

In this section, we discuss families of distributions and their $\sigop$-actions.
To consider families of distributions, a natural place to begin is the space
$$
\bD[r](\bAr) := \Hom_{\cont}(\bA[r],\bAr),
$$
that is, the space of $\bAr$-valued distributions.  Evaluating such distributions at varying $\kappa \in \Wr$ then gives rise to a family of single-variable distributions.  Moreover, these distributions are again quite concrete.  They are uniquely determined by their sequence of moments, and, in this case, each moment is a power series in $w$. 

However, the space $\bD[r](\bAr)$ turns out to be much larger than what we need to work with, and instead, we consider the space $\bDsbig$.  Note that there is a natural injection:
$$
\bDsbig = \Hom_{\cont}(\bA[r],\Qp) \ot \bAr \hookrightarrow \Hom_{\cont}(\bA[r],\bAr) = \bD[r](\bAr),
$$
but this map need not be surjective.  For example, the distribution $\mus \in \bD(\bAr)$ defined by
$$
\mus(z^j) = w^j  \text{~for~each~} j \geq 0
$$
is not in $\bDs$.  To see this, note that every distribution in $\bDs$ is a limit of finite sums of elements of the form $\mu \otimes f$ with $\mu \in \bD$ and $f \in \bAr$.  As such, for each $n$, there are only finitely many coefficients of $f$ which are not in $p^n \Zp$.  In particular, for any fixed element of $\bDs$, in {\it all} of its moments only finitely many coefficients are not in $p^n \Zp$.  Note that the distribution $\mus$ above clearly does not have this property.  




We again have a specialization map
$$
\bDsbig \to \bD[r]
$$
given by evaluation at $\kappa \in \Wr$.   We now seek to give an action of $\sigop$ on $\bDsbig$ which makes the above map equivariant when $\bD[r]$ is given the weight $\kappa$ action.

To do this, first note that $\bD[r]$ is an $\bA[r]$-module via
$$
(g \cdot \mu)(f) = \mu(gf)
$$
where $f,g \in \bA[r]$ and $\mu \in \bD[r]$.  Thus, $\bDsbig$ is naturally an $\bAsbig$-module.  Note also that $\bD[r](\bAr)$ is naturally an $\bAsbig$-module as
$$
\bD[r](\bAr) = \Hom_{\cont}(\bA[r],\bAr) \cong \Hom_{\cont,\bAr}(\bAsbig,\bAr).
$$
Furthermore, we can easily define a weight 0 action of $\sigop$ on $\bD[r](\bAr)$ via
\[ (\mus \big|_0\gamma)(f)=\mus(\gamma\cdot_0f)
\]
for $\mus \in \bD[r](\bAr)$.

\begin{lemma}
Both $\bDsbig$ and $\bD[r](\bAr)$ are $\sigop$-modules via the formula
$$
\mus | \gamma = (\aut(z,w) \cdot \mus) |_0 \gamma,
$$
for
$$
\gamma=\mat\in\sigop.
$$
\end{lemma}

\begin{proof}
This formula clearly gives an action on $\bD[r](\bAr)$.  To complete the proof, we must check that $\bDsbig \subseteq \bD[r](\bAr)$ is preserved by this action.  This detail is verified in \cite[page 30, Remark 3.1]{Bellaiche}; we note that in \cite{Bellaiche}, the notation $\bD[r](\bAr)$ refers to $\bD[r] \hat{\otimes} R$.
\end{proof}

We will also have the need to consider the larger distribution space
$$
\Docs := \Hom_{\cont}(\Aoc,\bAr)
$$
(when we solve the ``difference equation").  This space again is naturally a $\sigop$-module and we note that as before 
$$
\mus \in \Docs \text{~if~and~only~if~} \{|\mus(z^j)|\} \text{~is~} O(r^j) \text{~as~} j \to \infty \text{~for~all~} r>1.
$$

The following lemma will allow us to use the Hecke operator $U_p$ to pass from $\Docs$-valued modular symbols to $\bDs$-valued ones.

\begin{lemma}
\label{lemma:Upimprove}
If $\mus \in \Docs$, then $\mus \big|\!\smatrix{1}{a}{0}{p} \in \bDs$.
\end{lemma}

\begin{proof}
Since $\mus \in \Docs$, we have that the sequence $\{| \mus(z^j)| \}$ is $O(r^j)$ for every $r>1$. Furthermore, we have that
$$
(\mus \big|\!\smatrix{1}{a}{0}{p}) (z^j) = 
\mus  (\autv{1}{0}{m}(z,w) (a+pz)^j) = 
\sum_{n=0}^j \binom{j}{n} a^{j-n} p^n \mus(z^n),
$$
since $\autv{1}{0}{m}(z,w) = 1$ by Lemma \ref{lemma:Kprops}.
Because $p^n \mus(z^n) \to 0$, it is clear that the moments of $\mus \big|\!\smatrix{1}{a}{0}{p}$ are bounded and thus this distribution is in $\bD(\bAr)$.  Furthermore, 
 for any $M$, for $n$ large enough $p^n \mus(z^n) \in p^M \bAro$. Thus, modulo $p^M \bAro$, the moments of the distribution $\mus \big|\!\smatrix{1}{a}{0}{p}$ only depend on finitely many moments of $\mus$.  In particular, $\mus \big|\!\smatrix{1}{a}{0}{p}$ can be written as a limit of elements of $\bD \otimes \bAr$, and hence  $\mus \big|\!\smatrix{1}{a}{0}{p} \in \bDs$.
 \end{proof}

\subsection{Analyzing the automorphy factor}

By Lemma \ref{lemma:Kconverge},  $\aut(z,w)$ is in $\bA \ot \bAr$.  In this section, we will further analyze the coefficients of this automorphy factor in order to gain better control of the $\sigop$-action on families of distributions.
We begin by introducing some rings that will be useful for this purpose.

Consider an abstract Tate algebra, $\Qp\tate{x}$ and define
$$
S_x := \left\{  \sum_n a_n x^n \in \Qp\tate{x} ~:~ \ord_p(a_n) \geq n c_p
\right\} 
$$
where $c_p = 1 - \frac{1}{p-1} = \frac{p-2}{p-1}$, as before.  Note that by definition $S_x \subseteq \Zp\tate{x}$.

\begin{lemma}
\label{lemma:subring}
$S_x$ is a subring of $\Zp\tate{x}$.
\end{lemma}

\begin{proof}
We only need to check that $S_x$ is closed under multiplication.  To this end, let $f = \sum a_n x^n$ and $g = \sum b_n x^n$.  Then the $n$-th coefficient of $fg$ equals $\sum_{i+j=n} a_i b_j$, and we have 
$$
\ord_p( a_i b_j) = \ord_p(a_i) + \ord_p(b_j) \geq i \cdot c_p + j \cdot c_p = n \cdot c_p.
$$
Thus $fg \in S_x$ as desired.
\end{proof}

\begin{lemma}
\label{lemma:Sinvert}
If $f \in S_x$ with $f(0) \in \Zpx$, then $f^{-1} \in S_x$.
\end{lemma}

\begin{proof}
Let $f(x) = \sum_i a_i x^i$ and $g(x) = \sum_j b_j x^j$ with $f \cdot g = 1$.  
We check inductively that $\ord_p(b_n) \geq n \cdot c_p$.
  For $n=0$, this is immediate as $b_0 = a_0^{-1}$.  For $n>0$, we have
$$
b_n = \frac{1}{f(0)} \sum_{i=1}^n a_i b_{n-i}.
$$
By induction, for $i>0$ we have $\ord_p (b_{n-i}) \geq (n-i) \cdot c_p$,
and thus
$$
\ord_p(b_n) \geq \min_i\{ \ord_p(  a_i b_{n-i} ) \} \geq i \cdot c_p + (n-i) c_p = n \cdot c_p
$$
as desired.
\end{proof}

\begin{lemma}
\label{lemma:bounded}
If $r \in \Q$ and $f(x) = \sum_n a_n x^n \in \Qp\llbracket x\rrbracket$ are such that
\begin{enumerate}
\item $f(x)$ converges for all $x$ in the open disc of radius $p^r$ centered around 0, and
\item $|f(x)| \leq 1$ for all such $x$,
\end{enumerate}
then $\ord_p(a_n) \geq nr$.
\end{lemma}

\begin{proof}
Write $g(x) = f(x/p^r)$ which is then a power series which converges on the open unit disc of $\Cp$. This power series is bounded in size by 1 and thus is in $\O_{\Cp}\llbracket x\rrbracket$ (since the Gauss norm equals the sup norm).  Thus, $a_n/p^{rn} \in \O_{\Cp}$ as desired.
\end{proof}

\begin{remark}
From Lemma \ref{lemma:bounded}, we have that  $S_x$ is simply the collection of $\Qp$-power series which converge on the disc of radius $p^{c_p}$ and all of whose values have size less than or equal to 1.  This gives another way to see that $S_x$ is a ring.
\end{remark}

\begin{thm}
\label{thm:autsSz}
For $\aut(z,w)$ as defined in equation~\eqref{eqn:Kacm_defn}, we have
$$
\aut(z,w) \text{ is in } S_z \otz \bAro.
$$
\end{thm}

\begin{proof}

Write
$$
\aut(z,w) = \sum_{j=0}^\infty R_j(w) z^j= \sum_{j=0}^\infty T_j(z) w^j.
$$
To prove, $\aut(z,w) \in S_z \otz \bAr^0$, we must show that $T_j(z) \in S_z$ for each $j \geq 0$.  Thus, we must show that the coefficient of $z^i$ in $T_j(z)$ has $p$-adic valuation at least $ic_p$ for all $i,j \geq 0$.  But this is equivalent to showing that $p^{i c_p}$ divides $R_i(w)$ in $\bAro$ for all $i \geq 0$.

Next fix some $w_0$ with $|w_0| \leq 1$.  Then $\aut(z,w_0)$ is a power series which converges on the open disc of radius $p^{c_p}$ and all of its values have size bounded by 1 on this disc (by Lemma \ref{lemma:Kconverge}).   Thus, by Lemma \ref{lemma:bounded}, $R_i(w_0)$ has valuation at least $i c_p$.  But since this is true for every $w_0$ in the closed unit disc, we have that every coefficient of $R_i(w)$ has valuation at least $ic_p$ (since the Gauss norm is the same as the sup norm).
\end{proof}

\begin{thm}
\label{thm:autsSw}
We have
$$
\aut(z,w) \text{~is~in~} \bAo \otz S_w.
$$
\end{thm}

\begin{proof}
Mimicking the proof of Theorem \ref{thm:autsSz}, it suffices to show that $\aut(z,w)$ converges for $|z| \leq 1$ and $|w| < p^{c_p}$ to something of size less than or equal to 1.  To this end, recall that 
$$
\aut(z,w) 
= \omega(a)^m \cdot 
\sum_{n=0}^\infty p^n \binom{L}{n} w^n
$$
where $L = \log_\gamma(\frac{a+cz}{\omega(a)})$.  Since $p | c$, we have $L \in \Zp\llbracket z\rrbracket$.
Thus, for $|z|<1$, we have
$$
\ord_p \left(  p^n \binom{L}{n} w^n \right) = n - \ord_p(n!) + n \ord_p(w)
 \geq n \cdot \left( c_p + \ord_p(w) \right).
$$
If $\ord_p(w) > - c_p$ this expression is always positive and  goes to infinity as $n \to \infty$ as desired.
\end{proof}


The following lemma will be useful later.

\begin{lemma}
\label{lemma:cS}
Let
$$
g(w) := \left.\frac{\partial}{\partial z} \aut(z,w)\right|_{z=0}.
$$
Then $\displaystyle \frac{g(w)}{w}$ is in $c \cdot S_w^\times$.
\end{lemma}

\begin{proof}
We first show that $g(w)$ is in $c \cdot S_w$ (and then automatically $g(w)/w$ is in $c \cdot S_w$). 
The coefficient of $w^n$ in $\aut(z,w)$ is
$$
\frac{p^n}{n!} L(L-1) \dots (L-n+1)
$$
where $L = \log_\gamma(\frac{a+cz}{\omega(a)} )$.  We must show that the coefficient of $z$ in this expression has valuation at least $\ord_p(c) + n c_p$.  But this is easy as 
$\ord_p( p^n / n!) \geq c_p$ and  the coefficient of $z$ in $L$ is always divisible by $c$.  

To finish the proof, it suffices by Lemma \ref{lemma:Sinvert} to check that the coefficient of $w$ in $g(w)$ is in $c \Zpx$.  This coefficient is the same as the coefficient of $w z$ in $\aut(z,w)$ which is
$$
p \cdot \omega(a)^m \frac{\partial}{\partial z} \log_\gamma\left(\frac{a+cz}{\omega(a)}   \right)\big|_{z=0}  =
 \frac{p   }{\log \gamma} \cdot  \frac{\omega(a)^{m}}{a} \cdot c.
$$
which is indeed in $c \Zpx$.
\end{proof}

\section{Families of overconvergent modular symbols}
\label{sec:OMSs}

\subsection{Modular symbols}

We review here the theory of modular symbols as formulated in \cite{AshStevens-Duke,GS,PS1}.  To this end, let $\Delta_0 := \Div^0({\mathbb P}^1(\Q))$ denote the set of degree zero divisors on ${\mathbb P}^1(\Q)$ which we endow with a left action of $\GL_2(\Q)$ via linear fractional transformations.  Let $\Gamma$ denote a congruence subgroup of $\SL_2(\Z)$ and let $V$ denote a right $\Gamma$-module.  We define $\MS{V}$, the space of $V$-valued modular symbols of level $\Gamma$, to be the collection of additive homomorphisms $\varphi: \Delta_0 \to V$ such that $\varphi(\gamma D) = \varphi(D) \big| \gamma^{-1}$ for all $\gamma \in \Gamma$ and $D \in \Delta_0$.

The modules $V$ we will consider in this paper include $\Sym^k(\Qp^2)$, $\bD_k$ and $\bDs$.  The first space has an action of $\Gamma$ while the second two have an action of  $\Gamma_0 := \Gamma \cap \Gamma_0(p)$.  Moreover, in each of these cases, one can extend the action of $\Gamma_0(p)$ to the algebra
\[
	S_0(p)=\left\{\mat\in M_2(\Z):(a,p)=1,p\mid c,\text{ and }ad-bc\neq0\right\}
\]
and thus define a Hecke-action on the corresponding spaces of $V$-valued modular symbols.  

The space $\MS{\Sym^k(\Qp^2)}$ is the space of classical modular symbols; the systems of Hecke-eigenvalues occurring in this space match those occurring in $M_{k+2}(\Gamma)$ (see \cite[Proposition 2.5]{BellaicheDasgupta}).  The space $\MSo{\bD_k}$ is the space of overconvergent modular symbols; the systems of finite slope Hecke-eigenvalues occurring in this space essentially match those occurring in $M^\dag_{k+2}(\Gamma)$, the space of overconvergent modular forms (see \cite[Theorem 7.1]{PS2}).  

Lastly, the space $\MSo{\bDs}$ is the space of {\it families} of overconvergent modular symbols. Indeed, for each $\kappa \in \Wr$, the map $\kappa:\bDs \to \bD_\kappa$ induces a Hecke-equivariant map
\begin{align*}
\spec_\kappa: \MSo{\bDs} \to \MSo{\bD_\kappa}. 
\end{align*}
Thus, for $\Phis \in \MSo{\bDs}$ and $\kappa$ in $\Wr$, we have that  $\spec_\kappa(\Phis)$ is a weight $\kappa$ overconvergent modular symbol, and, moreover,  as $\kappa$ varies, $\spec_\kappa(\Phis)$ varies in a $p$-adic family.

\subsection{Constructing families of overconvergent modular symbols}
\label{sec:random}

In this section, we describe a method of producing ``random" families of overconvergent modular symbols.  Here we follow the methods described in \cite[Section 2]{PS1}  to explicitly write down modular symbols. 

\begin{prop}
\label{prop:msexist}
Assume $\Gamma_0$ is torsion-free.  Then there exist divisors $D_1, \dots, D_t$ in $\Delta_0$ and matrices $\gamma_1, \dots, \gamma_t$ in $\SL_2(\Z)$ such that for any right $\Gamma$-module $V$ and any $\phi \in \MSo{V}$, we have
$$ 
 \phi(\{0\}- \{\infty\}) \big| \Delta = \sum_{j=1}^t \phi(D_j) \big| (\gamma_j-1)
$$
where $\Delta := \smatrix{1}{1}{0}{1} - 1$.
Conversely, for any $v_1, \dots, v_t$ in $V$ satisfying
\begin{equation}
\label{eqn:diff}
v_\infty\big| \Delta = \sum_{j=1}^t v_j\big| (\gamma_j-1),
\end{equation}
there is a unique modular symbol $\phi \in \MSo{V}$ such
that 
$$
\phi(D_j) = v_j
$$
for each $j$.
\end{prop}

\begin{proof}
See \cite[Corollary 2.7]{PS1}.
\end{proof}

\begin{remark}
\label{rmk:fund}
In \cite{PS1} explicit algorithms are given to determine the $D_i$ and the $\gamma_i$.  This is the so-called process of ``solving the Manin relations".  In the end, the $\gamma_i$ together with the identity matrix form a subset of a full set of right coset representatives for $\Gamma_0$ in $\SL_2(\Z)$,  and the $D_i$ are $\Z[\Gamma]$-generators of $\Delta_0$.
\end{remark}

Proposition \ref{prop:msexist} gives us a strategy for explicitly writing down families of overconvergent modular symbols.  Just randomly pick elements $v_1, \dots, v_t$ in $\bDs$, and then try to solve equation \eqref{eqn:diff}.   

We will refer to equations of the form $w \big| \Delta = v$ as {\it difference equations}.  These equations were studied in detail in \cite[Section 4.2]{PS1} for the the module $\Doc$.  The following lemma generalizes the situation to $\Docs$.

\begin{lemma}
\label{lemma:diff}
Let $\Delta := \smatrix{1}{1}{0}{1} - 1$ denote the difference operator.  We have
\begin{enumerate}
\item the map $\Delta : \Docs \to \Docs$ is injective;
\item if $\mus \in \im(\Delta)$, then $\mus({\bf 1}) = 0$;
\item for $\mus \in \Docs$ with $\mus({\bf 1}) = 0$, there exists a unique $\nus \in \Docs$ such that $\nus \big| \Delta = \mus$.
\end{enumerate}
\end{lemma}

\begin{proof}
The first part follows verbatim as in \cite[Lemma 4.3]{PS1}.  The second part is clear as
$$
(\mu \big| \smatrix{1}{1}{0}{1})({\bf 1}) - \mu({\bf 1}) = \mu({\bf 1}) - \mu({\bf 1}) = 0
$$
since $\autv{1}{0}{m}(z,w) = 1$.  For the last part, just proceed as in \cite[Theorem 4.5]{PS1}.  Note that the newly constructed measure $\mu$ still takes values in $\bAr$.
\end{proof}

\begin{remark}\mbox{}
\begin{enumerate}
\item The explicit formulas for the solution of the difference equation given in \cite[Lemma 4.3]{PS1} apply equally well in the case of families.
\item
We note in the above lemma that if $\nus$ were in the smaller space $\bDs$, there is no reason for $\mus$ to again be in the $\bDs$ as denominators naturally appear in the solution of the difference equation.  These denominators are the primary reason for considering the space $\Docs$ in this paper.
\end{enumerate}
\end{remark}

Thus, to solve equation \eqref{eqn:diff} the only condition which we need to verify is that the right hand side has total measure zero.  However, for randomly chosen $v_i \in \Docs$, there is no reason for $\displaystyle \sum_{j=1}^t v_j\big| (\gamma_j-1)$ to have total measure zero.  Indeed, we compute
\begin{align*}
\left(\sum_{j=1}^t v_j\big| (\gamma_j-1)\right)({\bf 1})
&= 
\sum_{j=1}^t \left(v_j\big| \gamma_j \right)({\bf 1}) -  v_j({\bf 1}) \\
&= 
\sum_{j=1}^t \left(\left(\autv{a_j}{c_j}{m}(z,w) \cdot v_j\right) \big|_0 \gamma_j \right)({\bf 1}) -  v_j({\bf 1})\\
&= 
\sum_{j=1}^t  v_j \left(\autv{a_j}{c_j}{m}\right) -  v_j({\bf 1}),
\end{align*}
and see that the result is just some power series in $\bAr$.  We do note that by Lemma \ref{lemma:Kprops}, we have $\autv{a}{c}{0}(z,0) = 1$, and thus this power series specializes to 0 in weight 0 if $m \equiv 0 \mod{p-1}$, that is, this power series is divisible by $w$ if $m \equiv 0 \mod{p-1}$.  

In our quest to write down a family of overconvergent modular symbols, we have chosen the $v_j$ arbitrarily, and thus we still have a great deal of flexibility.  The following lemma explains precisely how to choose one of the $v_j$ more carefully to force the total measure of the right hand side of equation \eqref{eqn:diff} to vanish.  In the following lemma, $\mu_j$ denotes the distribution whose $j$-th moment is 1 and all of whose other moments vanish.

\begin{lemma}
\label{lemma:adjust}
Let $v'_1, \dots, v'_t$ be any elements of $\Docs$ and set 
$$
g = -\left(\sum_{j=1}^t v'_j\big| (\gamma_j-1)\right)({\bf 1}).
$$
If $m \equiv 0 \mod{p-1}$, fix any $i$ between 1 and $t$, and set
$$
v_i = \begin{cases}
v'_i & \text{if~} i \neq m \\
v'_i + \displaystyle \frac{g}{\frac{\partial}{\partial z} \autv{a_i}{c_i}{m}(z,w) \big|_{z=0}} \cdot \mu_1 & \text{if~} i = m 
\end{cases}.
$$
If $m \not\equiv 0 \mod{p-1}$, fix some $i$ between 1 and $t$ such that $a_i^m \not \equiv 1 \mod{p}$, and set
$$
v_i = \begin{cases}
v'_i & \text{if~} i \neq m \\
v'_i + \displaystyle \frac{g}{\autv{a_i}{c_i}{m}(0,w)-1} \cdot \mu_0 & \text{if~} i = m 
\end{cases}.
$$
Then, in either case, $v_j \in \Docs$ for all $j$, and  $\sum_{j=1}^t v_j\big| (\gamma_j-1)$ has total measure zero. 
\end{lemma}

\begin{proof}
We begin with the case $m \equiv 0 \mod{p-1}$.  We first justify that $v_i$ is in $\Docs$.  
That is, if  $h = \frac{\partial}{\partial z} \autv{a_i}{c_i}{m}(z,w) \big|_{z=0}$, we need to check that $g/h \in \bAr$.      Note that $h$ is simply the coefficient of $z$ in $\autv{a_i}{c_i}{m}(z,w)$ and is thus in $S_w$ by Theorem \ref{thm:autsSw}.  However, $h$ is not invertible in $S_w$.  Indeed, by Lemma \ref{lemma:Kprops} part (2), we have $h(0) = 0$.  Fortunately, $g$ also vanishes at $w = 0$ by the discussion immediately preceding this  lemma (this is where we are using the fact that $m \equiv 0 \mod{p-1}$).  Further, by Lemma \ref{lemma:cS}, $h/w$ is in $c_i \cdot S^\times$.  By Remark \ref{rmk:fund}, we have $c_i \neq 0$, and thus $h/w$ is invertible in $\bAr$.  Therefore, $g/h = (g/w) / (h/w)$ is in $\bAr$.

For the second part, we compute
\begin{align*}
\left( \sum_{j=1}^t v_j\big| (\gamma_j-1)\right)({\bf 1})
&=
\left( \sum_{j=1}^t v'_j\big| (\gamma_j-1)\right)({\bf 1})
+
\frac{g}{h} \cdot 
\left( \mu_1 \big| (\gamma_m-1) \right)({\bf 1}) \\
&=
-g + 
\frac{g}{h} \cdot \left( \left( \left( \autv{a_i}{c_i}{m}(z,w) \cdot \mu_1 \right) \big|_0 \gamma_i \right)({\bf 1}) - \mu_1({\bf 1})\right)
\\
&=
-g + 
\frac{g}{h} \cdot   \mu_1(\autv{a_i}{c_i}{m}(z,w))  
\\
&= 0
\end{align*}
as $\mu_1(\autv{a_i}{c_i}{m}(z,w))  = h$.

Now onto the case of $m \not \equiv 0 \mod{p-1}$.  We again justify that $v_i$ is in $\Docs$.  
That is, if $h = \autv{a_i}{c_i}{m}(0,w)$, we need to check that $g/(h-1) \in \bAr$.  By Theorem \ref{thm:autsSw}, $h-1$ is in $S_w$.    Further, by Lemma \ref{lemma:Kprops} part \eqref{lemma:Kpropspart2}, the constant term of $h-1$ is $\omega(a_i)^m-1$ which is a unit by assumption.  Thus, by Lemma \ref{lemma:Sinvert}, $h-1$ is invertible and thus $g/(h-1) \in \bAr$ as desired.

For the second part, we again compute
\begin{align*}
\left( \sum_{j=1}^t v_j\big| (\gamma_j-1)\right)({\bf 1})
&=
\left( \sum_{j=1}^t v'_j\big| (\gamma_j-1)\right)({\bf 1})
+
\frac{g}{h-1} \cdot 
\left( \mu_0 \big| (\gamma_m-1) \right)({\bf 1}) \\
&=
-g + 
\frac{g}{h-1} \cdot \left( \left( \left( \autv{a_i}{c_i}{m}(z,w) \cdot \mu_0 \right) \big|_0 \gamma_i \right)({\bf 1}) - \mu_0({\bf 1})\right)
\\
&=
-g + 
\frac{g}{h-1} \cdot   (\mu_0(\autv{a_i}{c_i}{m}(z,w))  - 1)
\\
&= 0
\end{align*}
as $\mu_0(\autv{a_i}{c_i}{m}(z,w)) = h$.
\end{proof}

\begin{remark}
In the case $m \not \equiv 0 \mod{p-1}$, if it happens that $a_i^m \equiv 1 \mod{p}$ for every $i$, then a simple computation shows that 
$\sum_{j=1}^t v_j \big| (\gamma_j - 1)$ has 0-th moment which vanishes in weight 0.  We could thus proceed as in the case of $m \equiv 0 \mod{p-1}$. We leave the details to the reader, but we note that in our computations we have never encountered this case.  Possibly this case never occurs or only occurs in very small level.
\end{remark}

\begin{cor}
Keeping the notation of Lemma \ref{lemma:adjust}, there exists $\Phis \in \MSo{\Docs}$ such that 
$$
\Phis(D_j) = v_j
$$
for each $j$.
\end{cor}

\begin{proof}
By Lemma \ref{lemma:adjust} and Lemma \ref{lemma:diff}, there exists $v_\infty \in \Docs$ such that 
$$
v_\infty \big| \Delta = \sum_{j=1}^t v_j\big| (\gamma_j-1).
$$
Thus, by Proposition \ref{prop:msexist}, there exists $\Phis \in \MSo{\Docs}$ such that 
$\Phis(\{0\}-\{\infty\}) = v_\infty$, and 
$$
\Phis(D_j) = v_j
$$
for each $j$ as desired.
\end{proof}

\begin{remark}
The assumption that $\Gamma_0$ be torsion-free is not at all essential.  In \cite[Section 2.5]{PS1}, there is a discussion on how to deal with torsion elements in constructing modular symbols.  Further, the arguments of Lemma \ref{lemma:adjust} carry through to this case with just minor changes.
\end{remark}

\subsection{Ordinary families of overconvergent modular symbols}
\label{sec:ordfam}

As Hida families are the primary object of interest in this paper, we now describe how to pass to the {\it ordinary subspace} of our spaces of modular symbols.  To this end, recall that if $X$ is a compact $\Zp$-module equipped with a compact operator $U_p$, we define the ordinary subspace $X^{\ord} := \bigcap_n U_p^n X$.  Then $X^{\ord}$ is the largest subspace of $X$ on which $U_p$ acts invertibly.  If moreover $X$ is profinite, then there is a canonical decomposition $X = X^{\ord} \oplus X^{\nil}$ where $X^{\nil}$ is the subspace of $X$ on which $U_p$ acts topologically nilpotently (see \cite[Proposition 2.3]{GS}).  Moreover, projection onto $X^{\ord}$ is given by the operator $e:=\ds\lim_{n\rightarrow\infty} U_p^{n!}$.

Unfortunately, $\MSo{\bDso}$ is not a profinite space since the Tate algebra $\bAro$ is not profinite, and thus it is not {\it a priori} clear that $\MSo{\bDso}$ admits its ordinary subspace as a direct summand.   However, $\bAro = \Zp\tate{w}$ is contained in $\bArc :=\Zp[[w]]$ which is a profinite ring.  Moreover, viewing $\bArc$ as the ring of bounded functions on the open disc of radius $1/p$ contained in $\bAro$, we see that this ring is preserved by the action of $\sigop$, and thus we get a Hecke-equivariant inclusion
$$
\MSo{\bDso} \subseteq \MSo{\bDsc}.
$$
Further, we obtain a direct sum decomposition into ordinary and non-ordinary parts: 
$$
\MSo{\bDsc} \cong \MSo{\bDsc}^{\ord} \oplus \MSo{\bDsc}^{\nil}.
$$
In what follows, the space $\MSo{\bDsc}^{\ord}$ will be our primary object of interest.  To ease notation we will denote this space simply by $\MSord$.

We note that in the ordinary case no Hecke information should be lost by working on  this open disc of radius $1/p$.  Indeed, as Hida families extend to all of weight space, one expects that the Hecke-eigenvalues of ordinary families of modular symbols should do the same. This fact is stated in the following theorem and proven in Appendix~\ref{app:GS}.

\begin{thm}
\label{thm:extend}
We have
\begin{enumerate}
\item $\MSord$ is a free $\bArc$-module and
$$
\rank_\Lambda(\MSord) = 
\rank_{\Zp}(\MSo{\Sym^k(\Zp^2)}^{\ord}) 
$$
for any $k \equiv m \mod{p-1}$,
\item for $T$ any Hecke operator, 
$$
\chr(T ~|~ \MSord)
$$ 
has coefficients in $\Zp[[W]] = \Zp[[pw]]$; that is, the coefficients of this characteristic polynomial extend to the open unit disc.
\end{enumerate}
\end{thm}

\subsection{Vector of total measures}
\label{sec:vtm}

In this section, we make the following simple but extremely helpful observation:\ an element of   $\MSord$ is completely determined by the total measures of all of its values.  Moreover, since a modular symbol is determined by its values on finitely many divisors, one only needs finitely many of these total measures to determine the symbol.  We can thus express any element of $\MSord$ as a vector with coordinates in $\bArc$, and thus reduce many computations with ordinary families of overconvergent modular symbols to computations in a free module over $\bArc$.

More precisely, choose $D_1, \dots, D_t \in \Delta_0$ which generate $\Delta_0$ as a $\Zp[\Gamma_0]$-module (see Remark \ref{rmk:fund}).  We then define the {\it vector of total measures} map
$$
\alpha : \MSord \lra \bArc^t 
$$
defined by sending $\Phi$ to the vector $\left( \Phi(D_i)({\bf 1}) \right)_{i=1}^t$.  

We note that this construction works equally well for a fixed weight, thus expressing an overconvergent modular symbol as an element of $\Zp^t$; that is, setting $\MSkord := \MSo{\bDo_k}^{\ord}$, we then have a map
$$
\alpha_k : \MSkord \lra \Zp^t 
$$
defined exactly as above.

\begin{prop}
\label{prop:vtm}
We have
\begin{enumerate}
\item \label{item:1}
the map $\alpha$ is injective;

\item \label{item:2} the map $\alpha_k$ is injective;

\item \label{item:3} the induced map 
$$
\overline{\alpha} : \MSord  \otimes \Lambda/ \m {\lra} (\Lambda/\m)^t \cong \Fp^t
$$
is injective.  Here $\m$ is the maximal ideal of $\Lambda=\Zp[[w]]$.
\end{enumerate}
\end{prop}

\begin{proof}
We first note that part \eqref{item:3} implies part \eqref{item:1}.  Indeed, if $K$ is the kernel of $\alpha$, then part \eqref{item:3} implies that $K \otimes \Lambda / \m = 0$ and thus $K=0$.  Similarly, part \eqref{item:3} implies part \eqref{item:2}.  Indeed, by Lemma \ref{lemma:compOMS}, 
$$
\MSord \otimes (\bArc / \p_k ) \cong \MSkord,
$$
and reducing the map $\alpha$ modulo $\p_k$ yields the map $\alpha_k$.  Moreover, $\m = \p_k + p\Lambda$.  Thus, if $K$ is now the kernel of $\alpha_k$, by part \eqref{item:3}, $K \otimes \Lambda/\m = 0$, and thus $K=0$.

So it suffices to prove part \eqref{item:3}.  In fact, from the observations above, it suffices to see that 
$$
\MSkord \otimes \Fp \lra \Fp^t
$$
is injective.  That is, it suffices to see that if $\Phi_k \in \MSkord$ with $\Phi_k(D_i)({\bf 1})$ divisible by $p$ for each $i$, then $\Phi_k$ is divisible by $p$.  
Seeking a contradiction, assume that $||\Phi_k|| = 1$.  Since the $D_i$ generate $\Delta_0$, we see that $\Phi_k(D)({\bf 1})$ is divisible by $p$ for every $D \in \Delta_0$.    But then
\begin{align*}
(\Phi_k \big| U_p)(D)(z^j) 
=
\sum_{a=0}^{p-1} (\Phi_k (\psmallmat{1}{a}{0}{p} D) \big| \psmallmat{1}{a}{0}{p})(z^j)
=
\sum_{a=0}^{p-1} \Phi_k (\psmallmat{1}{a}{0}{p} D)((a+pz)^j)
\end{align*}
which is divisible by $p$ since $\Phi_k (\psmallmat{1}{a}{0}{p} D)({\bf 1})$ is divisible by $p$ by assumption.  Thus, $|| \Phi_k | U_p || < 1$.  But since $\Phi_k$ is in the ordinary subspace, we have $\Phi_k = \lim_n \Phi_k | U_p^{n!}$ which implies $|| \Phi_k | U_p || = || \Phi_k || =1 $.  This  contradiction establishes  part \eqref{item:3} and completes the proof.
\end{proof}

Here's one example of the usefulness of these vector of total measure maps.

\begin{cor}
\label{cor:extend}
We have that $\left\{ \Phi_1^{\ord}, \dots, \Phi_j^{\ord} \right\}$ can be completed to a $\bArc$-basis of $\MSord$ if and only if $\left\{ \ol{\alpha}(\Phi_1^{\ord}), \dots, \ol{\alpha}(\Phi_j^{\ord}) \right\}$ is a linearly independent set in $\Fp^t$.
\end{cor}

\begin{proof}
By Proposition \ref{prop:vtm}, $\ol{\alpha}(\Phi_1^{\ord}), \dots, \ol{\alpha}(\Phi_j^{\ord})$ are linearly independent in $\Fp^t$ if and only if the images of $\Phi_1^{\ord}, \dots, \Phi_j^{\ord}$ in $\MSord/\m \MSord$ are linearly independent.  By a compact version of Nakayama's lemma, this is true if and only if 
$\Phi_1^{\ord}, \dots, \Phi_j^{\ord}$ is the start of a $\Lambda$-basis of $X$.
\end{proof}

\subsection{Bases of $\MSord$ and characteristic polynomials of Hecke operators}
\label{sec:charpolyhecke}

We present here a method of computing the characteristic polynomials of Hecke operators acting on the ordinary subspace of $\MSord$.  

We begin by describing a naive idea of how one can form a $\bArc$-basis of $\MSord$.  We first note that we  can assume that we know the $\bArc$-rank of $\MSord$ as Theorem \ref{thm:extend} expresses this rank in terms of the $\Zp$-rank of some classical space of modular symbols which by standard methods is readily computed.\footnote{\label{footnote:hida}The $\Zp$-rank of $\MSo{\Sym^k(\Zp^2)}^{\ord}$ is simply given by the number of non-zero roots of the characteristic polynomial of $U_p$ acting on $\MSo{\Sym^k(\Fp^2)}$.}
Let us assume then that we have in hand elements $\Phi_1^{\ord}, \dots, \Phi_j^{\ord}$ in $\MSord$ which are the start of a $\bArc$-basis of $\MSord$.  We now describe how to extend this set to a full $\bArc$-basis by working one element at a time.

To this end, produce some ``random" element $\Phi$ of $\MSo{\bDsc}$---for instance, using the methods described in section \ref{sec:random}.  Then, as described in section \ref{sec:ordfam}, by iterating $U_p$ we can form $\Phi^{\ord}$, the projection of $\Phi$ onto $\MSord$.  If $\Phi_1^{\ord}, \dots, \Phi_j^{\ord}$ together with $\Phi^{\ord}$ still form the  beginning of some $\bArc$-basis of $\MSord$ (which we can test via Corollary \ref{cor:extend}), we have succeeded in extending our partial basis.  Otherwise, we produce another ``random" symbol $\Phi$ and continue repeating this process.  As long as our method of producing such symbols is sufficiently random, we will eventually find a symbol $\Phi^{\ord}$ which extends our partial basis.

Now, with a $\bArc$-basis of $\MSord$ in hand, we next want to compute the characteristic polynomials of Hecke operators acting on $\MSord$.  To do this, we simply need to write down the associated matrix of any Hecke operators with respect to this basis.  To this end, for $T$ a Hecke operator, we must be able to write
$$
\Phi^{\ord}_i \big| T= \sum_j a_{ij}(w) \Phi^{\ord}_j
$$
with $a_{ij}(w) \in \bArc$.  To find the power series $a_{ij}(w)$ which solve these equations, one can use the vector of total measures described in section \ref{sec:vtm}.  Indeed, it suffices to solve
$$
\alpha(\Phi^{\ord}_i \big| T)= \sum_j a_{ij}(w) \alpha(\Phi^{\ord}_j)
$$
which is now a system of linear equations over $\Lambda$.

We note that even though the matrix associated to $T$ will be defined over $\Lambda$, by Theorem \ref{thm:extend}, the characteristic polynomials of these matrices will lie in $\Zp[[pw]]$, and thus extend to all of weight space.

Lastly we mention that the above method works equally well for the plus and minus subspace $(\MSord)^\pm$ by simply passing to the $\pm$-parts of the random symbols produced.

\subsection{Restricting to collections of congruent forms}
\label{sec:decomp}

The methods of the previous section describe how to form a basis of $\MSord$ and how to compute the Hecke action on this basis.  In the special case when the dimension of $(\MSord)^\pm$ is 1, our single basis element is then an eigensymbol, and thus immediately contains the information of families of Hecke eigenvalues.  However, it is extremely rare for $(\MSord)^\pm$ to be one-dimensional; this only happens for small primes and small tame level.  To partially circumvent this problem, we now  describe a decomposition of $\MSord$ into Hecke stable subspaces, comprising of congruent families, and it is not at all uncommon for pieces of this decomposition to be 1-dimensional.

Let $\T$ denote the Hecke algebra over $\Lambda$ acting on $\MSord$.  
The ring $\T$ is a semi-local ring with $\T \simeq \oplus_\m \T_{\m}$ where $\m$ varies over the maximal ideals of $\T$.  
This isomorphism induces a Hecke-equivariant isomorphism $\MSord \simeq \oplus_{\m} \MSord_{\m}$.  We now describe how to compute the characteristic polynomials of Hecke operators acting on $\MSord_{\m}$ for each individual maximal ideal $\m$.

Fix a prime $\ell$ and let $T$ denote either $T_\ell$ or $U_\ell$ depending on whether or not $\ell$ divides $Np$.  Set $\overline{f}_{\m,\ell}$ equal to the characteristic polynomial of $T$ acting on $\MSord/\m \MSord$ which is a polynomial defined over $\Fp$.\footnote{We note that this polynomial also arises as the characteristic polynomial of $T$ acting on the space of $p$-ordinary modular symbols of weight $k$ defined over $\Fp$ for any $k \equiv m \pmod{p-1}$ and is thus readily computed.}  For a fixed $\m$, one can find a prime $\ell$ so that any lift ${f}_{\m,\ell}(T)$ of $\overline{f}_{\m,\ell}(T)$ to characteristic 0 acts topologically nilpotently on $\MSord_{\m}$ and invertibly on $\MSord_{\m'}$ for all $\m' \neq \m$.  

Now, to form a basis of $\MSord_\m$, we can simply follow the method of \S\ref{sec:charpolyhecke} as long as we can produce sufficiently random symbols in $\MSord_\m$.  To do this, we form a random symbol $\Phi \in \MSord$ and then iterate the Hecke operator $\prod_{\m' \neq \m} f_{\m',\ell}(T)$ which results in projecting $\Phi$ to the subspace $\MSord_\m$, as desired.

Again, we mention that this method also works to produce a basis of $(\MSord_{\m})^\pm$.

\section{Explicit computations with families of OMSs}

In section \ref{sec:OMSs}, we described methods of computing with ordinary families of overconvergent modular symbols.  However, this discussion was all carried out on a theoretical level as a single $\Phi \in \MSo{\bDsc}$ is determined by an infinite amount of information.  In order to compute with these families in practice, one must have a systematic method of approximating each $\Phi$ with a finite amount of data.  Moreover, such approximations must be respected by the Hecke operators.  In what follows, we describe our method of approximating families of overconvergent modular symbols.  Further, we verify that the methods we described in the previous section still carry through with our approximated families.

\subsection{Finite approximation modules in families}

We begin by reviewing the methods of  \cite{PS1} where a systematic method of approximating elements of $\bD_k$ was given which was compatible with the $\sigop$-action.  These approximations allowed for explicit computations to be carried out in the space $\MSo{\bD_k}$.

In forming an approximation of a distribution $\mu$ in $\bDo_k$, we note that the naive method of considering the first $M$ moments of $\mu$ each modulo $p^M$  is  not stable under the matrix action on $\bD_k$.  Instead, in \cite{PS1}, a $\sigop$-stable filtration on $\bDo_k$ was introduced:
$$
\Fil^M(\bDo_k) = \{ \mu \in \bDo ~:~  \ord_p(\mu(z^j)) \geq M-j \text{~for~} 0 \leq j \leq M\},
$$
and thus one can approximate $\mu \in \bDo_k$ by looking at its image in the finite set $\F_k(M) := \bDo_k / \Fil^M(\bDo_k)$.  Explicitly, one is then  approximating a distribution $\mu \in \bDo_k$ by considering its $j$-th moment modulo $p^{M-j}$ for $0 \leq j \leq M$.
For this reason, we refer to the $\sigop$-stable space 
$\F_k(M)$ as a {\it finite approximation module}.  The space $\MSo{\F_k(M)}$ is thus a natural space to work in to approximate overconvergent modular symbols.

We seek to generalize this construction to the case of families; that is, we seek a $\sigop$-stable filtration on $\bDsc$.  One could hope  to define a nice filtration on $\bDsc$ by simply extending the above filtration on $\bD^0$ by $\bArc$-linearly.  
However, this filtration is not preserved by the $\sigop$-action defined in section \ref{sec:distfam}.  Indeed, the $\sigop$-action on $\bDsc$ is defined by combining the $\bA$-action on $\bD$ with a weight 0 action.  However, $\Fil^M(\bDo)$ is not preserved under the $\bA$-action.  For instance, multiplication by the element $z$ maps  $\Fil^M(\bDo)$ into $\Fil^{M-1}(\bDo)$.  

We do note however that the subring $\Zp\llbracket pz\rrbracket \subseteq \bA$ does preserve $\Fil^M(\bDo)$---this is immediate from the above definition of $\Fil^M(\bDo)$ as multiplication by $z$ simply shifts the moments of a distribution down by one.  Moreover, the $\sigop$-action on $\bDsc$ does not act through arbitrary elements of $\bA$; rather, we are only acting by power series of the form
$$
\aut(z,w) = \omega(a)^m \cdot 
\sum_{n=0}^\infty p^n \binom{\log_\gamma(\frac{a+cz}{\omega(a)})}{n} w^n,
$$
and thus the only power series in $z$ we need to act by are of the form 
\begin{equation}
\label{eqn:actingelements}
p^n \binom{\log_\gamma(1+pa+pbz)}{n}.
\end{equation}
But even for $n=1$, these power series need not be in $\Zp\llbracket pz\rrbracket$.  For instance,
$$
p \log_\gamma(1+pz) = \frac{p}{\log p} \left( pz - \frac{p^2 z^2}{2} + \dots + (-1)^p p^{p-1} z^p + \dots \right);
$$
note the troubling term is $p^{p-1} z^p$.  

We now turn to Theorem \ref{thm:autsSz} to see how far these power series are from being in $\Zp[[pw]]$.  Indeed, this theorem tells us that such power series are in $S_z$, and thus their $j$-th coefficients have valuation at least $c_p j$.

We are thus led to modify our filtration at any fixed weight as follows.  Set
$$
\wt{\Fil}^M(\bDo) = \{ \mu \in \bDo ~:~  \ord_p(\mu(z^j)) \geq  M- j \cdot c_p \text{~for~all~} j \geq 0\}.
$$
In $\Fil^M(\bDo)$, the sequence of lower bounds on the valuations of the moments was
$$
M, M-1, \dots, 2, 1.
$$
In $\wt{\Fil}^M(\bDo)$, the corresponding sequence begins as
$$
M, M, M-1, \dots, M-(p-3).
$$
This pattern of  $p-1$ terms, with the first two terms stable and the rest decreasing by 1,  then continues to repeat.

\begin{lemma}
\label{lemma:modFil}
The weight $k$ action of $\sigop$ on $\bD$ preserves $\wt{\Fil}^M(\bD^0)$.
\end{lemma}

\begin{proof}
We argue as in \cite[Prop 7.1]{PS1} with some small changes.  Let $\mu \in \wt{\Fil}^M \Dkrig^0$.   For $j \geq 0$, we must show that $\ord_p (\mu \big| \gamma)(z^j) \geq M - j \cdot c_p$.  We compute
\begin{align*}
(\mu \big| \gamma)(z^j) &= \mu \left( (a+cz)^{k-j} (b+dz)^j\right) \\
&= \mu \left( \left( \sum_{m=0}^\infty \binom{k-j}{m} a^{k-j-m} c^m z^m
\right)
\left( \sum_{m=0}^j \binom{j}{m} b^{j-m} d^m z^m \right) \right) \\
&= \mu \left( \sum_{s=0}^\infty c_s z^s \right) 
=  \sum_{s=0}^\infty c_s \mu(z^s),
\end{align*}
for some $c_s \in \Zp$.  Since $\mu \in \wt{\Fil}^M(\Dkrig^0)$, we have that $\ord_p \mu(z^s) \geq M - s \cdot c_p$.  For $s \leq j$, we then have $\ord_p \mu(z^s) \geq M - j \cdot c_p$.  For $s \geq j$, an easy computation with the explicit formula above yields that $c_s$ is divisible by $p^{s-j}$.  Thus, 
$$
\ord_p (c_s \mu(z^s)) \geq M - s \cdot c_p + s - j = M + \frac{s}{p-1} - j \geq M + \frac{j}{p-1} - j = M - j \cdot c_p
$$
as desired.
\end{proof}

\begin{lemma}
\label{lemma:Sz}
The action of $S_z \subseteq \bA$ on $\bD$ preserves $\wt{\Fil}^M(\bD^0)$.
\end{lemma}

\begin{proof}
It suffices to check that for $\mu \in \bD^0$ and for monomials of the form $a_n z^n$ such that $\ord_p(a_n) \geq  n \cdot c_p$, we have $a_n z^n \cdot \mu \in \bD^0$.  To this end, we compute
\begin{align*}
\ord_p ( (a_n z^n \cdot \mu)(z^j) ) 
&= \ord_p(a_n) + \ord_p( \mu(z^{j+n}) ) \\
&\geq  n \cdot c_p + M-(j+n) \cdot c_p \\
&= M-j \cdot c_p 
\end{align*}
as desired.
\end{proof}

We now simply define a filtration on $\bDsc$ by:
$$
\wt{\Fil}^M( \bDsc ) := \wt{\Fil}^M(\bDo) \otz \bArc.
$$
That is, $\mu \in \wt{\Fil}^M( \bDsc )$ if its $j$-th moment is an element of $\bArc$ whose $p$-adic valuation is at least $M - j \cdot c_p$, i.e.\ if when written as a power series in $w$, all of its coefficients have valuation at least $M - j \cdot c_p$.

\begin{lemma}
\label{lemma:filpreserve}
We have that $\wt{\Fil}^M( \bDsc )$ is preserved by the $\sigop$-action.
\end{lemma}

\begin{proof}
The $\sigop$-action is defined as:
$$
\mu | \gamma = (\aut(z,w) \cdot \mu) |_0 \gamma.
$$
The lemma thus follows from Theorem \ref{thm:autsSz},  Lemma \ref{lemma:modFil}, and Lemma \ref{lemma:Sz}.
\end{proof}

We set $\wt{\F}(M) := (\bDsc )/  \wt{\Fil}^M( \bDsc )$.  Unfortunately, note that $\wt{\F}(M)$ is not finite.  Indeed, this module still keeps track of the coefficients of infinitely many powers of $w$.  To fix this, fix $L>0$, and we define
$$
\wt{\Fil}^{M,L}( \bDsc) := \wt{\Fil}^M(\bD^0) \otz \bArc  + w^L \bDsc.
$$
That is $\mu \in \wt{\Fil}^{M,L}( \bDsc )$ if the first $L$ coefficients of 
its $j$-th moment (thought of as an element of $\bArc$) have valuation at least $M - j \cdot c_p$.  Set $$\wt{\F}(M,L) := (\bDsc) / \wt{\Fil}^{M,L}( \bDsc ),$$ and  
then
\begin{align*}
\wt{\F}(M,L)
&\cong
\bD^0 / \wt{\Fil}^M(\bD^0) \otimes (\bArc / w^L \bArc)   \\
&\cong \bD^0 / \wt{\Fil}^M(\bD^0) \otimes  \left( \Zp[w] / w^L \Zp[w] \right)
\end{align*}
which is finite.  We will refer to $\wt{\Fil}^M(\bDsc)$ as the $M$-th approximation module and to $\wt{\Fil}^{M,L}(\bDsc)$ as the $(M,L)$-th finite approximation module.

\begin{remark}
All of the previous discussion goes through equally well if we replace $\bDsc$ with $\bDso$ allowing us to define $\wt{\Fil}^M(\bDso)$ and $\wt{\Fil}^{M,L}(\bDso)$.
We further note that since $\bArc / w^L \bArc \cong \bAro / w^L \bAro$, we have
$$
\bDso / \wt{\Fil}^{M,L}(\bDso) \cong \bDsc / \wt{\Fil}^{M,L}(\bDsc).
$$
When working with these finite approximation modules, one cannot distinguish $R^0$ (the Tate algebra) from $\Lambda$ (the Iwasawa algebra).
\end{remark}

\subsection{Handling denominators}
\label{sec:denoms}

The one downside to the above formulation of finite approximation modules is that it only allows us to approximate families of distributions whose moments are integral power series.  However, in solving the difference equation, the resulting distributions don't have integral moments (they aren't even bounded!).  To fix this problem, we proceed as in \cite[page 29]{PS1}.  Set
$$
\wt{K}_{0}(\bAr) = \left\{ \mu \in \Doc(\bAr) \text{~such~that~} \ord_p(\mu(z^j)) \geq -j \cdot c_p     \right\}.
$$

\begin{lemma}
We have
\begin{enumerate}
\item $\wt{K}_{0}(\bAr)$ is a $\sigop$-module;
\item $p^{M} \wt{\K}_{0}(\bAr) \cap \bDso= \wt{\Fil}^M(\bDso).$
\end{enumerate}
\end{lemma}

\begin{proof}
Part one follows exactly as in Lemma \ref{lemma:filpreserve}.
Part two follows immediately from the definitions.
\end{proof}

We thus have the following alternative description of our approximation modules in families:
$$
\frac{\bDso + p^M \wt{\K}_0(\bAr)}{p^M \wt{\K}_0(\bAr)}  \cong \frac{\bDso}{p^M \wt{\K}_0(\bAr) \cap \bDso} \cong \frac{\bDso}{\wt{\Fil}^M(\bDso)} 
$$
Note that these maps are $\sigop$-isomorphisms.  Thus, as long as we are working with distributions in  $\bDso + p^M \wt{\K}_0(\bAr)$, it makes sense to project to the $M$-th approximation module.

\subsection{Solving the difference equation in $\wt{\F}(M)$}

We now use the description of $\wt{\F}(M)$ given in section \ref{sec:denoms} to explain how one solves the difference equation in these approximation modules.  We first review the case of a fixed weight, and then discuss the case of families.


\subsubsection{The case of a fixed weight}

The following is a slight improvement on \cite[Lemma 7.5]{PS1}.  We refer to \textit{loc.\ cit.} for undefined notation. In what follows, set
$$
\K_{0} = \left\{ \mu \in \Doc \text{~such~that~} \ord_p(\mu(z^j)) \geq -j     \right\}.
$$

\begin{lemma}
\label{lemma:denoms}
Let $\mu \in \Doc$ and $\nu \in \bDo$ with $\mu | \Delta = \nu$.  Then for any $M \geq 0$, we have
$$
p^m \mu \in \bDo + p^M \K_0
$$
where $m = \left\lfloor \frac{\log(M+1)}{\log p} \right\rfloor$.
\end{lemma}

\begin{proof}

By the explicit solution of the difference equation given in \cite[Theorem 4.5]{PS1}, it suffices to see for all $j \geq 1$ that
$$
p^m \cdot \frac{\eta_{j-1}}{j} \in \bDo + p^M \K_0.
$$
We must thus check that for $r \leq M$, we have
\begin{equation}
\label{eqn:1}
p^m \cdot \frac{\eta_{j-1}}{j} (z^r) \in \Zp
\end{equation}
and for $r > M$, we have
\begin{equation}
\label{eqn:2}
p^{r-M+m} \cdot \frac{\eta_{j-1}}{j} (z^r) \in \Zp
\end{equation}

We start with the case $r \leq M$.  To see \eqref{eqn:1}, it suffices to see that
$$
\frac{p^m}{j} \binom{r}{j-1} b_{r-j+1} \in \Zp
$$
for $r \geq j-1$.  We thus have that $M \geq j-1$ and so $\frac{p^m}{j} \in \Zp$.  If $\frac{p^m}{j} \in p\Zp$, then we are done by the Clausen--von Staudt theorem as each Bernoulli number is in $\frac{1}{p} \Zp$.  Thus, we just need to consider the case where $j = a p^m$ with $1 \leq a \leq p-1$ and deduce that $\binom{r}{j-1} b_{r-j+1} \in \Zp$.

If $r=j-1$ we are done as $b_0 =1$.  Then, for $r>j-1$, we have $b_{r-j+1}$ has a $p$ in its denominator if and only if $p-1$ divides $r-j+1$.  In this case, we have $$ r= j-1 = a-1 \mod{p-1}$$
and we must deduce that $\binom{r}{ap^m-1}$ is divisible by $p$.  	By Lucas' theorem, it suffices to see that one of the base $p$ digits of $ap^m-1$ is greater than one of the base $p$ digits of $r$.  The base $p$ expansion of $ap^m -1$ is $(a-1~p-1~ p-1~ \dots ~p-1)_p$.  Since $r \leq M < p^{m+1}$, the only possibility that $r$ has every base $p$ digit larger than those of $ap^m-1$ is if the base $p$ representation of $r$ is 
$(c~p-1~~p-1~~\dots~~p-1)_p$ with $p-1 \geq c > a-1$.  In this case, $r = (c+1)p^m-1$.  But then $r \equiv c \mod{p-1}$ which is impossible as $r \equiv a-1 \mod{p-1}$.  

Now in the second case where $r > M$, set $s = r-M$.  To see \eqref{eqn:2}, if suffices to see that
$$
\frac{p^{m+s}}{j} \binom{r}{j-1} b_{r-j+1} \in \Zp
$$
for $r \geq j-1$.  Note that 
$$
m = \left\lfloor \frac{\log(M+1)}{\log p} \right\rfloor \implies p^{m+1} \geq M+2 \implies
p^{m+s} \geq M+2 +s = r+2 > j
$$
and thus $\frac{p^{m+s}}{j}$ is divisible by $p$.  Again, by the Clausen--von Staudt theorem, we are done.
\end{proof}




\begin{cor}
\label{cor:filtdiff}
Let $\ol{\nu} \in \F_k(M)$ have total measure 0.  Then there exists $\ol{\mu} \in \F_k(M)$ such that
$$
\ol{\mu} | \Delta = p^m \ol{\nu}
$$
where $m = \left \lfloor \frac{\log(M+1)}{\log p} \right\rfloor$.
\end{cor}

\begin{proof}
Lift $\ol{\nu}$ to some element ${\nu}$ in $\bDo$ with total measure 0.  Solving the difference \cite[Theorem 4.5]{PS1}, then yields ${\mu} \in \Doc$ with ${\mu} \big| \Delta = {\nu}$.  Then by Lemma \ref{lemma:denoms}, we have $\mu \in \bDo + p^M \K_0$.  Projecting $\mu$ to $(\bDo + p^M \K_0)/ p^M \K_0 \cong \bDo / \Fil^M(\bDo)$ then yields a solution to the difference equation in the finite approximation module.
\end{proof}

We note that the above corollary tells us the existence of solution to the difference equation in $\F(M)$.  We now describe how to explicitly write down such a solution.  Moreover, by analyzing this explicit solution, we will see that a smaller power of $p$ is needed to control denominators.

To start, we note that the solution to the difference equation in $\F(M)$ is not unique. 

\begin{lemma}
\label{lemma:kerdiffeqn}
If $\mu \in \Fil^{M-1}(\bDo)$, then $\mu \big| \Delta \in \Fil^M(\bDo)$.
\end{lemma}

\begin{proof}
Take $\mu \in \Fil^{M-1}(\bDo)$, and we compute
$$
(\mu \big| \Delta)(z^j) = \mu((z+1)^j) - \mu(z^j)
= \sum_{i=0}^{j-1} \binom{j}{i} \mu(z^i).
$$
Since $\mu(z^i) \in p^{M-1-i}\Zp$ for $i$ between $0$ and $j-1$, we see that $(\mu \big| \Delta)(z^j) \in p^{M-1-(j-1)} \Zp = p^{M-j}\Zp$.  Thus $\mu \big| \Delta \in \Fil^{M}(\bDo)$.
\end{proof}

\begin{prop}
\label{prop:solvediff}
Take  $\ol{\nu} \in \F_k(M)$ with total measure zero, and set $m = \left\lfloor \frac{\log(M)}{\log p} \right\rfloor$.  Define $\ol{\mu}_0 \in \F_k(M-1)$ by
$$
\ol{\mu}_{0}(z^r) = \sum_{j=1}^{r+1} \frac{p^m \ol{\nu}(z^j)}{j} \binom{r}{j-1} b_{r+1-j}
$$
for $0 \leq r \leq M-2$.
If $\ol{\mu} \in \F_k(M)$ is any element which projects to $\ol{\mu}_{0}$ in $\F_k(M-1)$, then $\ol{\mu} \big| \Delta = p^m \ol{\nu}$ in $\F_k(M)$.
\end{prop}

\begin{proof}
We first note that the formula defining $\ol{\mu}_{0}(z^r)$ makes sense.  To see this, note that
\[\frac{p^m}{j} \binom{r}{j-1} b_{r+1-j} \in \Zp
\]
by the proof of Lemma \ref{lemma:denoms}.  Further, $\ol{\nu}(z^j)$ is well-defined modulo $p^{M-j}$, and thus $\ol{\mu}_0(z^r)$ is well-defined modulo $p^{M-1-r}$.  Hence, $\ol{\mu}_0$ is a well-defined element of $\F(M-1)$.  

Next, let $\nu$ denote any lift of $\ol{\nu}$ to $\bDo$ with total measure 0, and let $\mu \in \Doc$ be the unique distribution satisfying $\mu \big| \Delta = p^m \nu$ (by \cite[Theorem 4.5]{PS1}).  Then the image of $\mu$ in $\F(M-1)$ equals $\ol{\mu}_0$ since the explicit formulas in \cite[Theorem 4.5]{PS1} exactly match the formulas defining $\ol{\mu}_0$ in this proposition.  (Note that our choice of $m$ allows us to form this projection.)  Thus, for any $\ol{\mu} \in \F(M)$ lifting $\ol{\mu}_0$, we have that the image of $\mu$ in $\F(M)$ equals $\ol{\mu}$ up to some distribution taking values in $\Fil^{M-1}(\bDo)$.  Our proposition then follows from Lemma \ref{lemma:kerdiffeqn}.
\end{proof}

\subsubsection{The difference equation in $\wt{\F}(M)$}

We now generalize the discussion of the previous section to the case of families.

\begin{lemma}
\label{lem:diffdenoms}
Let $\mu \in \Docs$ be such that $\mu \big| \Delta = \nu \in \bDsc$.  
Then for any $M \geq 0$, we have
$$
p^m \mu \in \bDso + p^{M} \wt{\K}_{0}(\bAr)
$$
where $m = \left\lfloor \frac{\log(M/c_p+1)}{\log p} \right\rfloor$.
\end{lemma}

\begin{proof}
The proof proceeds nearly the same as in Lemma \ref{lemma:denoms}.  Indeed, we still have an explicit solution to the difference equation, and so it suffices to see that
$$
p^m \cdot \frac{\eta_{j-1}}{j} \in \bDso + p^M \wt{K}_0(R).
$$
By definition, this means we need to check that for $r \leq M/c_p$, we have
\begin{equation}
p^m \cdot \frac{\eta_{j-1}}{j} (z^r) \in \Zp,
\end{equation}
and for $r > M/c_p$, we have
\begin{equation}
p^{rc_p-M+m} \cdot \frac{\eta_{j-1}}{j} (z^r) \in \Zp.
\end{equation}
Analyzing these two cases then follows exactly as in Lemma \ref{lemma:denoms}.  
\end{proof}

Thus, by scaling our distributions by a small power of $p$ we will be able to solve the difference equation in these finite approximation modules.  

\begin{cor}
\label{cor:filtdifffam}
Let $\ol{\nu} \in \wt{\F}(M)$ have total measure 0.  Then there exists $\ol{\mu} \in \wt{\F}(M)$ such that
$$
\ol{\mu} | \Delta = p^m \ol{\nu}
$$
where $m = \left \lfloor \frac{\log(M/c_p+1)}{\log p} \right\rfloor$.
\end{cor}

\begin{proof}
The proof follows verbatim as in Corollary \ref{cor:filtdiff}.
\end{proof}

Just as in the fixed weight case, the solution to the difference equation is not unique in $\wt{\F}(M)$.  Unfortunately, the analogue of Lemma \ref{lemma:kerdiffeqn} (which describes part of the kernel of $\Delta$) does not quite hold in families.  We instead just state a slightly weaker version of Proposition \ref{prop:solvediff} in families.

\begin{prop}
\label{prop:explicitfinitediff}
Take  $\ol{\nu} \in \wt{\F}(M)$ with total measure zero, and set $m = \left\lfloor \frac{\log(M/c_p)}{\log p} \right\rfloor$.  Define $\ol{\mu} \in \wt{\F}(M-1)$ by
$$
\ol{\mu}(z^r) = \sum_{j=1}^{r+1} \frac{p^m \ol{\nu}(z^j)}{j} \binom{r}{j-1} b_{r+1-j}
$$
for $0 \leq r < M/c_p-1$. Then $\ol{\mu} \big| \Delta = p^m \ol{\nu}$ in $\wt{\F}(M-1)$.
\end{prop}

\begin{proof}
To see that the above formulas yield a well-defined element of $\wt{\F}(M-1)$, note that
\[\frac{p^m}{j} \binom{r}{j-1} b_{r+1-j} \in \Zp
\]
by the proof of Lemma \ref{lemma:denoms}.  Further, $\ol{\nu}(z^j)$ is well-defined modulo $p^{\lceil M-jc_p\rceil}$.  From the above formula, we then see that $\ol{\mu}(z^r)$ is well-defined modulo $p^{\lceil M-(r+1)c_p\rceil} = p^{\lceil M-c_p -r c_p\rceil}$.  Hence, $\ol{\mu}(z^r)$
 is well-defined modulo $p^{\lceil M-1-r c_p\rceil}$, and $\ol{\mu}$ is a well-defined element of $\F(M-1)$.  

Next, let $\nu$ denote any lift of $\ol{\nu}$ to $\bDso$ with total measure 0, and let $\mu \in \Docs$ be the unique distribution satisfying $\mu \big| \Delta = p^m \nu$ (by Lemma \ref{lemma:diff}).  Then the image of $\mu$ in $\wt{\F}(M-1)$ equals $\ol{\mu}$ since the explicit formulas defining $\mu$ exactly match the formulas defining $\ol{\mu}$ in this proposition.  By Lemma \ref{lem:diffdenoms}, $\mu \in \bDso + p^{M-1} \wt{\K}_0(R)$, and thus projecting to $\wt{\F}(M-1)$ gives the desired result. 
\end{proof}

\subsection{The ordinary subspace of $\MSo{\wt{\F}(M)}$}

Since $\wt{\F}(M)$ is defined by taking the reduction modulo various powers of $p$, the space $\MSo{\wt{\F}(M)}$ has the potential to have a complicated structure even as a $\Zp$-module.    However, if we restrict to the ordinary subspace, the following proposition proves that passing to the $M$-th approximation module is equivalent to reducing modulo $p^M$.

\begin{prop}
\label{prop:ordfinite}
The natural map $\bDsc \to \wt{\F}(M)$ induces an isomorphism
$$
\MSord \otimes \Z/p^M \Z \stackrel{\sim}{\lra} \MSo{\wt{\F}(M)}^{\ord},
$$
and thus
$$
\MSord \otimes \Lambda/(p^M,w^L)\Lambda\stackrel{\sim}{\lra} \MSo{\wt{\F}(M,L)}^{\ord}.
$$
In particular, $\MSo{\wt{\F}(M,L)}^{\ord}$ is a free $\Zp[w]/(p^M,w^L)$-module, and for $T$ a Hecke operator
$$
\chr(T ~|~\MSord) \equiv 
\chr(T ~|~\MSo{\wt{\F}(M,L)}^{\ord}) \mod{p^M,w^L}.
$$
\end{prop}

\begin{proof}
We first show that 
$$
\MSord  \to \MSo{\wt{\F}(M)}^{\ord}
$$
is surjective.  To this end, consider the exact sequence
$$
0 \lra \wt{\Fil}^M(\bDo) \ot \bArc \lra \bDo \ot \bArc \lra \wt{\F}(M) \lra 0.
$$
Identifying $\MSo{V}$ with $H^1_c(\Gamma_0,V)$ and invoking the long exact sequence for cohomology, it suffices to show that $H^2_c(\Gamma_0,\wt{\Fil}^M(\bDo) \ot \bArc)^{\ord} = 0$.  But
$$
H^2_c(\Gamma_0,\wt{\Fil}^M(\bDo) \ot \bArc) \cong H_0(\Gamma_0,\wt{\Fil}^M(\bDo) \ot \bArc) \cong \left(\wt{\Fil}^M(\bDo) \ot \bArc\right)_{\Gamma_0}
$$
and
$$
\left(\left(\wt{\Fil}^M(\bDo) \ot \bArc\right)_{\Gamma_0}\right)^{\ord}
\cong
\left(\left(\wt{\Fil}^M(\bDo) \ot \bArc\right)^{\ord}\right)_{\Gamma_0}
\subseteq
\left(\left(\bDo \ot \bArc\right)^{\ord}\right)_{\Gamma_0}
$$
which vanishes as in the proof of Lemma \ref{lemma:compOMS}.

To check injectivity of the map of this proposition,  we must take 
$\Phi  \in \MSord $ which takes values in $\wt{\Fil}(M)$, and show that $\Phi$ is divisible by $p^M$.  To this end, take the largest possible $r$ such that $p^r$ divides $\Phi$, and assume that $r<M$.  Then $\Psi := p^{-r} \Phi$ has size 1 and takes values in
$$
p^{-r} \left( \wt{\Fil}(M) \cap p^r \bDsc \right) \subseteq \wt{\Fil}(M-r).
$$
But since $M-r>0$, this means that the total measure of each value of $\Psi$ is divisible by $p$.  Arguing as in Proposition \ref{prop:vtm}, we then have that $||\Psi | U_p|| <1$.  This is a contradiction since $\Psi$ is in the ordinary subspace.

The remainder of the proposition all follows formally from the first claim and the fact that $\MSo{\bDsc}$ is a free $\bArc$-module.
\end{proof}

\subsection{Vector of total measures}

As before, fix $D_1, \dots, D_t \in \Delta_0$ which generate $\Delta_0$ as a $\Zp[\Gamma_0]$-module.  We again define a {\it vector of total measures} map, but now for $\wt{\F}(M,L)$-valued symbols.
Set
$$
\alpha_{M,L} : \MSo{\wt{\F}(M,L)}^{\ord} \lra \left( \bArc/(p^M,w^L)\bArc \right)^t 
$$
defined by sending $\Phi$ to the vector $\left( \Phi(D_i)({\bf 1}) \right)_{i=1}^t$.   

\begin{prop}
\label{prop:vtmfinite}
The map $\alpha_{M,L}$ is injective.
\end{prop}

\begin{proof}
By Proposition \ref{prop:ordfinite}, $\MSo{\wt{\F}(M,L)}^{\ord}$ is simply the reduction of $\MSord$ modulo $(p^M,w^L)\Lambda$.  Likewise, $\alpha_{M,L}$ is simply the reduction of $\alpha$ (from section \ref{sec:vtm}) modulo $(p^M,w^L) \Lambda$.  Thus, if $K$ is the kernel of $\alpha_{M,L}$, by Proposition \ref{prop:vtm}, we have $K \otimes \Lambda / \m = 0$.  Thus, $K = 0$ and $\alpha_{M,L}$ is injective.
\end{proof}

\begin{cor}
\label{cor:extendfinite}
We have $\left\{ \Phi_1^{\ord}, \dots, \Phi_j^{\ord} \right\}$ is the start of $\bArc / (p^M,w^L) \bArc$-basis of $\MSo{\wt{\F}(M,L)}^{\ord}$ if and only if $\left\{ \alpha_{1,1}(\Phi_1^{\ord}), \dots, \alpha_{1,1}(\Phi_j^{\ord}) \right\}$ is a linearly independent set in $(\Lambda / (p,w) \Lambda)^t \cong \Fp^t$.
\end{cor}

\begin{proof}
The same argument in Corollary \ref{cor:extend} applies (invoking Proposition \ref{prop:vtmfinite} instead of Proposition \ref{prop:vtm}).
\end{proof}




\subsection{Characteristic polynomials of Hecke operators}

In section \ref{sec:charpolyhecke}, we sketched a method of computing the characteristic polynomials of Hecke operators acting on $\MSord$.  In this section, we explain how to carry this method out in practice in the finite spaces $\MSo{\wt{\F}(M,L)}^{\ord}$.  Recall that by Proposition \ref{prop:ordfinite}, for $T$ a Hecke operator
$$
\chr(T ~|~\MSord) \equiv 
\chr(T ~|~\MSo{\wt{\F}(M,L)}^{\ord}) \mod{p^M,w^L}.
$$
Thus, we can (in theory) recover the true characteristic polynomials to any degree of accuracy by taking $M$ and $L$ large enough.

The method of section \ref{sec:charpolyhecke} to form a basis $\MSord$ was to produce ``random" elements of $\MSord$ until one was the start of a $\Lambda$-basis of the space.  Then produce random elements until one has two elements forming the start of a $\Lambda$-basis.  Continue this until we have a full basis (whose size we know as in footnote \ref{footnote:hida}).  

To carry this method out in $\MSo{\wt{\F}(M,L)}^{\ord}$, we note that we can form elements in  $\MSo{\wt{\F}(M,L)}$ as described in section \ref{sec:random}.  Note that this requires solving the difference equation in  $\wt{\F}(M,L)$ which is done explicitly in Proposition \ref{prop:explicitfinitediff}.  To form elements of $\MSo{\wt{\F}(M,L)}^{\ord}$ one then just needs to iterate the $U_p$-operator.\footnote{To verify if a symbol $\Phi$ is actually in the ordinary subspace, one looks at $\Phi$, $\Phi | U_p$, $\Phi |U_p^2$, $\dots$, until there is a relation.}

Further, to determine when elements of $\MSo{\wt{\F}(M,L)}^{\ord}$  are the start of a $\Lambda/ (p^M,w^L) \Lambda$-basis, we can invoke Corollary \ref{cor:extendfinite} and examine the associated vectors of total measures modulo $(p,w)\Lambda$.  If these vectors are linearly independent over $\Fp$, then the original elements are the start of a basis.

Lastly, if we have a basis $\B = \{\Phi_1^{\ord}, \dots , \Phi_d^{\ord}\}$ of $\MSo{\wt{\F}(M,L)}^{\ord}$ over $\Lambda/(p^M,w^L)$ in hand, we describe now how to compute the matrix of a Hecke operator $T$ with respect to $\B$ (which in particular gives the characteristic polynomial of $T$).  To this end, we have
\begin{equation}
\label{eqn:hecke1}
\Phi_i | T = \sum_j a_{ij} \cdot \Phi_j
\end{equation}
for some $a_{ij} \in \Lambda / (p^M,w^L)\Lambda$, and our job is to find the $a_{ij}$.  Applying $\alpha_{M,L}$, the vector of total measures map, we get
\begin{equation}
\label{eqn:hecke2}
\alpha_{M,L}(\Phi_i | T) = \sum_j a_{ij} \cdot \alpha_{M,L}(\Phi_j)
\end{equation}
Since $\alpha_{M,L}$ is injective (Proposition \ref{prop:vtmfinite}), any solutions to \eqref{eqn:hecke2} will also be solutions to \eqref{eqn:hecke1}.  Thus we have reduced our question to solving linear equations over $\Lambda / (p^M,w^L) \Lambda$.  

Since the maximal ideal of $\Lambda / (p^M,w^L) \Lambda$ is not principal, solving linear equations over this ring is not as simple as over say $\Z/p^M\Z$.  So we include here at least  a few words about how one can do this.  Assume we have a consistent system of linear equations over $\Lambda / (p^M,w^L) \Lambda$:
$$
\sum_i a_i(w) \cdot {\bf v_i}(w) = {\bf u}(w);
$$
that is, the ${\bf v}_i(w)$ and ${\bf u}(w)$ in $\left(\Lambda / (p^M,w^L) \Lambda\right)^t$ are given and we must find $a_i(w)$ in $\Lambda / (p^M,w^L) \Lambda$ solving this equation.  Evaluating at $w=0$ yields
$$
\sum_i a_i(0) \cdot {\bf v_i}(0) = {\bf u}(0),
$$
which is a consistent system over $\Z/p^M\Z$.  Standard methods then gives us the values of $a_i(0)$ for each $i$.  Then differentiating and evaluating at $w=0$ gives:
$$
\sum_i a'_i(0) \cdot {\bf v_i}(0)  = {\bf u}'(0) - \sum_i a_i(0) \cdot {\bf v_i}'(0)
$$
This is another system of linear equations over $\Z/p^M\Z$ (with the $a'_i(0)$ as the unknowns) which we can again solve.  Repeating this method gives the values of each derivative of $a_i(w)$ at $w=0$.  From this information, we can recover $a_i(w)$ for each $i$ as desired.

\section{Data and examples}
In this section, we describe some sample computations that we have carried out with the algorithms implemented as part of this project. Specifically, we include computations of formal $q$-expansions of Hida families, the structure of Hida--Hecke algebras, $L$-invariants of modular forms and their symmetric squares, and  two-variable $p$-adic $L$-functions of Hida families. For the sake of presentation, we have elided some of the data in this section; the full data is presented in Appendix~\ref{DataAppendix}.

We briefly indicate about how much time and space each example took to compute on a `member server' on the SageMathCloud in Sage 5.11. These are meant to be ballpark estimates; for instance, we ran example~\ref{ex:X011_qexp} several times taking between 6 minutes and 10 minutes with this difference attributed mostly to the varying load of the server.

\subsection{Examples of $q$-expansions in families}\label{sec:examples_qexp}
We may view a $p$-ordinary family of eigenforms as a formal $q$-expansion
\begin{equation}\label{eqn:formal_qexp}
	\mathcal{F}=\sum_{n\geq1}a_n(k)q^n,
\end{equation}
where the $a_n(k)$ are Iwasawa functions of the $p$-adic weight variable $k$. 

If $\mathcal{F}$ is a normalized (i.e.\ $a_1=1$) eigenform, then the $a_n(k)$ are determined by the $a_\ell(k)$, for $\ell$ prime, using the standard Hecke operator recurrence relations. We may thus obtain the $q$-expansion from knowing the $a_\ell(k)$. Furthermore, $a_\ell(k)$ is the Hecke eigenvalue of the Hecke operator $T_\ell$. Thus, if $\Phi$ is a family of Hecke eigensymbols, we may compute the corresponding $a_\ell$ by comparing $\Phi|T_\ell$ with $\Phi$. 

More precisely, since $a_\ell(k)$ is an Iwasawa function, there exist $A_\ell(W)\in\Z_p\llbracket W \rrbracket$ such that $a_\ell(k)=A_\ell((1+p)^k-1)$.  Comparing $\Phi|T_\ell$ with $\Phi$ directly yields $A_\ell(W)$, and then a simple substitution yields $a_\ell(k)$.  In fact, our computations take place in the larger ring $\Zp\llbracket w \rrbracket = \Zp\llbracket W/p \rrbracket$.  However, by Theorem \ref{thm:extend}, we know that the eigenvalues $A_\ell(w)$ must land in the subring $\Zp\llbracket W \rrbracket$.

Below, we provide the first few $a_\ell$ for some examples of families passing through specific newforms. See \S\S\ref{sec:charpolyhecke} and \ref{sec:decomp} for the method used to isolate these examples. We remind the reader that specializing the variable $k$ below to a specific non-negative integer $k_0$ gives modular forms of weight $k_0+2$.

\begin{example}\label{ex:X03:_qexp}
Let $p=5$ and $N=1$ in weights congruent to $2$.  In this case $\MSord$ is 1-dimensional and this dimension is entirely explained by the ordinary Eisenstein family.  Using the methods of section \ref{sec:charpolyhecke}, we can produce a basis of this space and since this basis is 1-dimensional, its unique element $\Phi$ is an eigensymbol.  By comparing $\Phi | T_\ell$ with $\Phi$ we can thus compute families of ordinary Eisenstein eigenvalues.  For example, we compute that:
$$
a_2(k) =  3+ \left(4p + p^{2} + 4p^{4} + p^{5} + 4p^{6} + 4p^{7} + 3p^{8}\right)k+ \cdots+ \left(4p^{8}\right)k^{9} + O(p^{9}, k^{10})
$$
We note that in this case, we have an exact formula for $a_\ell(k)$ as $a_\ell(k) = 1 + \omega(\ell)\langle\ell\rangle^{k+1}$ for $\ell \neq 5$.  Expanding as a power series, we get
$$
a_\ell(k) = 1+\omega(\ell)\sum_{n\geq0}\frac{(k+1)^n\log_p^n(\ell)}{n!},
$$
which we note matches perfectly with the computations listed above. At $\ell=p=5$, the Hecke eigenvalue is the constant $1$ in the family and our computations witness this.
\end{example}

\begin{example}\label{ex:X011_qexp}
	Let $p=11$ and let $f$ be the unique cuspidal newform of weight $2$ and level $\Gamma_0(11)$ (i.e.\ the one corresponding to the elliptic (and modular) curve $X_0(11)$). In this case, $\MSord$ is 3-dimensional.  As in the previous example, one dimension is explained by the ordinary Eisenstein family.  The other two dimensions are explained by the Hida family passing through $f$ with this family contributing one dimension to each of the plus and minus subspaces.  Since the Eisenstein family lands in the plus subspace, we focus on the minus subspace as $(\MSord)^-$ is 1-dimensional.  Again, a basis of this space is automatically an eigensymbol and from this eigensymbol we computed $a_\ell$ for $\ell\leq 11$. The full data is included in Appendix~\ref{DAq}. A sample eigenvalue is
\begin{align*}
a_{11} &= 1+ \left(8p + 2p^{2} + \cdots + p^{10}\right)k+\cdots+ \left(3p^{10}\right)k^{11} + O(p^{11}, k^{12}).
\end{align*}
Note that plugging in $k=12-2$ into the data, one can verify that the eigenvalues agree with those of (the ordinary $11$-stabilization of) the modular discriminant $\Delta$ (up to precision $11^{11}$).

This example took about 7 minutes: computing the eigenfamily took about 5.5 minutes and computing the 5 eigenvalues took about 1.5 minutes. Computing the modular symbol space and the eigenfamily used about 55MB while the eigenvalues used about 45MB.
\end{example}

\begin{example}\label{ex:15a1_qexp}
	Let $p=5$ and let $f$ be the unique newform of weight $2$ and level $\Gamma_0(15)$ (i.e.\ the one corresponding to the elliptic (and modular) curve $X_0(15)$). The minus subspace is again 1-dimensional so a basis of it is an eigensymbol from which we computed $a_\ell$ for $\ell\leq 11$. The full data is included in Appendix~\ref{DAq}. A sample eigenvalue is
\begin{align*}
a_{5} &= 1+ \left(4p + 4p^{2} + \cdots + 4p^{10}\right)k+ \cdots + \left(2p^{10}\right)k^{11} + O(p^{11}, k^{12})
\end{align*}
Plugging in $k=6-2$, one can verify that the eigenvalues agree with those of (the ordinary $5$-stabilization of) of the unique newform of weight $6$ and level $\Gamma_0(3)$ (up to precision $5^{10}$). In weight $22$ and level $\Gamma_0(3)$, there are three Galois conjugacy classes of newforms; their Hecke eigenvalue fields are $\Q, \Q$, and $\Q(\sqrt{11\cdot59})$, respectively. The first two are not $5$-ordinary. The third Galois conjugacy class is only ordinary under one of the two embeddings of $\Q(\sqrt{11\cdot59})$ into $\Q_5$. We plugged in $k=22-2$ to the computed $a_\ell$ and they agreed with the Hecke eigenvalues of this weight $22$ and level $\Gamma_0(3)$ newform (up to precision $5^{11}$).

This example took about 11 minutes: 7 minutes for the eigenfamily and 4 minutes for the 5 eigenvalues. The eigenfamily computation used 45MB and the eigenvalues used 120MB.
\end{example}

\begin{example}\label{ex:19a1_qexp}
	Let $p=5$ and let $f$ be the ordinary $5$-stabilization of the (unique) newform of weight $2$ and level $\Gamma_0(19)$ (i.e.\ the one corresponding to the elliptic (and modular) curve $X_0(19)$).  In this case, $(\MSord)^-$ is 8-dimensional.  However, if $\m$ is the maximal ideal corresponding to $f$, we have that $(\MSord_\m)^-$ is 1-dimensional.  Indeed, $a_5(f) \equiv 3 \pmod{5}$ and $a_5(g) \equiv \pm 1 \pmod{5}$ for all of the remaining eigenforms (as they are all new at 5).  In particular, the operator $T_5^2-1$ acts invertibly on $(\MSord_\m)^-$ and topologically nilpotently on $(\MSord_{\m'})^-$ for all $\m' \neq \m$.  In particular, using the methods of section \ref{sec:decomp}, we can form a basis of $(\MSord_{\m})^-$ and obtain an eigensymbol as this space is 1-dimensional. The $a_\ell$ for $\ell\leq19$ are included in Appendix~\ref{DAq}. Here is $a_5$:
\begin{align*}
a_{5} &= 3 + 3p + 3p^{2} + p^{3} + 4p^{6}+ \left(5 + 2p^{2}+\cdots+ 2p^{6}\right)k+ \cdots+ \left(3p^{5} + 2p^{6}\right)k^{6} + O(p^{7}, k^{8})
\end{align*}
In weight $6$ and level $\Gamma_0(19)$, there are $4$ Galois conjugacy classes of newforms; their Hecke eigenvalue fields are $\Q, \Q$, $\Q(\sqrt{3\cdot59})$, and $K_4$, respectively, where $K_4$ is a totally real $S_4$-quartic extension of discriminant $101148696=2^3 \cdot 3^3 \cdot 11 \cdot 42571$. All of these are ordinary at $5$. The fourth Galois conjugacy class has one embedding into $\Q_5$ (the remaining three are into $\Q_{5^3}$) and the corresponding newform is the only one whose $a_5$ is congruent to the $a_5$ of the elliptic curve $X_0(19)$. Plugging in $k=6-2$ in the computed $a_\ell$ agrees with the $q$-expansion of this newform (up to precision $5^{7}$).

This example took about 58 minutes: 35 minutes for the eigenfamily and 23 minutes for the 8 eigenvalues. The eigenfamily computation used 77MB and the eigenvalues used 719MB.
\end{example}

\begin{example}\label{ex:N95p5_qexp}
	Let $p=5$. There are two Galois conjugacy classes of weight $2$ newforms of level $\Gamma_0(95)$; their Hecke eigenvalue fields are $K_3$ and $K_4$, where $K_3$ is the unique (real) cubic field of discriminant $148$, and $K_4$ is the unique totally real quartic field of discriminant $11344$. These represent the remaining $7$ dimensions in $(\MSord)^{-}$ from the previous example. In this example, we deal with the conjugacy class with cubic Hecke eigenvalue field as its $a_5$ is $1$ (while the other conjugacy class has $a_5=-1$) and hence the forms in this class will have an $L$-invariant (studied in \S\ref{sec:Linvar} below). The field $K_3$ has only one embedding into $\Q_5$ (the remaining embeddings landing in $\Q_{5^2}$). We let $f$ be the newform corresponding to the embedding in $\Q_5$ and remark it is congruent to an Eisenstein series. We may isolate it from its two other conjugates (as in section \ref{sec:decomp}) as its $a_{61}$ is $2\mod{5}$, whereas its conjugates have $a_{61}\equiv3\mod{5}$. The $a_\ell$ for $\ell\leq11$ are included in Appendix~\ref{DAq}. Here is $a_5$, which is used to compute the $L$-invariant below:
\begin{align*}
a_{5} &= 1+ \left(2p^{2} + 3p^{5}\right)k+ \left(p^{2} + 2p^{4} + p^{5}\right)k^{2}+ \left(2p^{4} + p^{5}\right)k^{4}+ \left(2p^{5}\right)k^{5}+ \left(3p^{5}\right)k^{6} + O(p^{6}, k^{7})
\end{align*}

This example took about 4 hours and 25 minutes: 4 hours and 20 minutes for the eigenfamily and 5 minutes for the 5 eigenvalues. The eigenfamily computation used 1150MB and the eigenvalues used 110MB.

\end{example}

\begin{example}\label{ex:big_ex_qexp}
For a bigger example, consider $p=11$ with tame level $N=31$. We take $f$ to be 11-stabilization of the unique weight two newform on $\Gamma_0(31)$. This form has coefficients in $\Q$. The space $(\MSord)^-$ is 29-dimensional in this case, but its localization at the maximal ideal corresponding to $f$ is $1$-dimensional. We isolate the Hida family through $f$ by iterating the operators 
$T_2-8,U_{11}-1,U_{11}+1$. The $a_\ell$ for $\ell\leq11$ are included in Appendix~\ref{DAq}. Here is the $a_{11}$ in the family:
\[
	a_{11} = 2 + 5p + 9p^{2} + 8p^{3} + 8p^{4} + 6p^{5} + 3p^{6} + 2p^{7}+ \left(10p + \cdots + 8p^{7}\right)k+\cdots+ 10p^{7}k^{7} + O(p^{8}, k^{8})
\]

This example took  about 6 hours and 37 minutes: 6 hours and 20 minutes for the eigenfamily and 17 minutes for the 5 eigenvalues. The eigenfamily computation used 370MB and the eigenvalues used 290MB.
\end{example}


\subsection{The structure of Hida algebras}
Here, we summarize some computations of the structure of the connected components of Hida algebras.

\begin{example}
\label{ex:p=3N=11}
Let $p=3$ and $N=11$.  In this case, there are two $3$-ordinary cuspforms in any even weight.  In weight 2, one of these forms comes from the ordinary 3-stabilization of the cuspform associated to $X_0(11)$.  The other form is the unique newform of level 33, and, moreover, these two forms admit a congruence modulo 3.  
Note then $(\MSord)^-$ is 2-dimensional with $(\MSord_\m)^- \simeq (\MSord)^-$ making this example fundamentally different from the examples in section \ref{sec:examples_qexp} where we were always able to cut down to a 1-dimensional space.

Nonetheless, we know that $\T_\m$ is a $\Lambda$-algebra of rank 2 and we seek to understand its structure.  From a geometric perspective, possibilities for $\Spec(\T_\m)$ include two copies of weight space glued to together at a finite collection of points (possibly only at the point of characteristic $p$) or a ramified cover of weight space ramified at finitely many points.  

To further understand $\T_\m$, consider a Hecke operator $T_\ell$ (or $U_\ell$) acting on $(\MSord)^-$ and let $f_\ell$ denote its characteristic polynomial, which is a monic polynomial of degree 2 over $\Lambda=\Zp\llbracket W\rrbracket$.  Let $d_\ell(W) \in \Lambda$ denote the discriminant of this polynomial.  By $p$-adic Weierstrass preparation, we can write 
$$
d_\ell(W) = p^{\mu_\ell} \cdot P_\ell(W) \cdot V_\ell(W)
$$
where $P_\ell(W)$ is a distinguished polynomial of degree say $\lambda_\ell$ and $V_\ell(W)$ is a unit power series.  We can use information from this decomposition to understand $\T_\m$.  For example, if $\lambda_\ell$ is odd, then $\Lambda[T_\ell]$ is a ramified extension of $\Lambda$ forcing $\T_\m / \Lambda$ to be ramified.

Using the methods outlined in section \ref{sec:charpolyhecke}, we computed approximations to the characteristic polynomials $f_\ell$ described above.  For example, for $\ell = 2$, we computed this discriminant to be:
\begin{multline*}
d_2(W) = p^{2} + O(p^{11}) + \left(2 + 2p + 2 p^{2} + 2 p^{4} + p^{5} + p^{7} + O(p^{10})\right)W + \dots + (1 + O(p))W^8 + O(W^9) 
\end{multline*}
Note that this power series has $\lambda$-invariant 1 and thus has a unique root $\alpha_2$ which is defined over $\Zp$.
We explicitly found the following approximation to $\alpha_2$:
$$
\alpha_2 \equiv p^2 + 2  p^4 + 2 p^6 + p^7 + p^8 + p^9 \mod{p^{11}}.
$$
We thus get that 
$$
d_2(W) = (W-\alpha_2) \cdot V_\ell(W)
$$
where $V_\ell(0) \equiv -1 \pmod{p}$.  In particular, $-V_\ell(W)$ is a square in $\Lambda$ and we see that $d_2(W)$ and $-(W-\alpha_2)$ differ multiplicatively by a square.  In particular, $\Lambda[T_2] = \Lambda[\sqrt{-(W-\alpha_2)}]$

Now for $R$ a ring which is a finite and free $\Lambda$-module, let $\disc(R) \subseteq \Lambda$ denote its discriminant ideal.  Write $M \in M_2(\Lambda)$ for the change of basis matrix corresponding to the embedding $\Lambda[T_2] \subseteq \T_\m$ both of which are free $\Lambda$-modules of rank 2.  We then have
$$
\disc(\Lambda[T_2]) = \det(M)^2 \cdot \disc(\T_\m).
$$
Since $\disc(\Lambda[T_2]) = (W-\alpha_2)\Lambda$ is a square-free ideal, we must have that $\det(M)$ is a unit and $\Lambda[T_2] = \T_\m$.  In particular, $\T_\m \simeq \Lambda[\sqrt{-(W-\alpha_2)}]$.

As a check, we computed $d_\ell(W)$ for all  primes $\ell < 11$.  In each case, $d_\ell(W)$ had $\lambda$-invariant equal to 1, and its unique root $\alpha_\ell$ was congruent to $\alpha_2$ modulo the precision of the computation.

Computing the basis of the two-dimensional $(\MSord_\m)^-$ took 27.5 minutes and 107MB, while computing the Hecke polynomials for primes $\leq11$ took 33.5 minutes and 462MB.
\end{example}

\begin{example}
Let $p=37$ and $N=1$.  This example gains its fame from the fact that $37$ is an irregular prime with $B_{32}$ having positive valuation at 37.  In particular, there is a cuspform $f$ of weight 32 congruent to the Eisenstein series $E_{32}$ modulo 37.  For this reason we consider the 30th component of weight space (corresponding to the classical weight 32).  On this component there are exactly 3 ordinary normalized eigenforms:\ $f$, $E_{32}^{\ord}$ (the ordinary $37$-stabilization of $E_{32}$), and a third form not congruent to either $f$ or $E_{32}^{\ord}$.

Let $\m$ denote the maximal ideal of $\T$ corresponding to $f$ and $E_{32}^{\ord}$.  
In this case, the Eisenstein symbols live in the plus part of $\MSord$ and thus $(\MSord_\m)^+$ is rank 2 over $\Lambda$ (with one dimension coming from the Eisenstein series and the other coming from the Hida family through $f$).  Using the methods of sections \ref{sec:decomp}, we can form a basis (of size 2) of this space.  As in the previous example, we compute the discriminant of the characteristic polynomial of $T_\ell$ for various $\ell$.  For example, for $\ell = 2$, we get:
\begin{align*}
d_2(W) &= 16 p^{2} + 23 p^{3} + 6 p^{5} + 12 p^{6} + 24 p^{7} + 12 p^{8} + 27 p^{9} + 17 p^{10} + 14 p^{11} + O(p^{12})\\&
+ \left(34 p + 13 p^{2} + 2 p^{3} + 27 p^{4} + 36 p^{5} + 29 p^{6} + 32 p^{7} + 32 p^{8} + 35 p^{9} + 17 p^{10} + O(p^{11})\right)W\\&
+ \left(25 + 20 p + 26 p^{2} + 9 p^{3} + 22 p^{4} + 6 p^{5} + 3 p^{6} + 8 p^{7} + 17 p^{8} + 24 p^{9} + O(p^{10})\right)W^{2}\\&
 + \cdots + \left(6 + O(p)\right)W^{11} + O(W^{12})
\end{align*}
which we note has $\lambda$-invariant 2.  In fact, looking at the Newton polygon of this power series we see that it has two roots each of valuation 1; call these roots $\alpha_2$ and $\beta_2$.  By inspection, we can only find a single root (mod $p^6$), namely:
$$
\alpha_2 = 23 p + 10 p^{2} + 35 p^{3} + 36 p^{4} + 34 p^{5} + O(p^{6})
$$
This suggests that $\alpha_2 = \beta_2$, that is, that $d_2(W)$ has a double root at $\alpha_2$.  


We note that a computer computation alone could never prove the equality $\alpha_2 = \beta_2$ as we are always working modulo a power of $37$.  Nonetheless, in this example, we can argue as follows.  First note that if $\alpha_2 \neq \beta_2$, then by the same arguments as in Example \ref{ex:p=3N=11} (since $\disc(\Lambda[T_2])$ is squarefree), we have 
$$
\T_\m = \Lambda[T_2] \simeq \Lambda[\sqrt{u(W-\alpha_2)(W-\beta_2)}]
$$
with $u\in \Z_{37}^\times$.  In particular, $\T_\m$ is a domain and its spectrum is thus a single irreducible component.
However, looking at the associated Galois representations we will see that this is impossible.  Indeed, at the Eisenstein points in the family, the associated Galois representation is reducible while at generic cuspidal points this representation is irreducible.  If $\Spec(\T_\m)$ were irreducible, then all Galois representations would have the same behavior (irreducible vs.\ reducible) except at a finite set of points.  This contradiction forces $\alpha_2 = \beta_2$.

Hence $d_2(W) = (W-\alpha_2)^2 \cdot V_\ell(W)$ with $V_\ell(0) \equiv 25 \mod{p}$.  In particular, $V_\ell(W)$ is a square and thus $\Lambda[T_2] \simeq \Lambda[Y] / (Y^2 - (W-\alpha_2)^2)$.  Arguing again with discriminant ideals, we have
$$
\disc(\Lambda[T_2]) = \det(M)^2 \cdot \disc(\T_\m)
$$
where $M$ is the change of basis matrix coming from the inclusion $\Lambda[T_2] \subseteq \T_\m$.  Since 
$\disc(\Lambda[T_2]) = (W-\alpha_2)^2 \Lambda$, we have $\det(M)^2 = 1$ or $(W-\alpha_2)^2$.  In the later case, we would have that the discriminant ideal of $\T_\m$ over $\Lambda$ is a unit, implying that $\T_\m$ is an \'{e}tale $\Lambda$-algebra, and in particular, that the map $\Lambda\rightarrow\T_\m$ is unramified. This implies that $(p,W)\T_\m=\m\T_m$. Since $\m$ corresponds to a $q$-expansion in $\mathbb{F}_p$, we have that $\T_\m/(p,W)\T_\m\cong\T_\m/\m\T_m\cong\mathbb{F}_p$ which is a one-dimensional vector space over $\Lambda/(p,W)\cong\mathbb{F}_p$. By Nakayama's Lemma, $\T_m$ must be rank one over $\Lambda$, a contradiction. Thus, $\T_\m\cong\Lambda[T_2]\cong\Lambda[Y] / (Y^2 - (W-\alpha_2)^2)$.

As a check, we computed $d_\ell(W)$ for $\ell \leq 11$ and in each case $\lambda(d_\ell)=2$ and $\alpha_2$ was a root of $d_\ell(W)$ modulo our precision.  

Geometrically, the spectrum of this ring is two copies of weight space glued together at the weight $\alpha_2$.  This picture is completely consistent with what is known already in this example.  Indeed, the 37-adic $\zeta$-function has $\lambda$-invariant 1.  Thus, the Eisenstein family and the cuspidal family meet at a unique weight $k_z$ -- namely the unique root of $\zeta_{37}(1-k)$.  In \cite[\S6.2.1]{Yves}, this weight is computed to tremendous precision (1000 $p$-adic digits) with the first few digits being
$$
k_z = 13 + 20p + 30p^{2} + 8p^{3} + 11p^{4} + O(p^{5}).
$$
To compare with our computations of $\alpha_2$, we note that the weight $k_z$ in the $T$-variable corresponds to $(1+p)^{1-k_z-2}-1$ and we do indeed have that
$$
(1+p)^{1-k_z-2}-1 \equiv \alpha_2 \mod{p^6}.
$$

Computing the basis of the two-dimensional $(\MSord_\m)^+$ took 1 hour and 56 minutes and 410MB, while computing the Hecke polynomials for primes $\leq11$ took 9.5 minutes and 169MB.
\end{example}


\subsection{$L$-invariants}\label{sec:Linvar}
$L$-invariants arise when a $p$-adic $L$-function vanishes at a point of interpolation due to the vanishing of the Euler-type interpolation factor. The earliest known example of this phenomenon of so-called ``trivial zeroes'' is due to Ferrero--Greenberg \cite{FG}: if $\psi$ is an even Dirichlet character, then
	\[L_p(1-n,\psi)=\left(1-\psi(p)p^{n-1}\right)L(1-n,\psi\omega^{-n}),\text{ for }n\in\Z_{\geq1},
	\]
	so $\left(1-\psi(p)p^{n-1}\right)$ vanishes at $n=1$ whenever $\psi(p)=1$. In \cite{MazurTateTeitelbaum}, Mazur--Tate--Teitelbaum discovered the same type of vanishing occurs for the $p$-adic $L$-function of an elliptic curve, $E$, with split, multiplicative reduction at $p$. The interpolation property gives
	\[ L_p(1,E)=(1-a_p(E)^{-1})\frac{L(1,E)}{\Omega_E},
	\]
	where $\Omega_E$ is the N\'{e}ron period of $E$. When $E$ is split multiplicative at $p$, $a_p(E)=1$, so that the $p$-adic $L$-function vanishes for trivial reasons at $s=1$ (more generally, weight $2$ newforms of level exactly divisible by $p$ and whose $a_p=1$ share the same behaviour). They introduced a new quantity, the $p$-adic $L$-invariant, $\L_p(E)$, given in terms of the $p$-adic Tate parameter of $E$ and conjectured that
	\[
		\left.\frac{d}{ds}L_p(s,E)\right|_{s=1}=\L_p(E)\frac{L(1,E)}{\Omega_E}.
	\]
	This was proved by Greenberg and Stevens in \cite{GS} (the more general case of newforms included). A main ingredient of their proof was a formula they gave for the $L$-invariant of a weight $2$ newform which is germane to our work:
	\[ \L_p(f)=-2\left.\frac{d}{dk}\log_pa_p(k)\right|_{k=0},
	\]
	where $a_p(k)$ is the $p$th Fourier coefficient in the formal $q$-expansion of Equation~\eqref{eqn:formal_qexp} for the Hida family through $f$. When $f$ corresponds to an elliptic curve, Sage can already compute these $L$-invariants using Tate parameters, and we show below that our method provides the same answer. On the other hand, we can also compute cases that don't correspond to elliptic curves, i.e.\ newforms whose Hecke eigenvalues don't lie in $\Q$. Furthermore, our code provides new computations for the $L$-invariant of the symmetric squares of an eigenform (really, the trace-zero adjoint of an eigenform, which is a twist of the symmetric square). In \cite{G94}, Greenberg proposed a general theory of $L$-invariants for ordinary motives providing an arithmetic candidate $L^{\Gr}_p(M)$ for the $p$-adic $L$-invariant of an ordinary motive $M$. When $M$ is the trace-zero adjoint, $\ad^0\!f$, of a newform, Hida and the third author \cite{Hida-Linv,Harron-thesis} gave a formula for Greenberg's $L$-invariant which Dasgupta \cite{samit} has recently shown is the actual $L$-invariant: if $f$ is any $p$-ordinary newform of weight $k_0+2$ of such that $p^2$ does not divide its conductor, then the $p$-adic $L$-function of $\ad^0\!f$ has a trivial zero at $s=1$ and its $L$-invariant is given by
	\begin{equation}\label{eqn:symsquare_Linvar}
		\L_p(\ad^0\!f)=-2\left.\frac{d}{dk}\log_pa_p(k)\right|_{k=k_0}.
	\end{equation}
	We remark that unlike the case of modular forms themselves which (in the $p$-ordinary case) only have trivial zeroes in weight $2$ and conductor exactly divisible by $p$, the trace-zero adjoint always has a trivial zero. We may therefore consider the $L$-invariant as varying in the Hida family and our computations allow us to compute the Iwasawa function giving the adjoint $L$-invariant in a family. Also, note that the non-vanishing of the $L$-invariant, an important part of Greenberg's conjecture, is only known in the cases of Dirichlet characters and split, multiplicative elliptic curves (where the result is one from transcendence theory: the theorem of St-\'Etienne \cite{STE}, which says that the Tate parameter is transcendental). In particular, the $L$-invariants of newforms that do not correspond to elliptic curves are not known to be non-zero. Furthermore, up to now, all that has been known for the adjoint of a newform is that either the $L$-invariants are all zero in a Hida family, or all but finitely $L$-invariants are nonzero (since an Iwasawa function has finitely many zeroes). Our computations provide non-vanishing results for specific forms and forms in families. In particular, they show that the adjoint $L$-invariants of all forms in the Hida families of Examples \ref{ex:X011_qexp}--\ref{ex:19a1_qexp} are nonzero and there can be at most one form in Example~\ref{ex:N95p5_qexp} with vanishing adjoint $L$-invariant.

We collect some values of $L$-invariants of modular forms as well as their trace-zero adjoints. Again, full data is available in Appendix~\ref{DAL}.

\begin{example}
	The Hida families in Examples \ref{ex:X011_qexp} and \ref{ex:15a1_qexp} pass through elliptic curves with split, multiplicative reduction at $p$ (when $k=0$) and the data of $a_p(k)$ provided above allows us to compute the $p$-adic $L$-invariants of these curves. We obtain:
	\begin{align*}
		\L_{11}(11a1)&=6p + 5p^{2} + 7p^{3} + 7p^{4} + 7p^{5} + 7p^{6} + 4p^{7} + 3p^{8} + 6p^{9} + 7p^{10} + O(p^{11})\\
		\L_5(15a1)&=2p + p^{3} + p^{4} + 3p^{5} + 4p^{6} + 2p^{7} + 4p^{8} + 4p^{9} + p^{10} + O(p^{11})
	\end{align*}
	We verified that these agree with Sage's already available computation of these $L$-invariants (which is \emph{much} quicker) to the given precision.
\end{example}

\begin{example}
	The Hida family in Example~\ref{ex:N95p5_qexp} passes through a newform with $p=5$ exactly dividing the level and $a_p=1$. We compute its $L$-invariant, concluding that it is non-zero, to be:
	\[
		\L_5(f)=p^{2} + 4p^{3} + 4p^{4} + 3p^{5} + O(p^{6}).
	\]
\end{example}

\begin{example}\label{ex:sym2family_Linvar}
	In fact, for each of the Hida families in Examples~\ref{ex:X011_qexp}--\ref{ex:N95p5_qexp} (which we now label $\F_{11}, \F_{15}$, $\F_{19}$, and $\F_{95}$, respectively), we can use Equation~\eqref{eqn:symsquare_Linvar} to provide a formula for the symmetric square $L$-invariant as a function of the weight. The full results can be found in Appendix~\ref{DAL}.
	\begin{align*}
		\L_{11}(\ad^0\!\F_{11})&=6p + \cdots + 7p^{10}+ \left(10p^{2} + \cdots + 9p^{9}\right)k+\cdots+ \left(4p^{10}\right)k^{9} + O(p^{11}, k^{10}).
	\end{align*}
	Plugging in $k=12-2$ yields the value of the $11$-adic symmetric square $L$-invariant of $\Delta$:
	\[
		\L_{11}(\ad^0\!\Delta)=6p + 6p^{2} + 5p^{3} + 3p^{4} + 3p^{5} + 3p^{6} + 8p^{7} + 6p^{8} + 5p^{9} + 8p^{10} + O(p^{11}).
	\]
	
	
	For the family $\F_{15}$ of Example \ref{ex:15a1_qexp}, we obtain:
	\begin{align*}
		\L_{5}(\ad^0\!\F_{15})&=2p + \cdots + 4p^{9} + p^{10}+ \left(3p^{2} + \cdots + 4p^{10}\right)k+\cdots+ \left(3p^{9} + 3p^{11}\right)k^{9} + O(p^{11}, k^{10}).
	\end{align*}
	
	For the family $\F_{19}$ of Example \ref{ex:19a1_qexp}, we obtain:
	\begin{align*}
		\L_{5}(\ad^0\!\F_{19})&=p + 4p^{2} + 2p^{3} + 4p^{4} + p^{5} + 4p^{6}+ \left(3p^{4} + 3p^{5} + 3p^{6}\right)k+\cdots+ \left(2p^{6}\right)k^{4} + O(p^{7}, k^{5}).
	\end{align*}
	
	For the family $\F_{95}$ of Example~\ref{ex:N95p5_qexp}, we obtain:
	\begin{align*}
		\L_{5}(\ad^0\!\F_{95})&=p^{2} + 4p^{3} + 4p^{4} + 3p^{5}+ \left(p^{2} + 4p^{3} + 4p^{4}\right)k+\cdots+ \left(3p^{4} + 4p^{5}\right)k^{3} + O(p^{6}, k^{4}).
	\end{align*}
	
	As Iwasawa power series (in the variable $W$), these $L$-invariants have $\mu$-invariant at least $1$ (since a $\log_p(1+p)$ appears upon taking the derivative with respect to $k$). We verify that the first three $L$-invariant functions indeed have $\mu=1$ and $\lambda=0$, thus implying they never vanish. For $\L_{5}(\ad^0\!\F_{95})$, the $\mu$- and $\lambda$-invariants are both $1$ and a computation shows that the $L$-invariant vanishes at a weight congruent to $4\cdot 5 + 5^4\mod{5^5}$; of course, this weight is not expected to be classical.
\end{example}

\subsection{Two-variable $p$-adic $L$-functions}
A $p$-ordinary eigenform $f$ of classical weight $k\geq2$ has a $p$-adic $L$-function $L_p(s,f)$ attached to it following the work of Manin, Amice--V\'elu, and Vi\v{s}ik. Varying the form $p$-adically in a Hida family one can expect to `glue' the one-variable functions together to obtain a two-variable $p$-adic $L$-function $L_p(s,\kappa)$ where $\kappa$ is a weight variable around a neighborhood of $k$. That this is the case is due to Ohta (unpublished), Mazur--Kitagawa \cite{Kitagawa}, and Greenberg--Stevens \cite{GS}; it was a fundamental ingredient in the latter's proof of the Mazur--Tate--Teitelbaum conjecture. Greenberg has conjectured that the generic order of vanishing of $L_p(s,\kappa)$ along the line $s=\kappa/2$ is at most one (and congruent to the sign of the functional equation of $L_p(s,f)$ modulo $2$) (see \cite[p.~439]{NP} for this statement and for some important consequences of its proof). Additionally, Greenberg and Stevens end the introduction to \cite{GS} by asking about the linear factors of the leading term in the expansion of $L_p(s,\kappa)$ about $s=1$ and $\kappa=2$ when the sign is $-1$.

In this section, we include a few sample computations of two-variable $p$-adic $L$-functions through overconvergent modular symbols.  To motivate these computations, we quickly review the single variable case.  In \cite[Theorem 8.3]{Stevens-Rigid}, Glenn Stevens gives a construction of the $p$-adic $L$-function of an eigenform $f$ solely in terms of its corresponding overconvergent modular eigensymbol $\Phi_f$, the unique overconvergent eigensymbol with the same system of eigenvalues as $f$.  To form the $p$-adic $L$-function of $f$, one simply takes $L_p(f):=\Phi_f(\{\infty\}-\{0\})$ which is a distribution on $\Zp$ and restricts this distribution to $\Zpx$.  
Now if $\Phi \in \MSo{\bDsc}$ is a {\it family} of overconvergent eigensymbols, we analogously define the two-variable $p$-adic $L$-function $L_p(\Phi)$ to be the restriction of  $\Phi(\{\infty\}-\{0\})$ to $\Zpx$.  The result is a family of distributions on $\Zpx$ whose specialization to any weight is the $p$-adic $L$-function of the corresponding eigenform in that weight.

In \cite[\S9]{PS1}, the penultimate author and Stevens explain how in practice one can compute single variable $p$-adic $L$-functions from overconvergent modular eigensymbols. The same method, which we describe now, applies in our case and allows us to compute two-variable $p$-adic $L$-functions attached to eigenfamilies of overconvergent modular symbols. 

For $\Phi \in \MSo{\bDsc}$ an eigenfamily, set $L_p(\Phi) := \Phi(\{\infty\}-\{0\})\big|_{\Zp^\times}$.  Let $T$ denote the cyclotomic variable of the $p$-adic $L$-function so that $L_p(\Phi) = \sum a_n T^n$, where the $a_n$ are functions on weight space.  Thus, the transformation from the $T$ variable to the $s$ variable is obtained by setting $T=\gamma^s-1$. Then, as in \cite[\S9.2]{PS1}, we have
\[
a_n = \int_{\Zp^\times}\binom{\log_\gamma\langle z\rangle}{n}dL_p(\Phi)=\sum_{j\geq1}c_j^{(n)}\sum_{a=1}^{p-1}\omega(a)^{-j}\int_{a+p\Zp}(z-\omega(a))^jdL_p(\Phi) \in \Lambda,
\]
where the coefficients $c_j^{(n)}$ are defined by
\[
	\binom{\log_\gamma y}{n}=\sum_{j\geq1}c_j^{(n)}y^j.
\]
To approximate the coefficient $a_n$, we must truncate the above infinite sum describing it.  
Since $v_p\left(\ds\int_{a+p\Zp}(z-\omega(a))^jd\mu_\Phi\right)\geq p^{j-v_p(\lambda)}$, we can determine the error in approximating the two-variable $p$-adic $L$-function in this way from the following lemma.

\begin{lemma} For $n\geq1$,
	\begin{enumerate}
		\item If $j\leq n$, then
			\[v_p(c^{(n)}_j)\geq-\left\lfloor\frac{n}{p-1}\right\rfloor-j,\]
		\item if $j>n$, then
			\[v_p(c^{(n)}_j)\geq-\left\lfloor\frac{n}{p-1}\right\rfloor-n-\left\lfloor\frac{j}{p}\right\rfloor.\]
	\end{enumerate}
\end{lemma}
\begin{proof}
By grouping like powers of $y$ in the product
\[
	\left(\frac{1}{\log_p\gamma}\left(y-\frac{y^2}{2}+\cdots\right)\right)\cdot\left(\frac{1}{\log_p\gamma}\left(y-\frac{y^2}{2}+\cdots\right)-1\right)\cdots\left(\frac{1}{\log_p\gamma}\left(y-\frac{y^2}{2}+\cdots\right)-(n-1)\right),
\]
we obtain the formula
\begin{equation}
	c^{(n)}_j=\frac{1}{n!}\sum_{\substack{a_0+\cdots+a_{n-1}=j \\ a_k\geq0,a_0>0}}(-1)^{j+n}\frac{\prod_{a_k=0}k}{\prod_{a_k\neq0}a_k\log_p\gamma}.
\end{equation}
Using that $v_p(\log_p(\gamma))=1$, we obtain
\begin{align*}
	v_p(c^{(n)}_j)&\geq-v_p(n!)+\min_{\substack{a_0+\cdots+a_{n-1}=j \\ a_k\geq0,a_0>0}}\left(\sum_{a_k=0}v_p(k)-\sum_{a_k\neq0}(v_p(a_k)+1)\right)\\
	&\geq-v_p(n!)-\max_{\substack{a_0+\cdots+a_{n-1}=j \\ a_k\geq0,a_0>0}}\sum_{a_k\neq0}(v_p(a_k)+1).
\end{align*}
We thus proceed to find an upper bound for
\[
	\max_{\substack{a_0+\cdots+a_{n-1}=j \\ a_k\geq0,a_0>0}}\sum_{a_k\neq0}(v_p(a_k)+1).
\]
Consider $j\leq n$. For a given partition $a_0+\cdots+a_{n-1}=j$, we obtain a $1$ in the sum for each non-zero $a_k$. If we combine $m$ non-zero terms $a_{k_1},\dots,a_{k_m}$ into one term, we gain $v_p(a_{k_1}+\cdots+a_{k_m})$, but lose $m-1+v_p(a_{k_1})+\cdots+v_p(a_{k_m})$. This is never a net gain, so the optimal partition is that with $j$ ones and $n-j$ zeroes, yielding an upper bound of $j$.

For $j>n$, note that the sum of the valuations occurring above is bounded by the same sum where there is no restriction on the length of the partition. Of all partitions of $j$, the one maximizing the sum of the $v_p(a_k)$ is $j=p+p+\cdots+p+r$ with $0\leq r<p$. Indeed, it is clearly optimal for the parts of the partition to be powers of $p$ and, for $r>1$, the part $p^r$ contributes $r$ while the sum $p^{r-1}+\cdots+p^{r-1}$ ($p$ times) contributes $p(r-1)>r$.
\end{proof}

With this lemma in hand, we computed several examples of two-variable $p$-adic $L$-functions. Our code produces a power series $F(T,w)$ with $T$ the cyclotomic variable and $w$ the same weight variable as above. In these examples we have made the following normalizations. As throughout the whole paper, the weight variable $k$ is normalized to correspond to modular forms of weight $k+2$ and is obtained by substituting $w=((1+p)^k-1)/p$. The cyclotomic variable $s$ is shifted by $1$ so that $s=0$ corresponds to the central point of the $L$-function of a weight two modular form; it is obtained by substituting $T=(1+p)^s-1$. Furthermore, the $p$-adic $L$-function we compute is only well-defined up to a unit power series in $w$ and we normalize it so that the first non-zero coefficient in $T$ is a power of $p$ times a power of $w$.

\begin{example}\label{ex:X011_padicL}
	We consider the same Hida family as in Example~\ref{ex:X011_qexp}: the unique $11$-adic Hida family of tame level $1$ and branch $m=0$. In order to get a non-zero $p$-adic $L$-function, we must use a symbol in the plus subspace, which no longer has dimension $1$. Indeed, the presence of an Eisenstein family raises the dimension of $(\MSord)^+$ to $2$. Still, localizing at the maximal ideal $\m$ corresponding to $X_0(11)$, we have that $(\MSord_\m)^+$ is one-dimensional since $X_0(11)$ is not Eisenstein at $2$. We obtain the following two-variable $p$-adic $L$-function (see Appendix~\ref{DApL} for the full expansion):
	\begin{align*}
		L_{11}(\F_{11},s,k)&=\alpha k -2\alpha s+\cdots+ (6p^7 )k^7 + (2p^7 )k^6s +\cdots + (7p^7 )ks^6 + (9p^7 )s^7+O(p^8,(k,s)^8)\\
			&=(k-2s)(\alpha +(6p^2 + 10p^3 + \cdots )k + (2p^3 + \cdots )k^2+\cdots),
	\end{align*}
	where $\alpha=p + 5p^2 + 9p^3 + 9p^4 + 9p^5 + 5p^6 + 8p^7+O(p^8)$. Note that despite the non-vanishing of the central $L$-value of $X_0(11)$, the $p$-adic $L$-function vanishes there due to the presence of an exceptional zero.
	
	In order to numerically verify Greenberg's conjecture, we consider the power series $F(T,w)$ that yields $L_{11}(\F_{11},s,k)$ as described above. The line $s=k/2$ corresponds to $(1+T)^2=1+pw$. Factoring
	\[
		F(T,w)=((1+T)^2-(1+pw))F_1(T,w)
	\]
	and specializing to $w=((1+T)^2 -1)/p$ yields 
	\[
		F_1(T,((1+T)^2 -1)/p)=10 + 10p + T + O(p^2,T^2).
	\]
	As this is a unit power series, this confirms Greenberg's conjecture for $\F_{11}$, showing that the order of vanishing along $s=k/2$ is exactly one throughout the Hida family.
	
	Given the family of overconvergent modular symbols $\F_{11}$, computing the $p$-adic $L$-function took 30 seconds and 3MB.
\end{example}

\begin{example}\label{ex:37a_padicL}
	Now, consider the elliptic curve 37a, the curve of rank $1$ of least conductor, and let $p=5$. This curve has $a_5=-2$. The space $(\MSord)^+$ has dimension $16$ with one dimension coming from the (ordinary $5$-stabilization of the) Eisenstein series of level 37 and the remaining dimensions being new of level $37\cdot5$; hence the remaining dimensions have $a_5\equiv\pm1\mod{5}$. Thus, localizing at the maximal ideal $\m$ corresponding to 37a yields a one-dimensional space. We computed the following $p$-adic $L$-function (full data in Appendix~\ref{DApL}):
	\begin{align*}
		L_{5}(\text{37a},s,k)&=\alpha k -2\alpha s+\cdots + (4p^5 + 3p^6 )k^6 +\cdots + (4p^5 )s^6+O(p^7,(k,s)^7),\\
		&=(k-2s)(\alpha + (3p^2 + 4p^3 + 3p^4 + 3p^5 + p^6)k + (p^2 + 2p^4 + p^5 + 4p^6)s+\cdots)
	\end{align*}
	where $\alpha=p + 2p^2 + 4p^3 + 2p^4 + p^6 + O(p^7)$.
	
	Again, factoring
	\[
		F(T,w)=((1+T)^2-(1+pw))F_1(T,w)
	\]
	and specializing to $w=((1+T)^2 -1)/p$ yields 
	\[
		F_1(T,((1+T)^2 -1)/p)=4 + 4p + 4pT + O(p^2,T^2).
	\]
	Again, this confirms Greenberg's conjecture that the order of vanishing along $s=k/2$ is exactly one throughout the Hida family.
	
	Given the family of overconvergent modular symbols, computing the $p$-adic $L$-function took 174 seconds and 66MB.
\end{example}

\begin{example}\label{ex:91b_padicL}
	The elliptic curve 91b1 is a curve of rank $1$ that has split multiplicative reduction at $p=7$. As such, its one-variable $p$-adic $L$-function vanishes to order $2$ at the central point, despite the classical $L$-function only vanishing to order $1$ there. Greenberg's conjecture states that the order of vanishing of the two-variable $p$-adic $L$-function along the line $s=k/2$ should however be $0$, generically. Our calculations verify this. First off, $(\MSord)^+$ is $9$-dimensional, with $2$ dimensions coming from Eisenstein series, 2 from isogeny classes of elliptic curves (91a and 91b), 2 from a Galois conjugacy class of newforms defined over $\Q(\sqrt{2})$, and 3 from a Galois conjugacy class of newforms defined over the cubic field of discriminant 316. The curve 91b1 is not Eisenstein at $2$ and has $a_7=a_{13}=1$. The quadratic (resp.~cubic) Galois conjugacy class has $a_{13}=-1$ (resp.~$a_7=-1$), so that, after localizing at the maximal ideal corresponding to the curve $91b1$, we obtain a one-dimensional space. The two-variable $p$-adic $L$-function we compute is
	\begin{align*}
		L_{7}(\text{91b1}, s,k)&=(p^2 + 6p^3 + 4p^4)k^2 + (2p^2 + p^3)ks + (5p^2 + 5p^3 + 6p^4)s^2+\cdots+O(p^5,(k,s)^5)
	\end{align*}
	(the full data is available in Appendix~\ref{DApL}).
	
	Considering $F(T,w)$ specialized to $w=((1+T)^2 -1)/p$, we obtain
	\[
		F(T,((1+T)^2 -1)/p)=\left(6 + 3p + 6T + O(p^2,T^2)\right)T^2.
	\]
	Since this only vanishes at $T=0$, the generic order of vanishing along $s=k/2$ is $0$, and in fact, the two-variable $p$-adic $L$-function only vanishes at $(s,k)=(0,0)$ on the line $s=k/2$.

	Given the family of overconvergent modular symbols, computing the $p$-adic $L$-function took 45 seconds and 24MB.
\end{example}



\begin{appendix}
\section{Comparing $\MSord$ with Greenberg--Stevens modular symbols}
\label{app:GS}

We note that Hida theory implies that ordinary $p$-adic families of cuspidal eigenforms extend to all of weight space.  We would thus hope to see that the Hecke-eigenvalues occurring in the ordinary subspace of either $\MSo{\bDso}$ or $\MSo{\bDsc}$ extend to bounded functions on the entire open disc of radius 1.  In particular, in this appendix,  we will establish this fact  by comparing $\MSo{\bDsc}^{\ord}$ to the space of Greenberg--Stevens two-variable modular symbols.

\subsection{Relevant measure spaces}

Let $\M(\Zpx \times \Zp)$ denote the space of $\Zp$-valued measures on $\Zpx \times \Zp$; that is, the continuous dual of the space of continuous functions on $\Zpx \times \Zp$.  We endow this space with a right action of $\SL_2(\Z)$ by
$$
(\mu | \gamma)(f(x,y)) = \mu(\gamma \cdot f(x,y)) = \mu(f(ax+cy,bx+dy))
$$
and with the structure of a $\tLambda := \Zp[[\Zpx]]$-module by:
$$
([a] \cdot \mu)(f(x,y)) = \mu(f(ax,ay))
$$
where $a \in \Zpx$ and $[a]$ is the natural image of $a$ in $\tLambda$.  

For $\kappa \in \W$, there is a ``specialization to weight $\kappa$" map:
\begin{align*}
\M(\Zpx \times \Zp) &\to \bD^0_\kappa \\
\mu &\mapsto \mu_\kappa
\end{align*}
where 
$$
\mu_\kappa(f) := \int_{\Zpx \times \Zp} f(y/x) \kappa(x) ~d\mu
$$
for $f \in \bA_\kappa$.

\begin{prop}
For $\kappa \in \W$, the specialization to weight $\kappa$ map 
$$
\M(\Zpx \times \Zp) \to \bD^0_\kappa 
$$
is $\sigop$-equivariant.  Further, this map is $\tLambda$-linear if  $\tLambda$ acts on $\bD^0_\kappa$ by $[a] \cdot \mu = \kappa(a) \mu$.
\end{prop}

\begin{proof}
For $\gamma = \smallmat$ in $\sigop$ and $f \in \bA$, we have
$$
(\mu \big| \gamma)_\kappa(f) =
(\mu\big|\gamma)( f(y/x) \kappa(x)) =
\mu\left( f\left(\frac{bx+dy}{ax+cy}\right) \cdot \kappa(ax+cy)\right),
$$
while
\begin{align*}
(\mu_\kappa \big|_\kappa \gamma)(f) 
&= \mu_\kappa\left(\kappa(a+cz) \cdot f\left(\frac{b+dz}{a+cz}  \right)\right) \\
&= \mu \left(\kappa(a+c(y/x)) \cdot f\left(\frac{b+d(y/x)}{a+c(y/x)}\right)  \kappa(x) \right)\\
&= \mu \left(\kappa(ax+cy) \cdot f\left(\frac{bx+dy}{ax+cy}\right)   \right),\\
\end{align*}
as desired.

Also, for $a \in \Zpx$, we have
\begin{align*}
([a] \cdot \mu)_\kappa(f) &= ([a] \cdot \mu)(f(y/x) \kappa(x)) = \mu(f(ay/ax) \kappa(ax))\\
&= \kappa(a) \mu(f(y/x) \kappa(x)) 
= \kappa(a) \mu_\kappa(f).
\end{align*}
\end{proof}


Note that  $\bDsc$ and $\bD(\bArc) := \Hom_{\cont}(\bA^0,\bArc)$ are both naturally $\tLambda := \Zp[[\Zpx]]$-modules.  Indeed, $\tLambda$ is naturally identified with measures on $\Zpx$, and  the Amice transform identifies measures on $\Zpx$ with bounded (rigid) functions on $\W$
which, by restriction to $W_m$,  naturally give elements of $\bArc$.
Then $\tLambda$ acts on $\bD(\bArc)$ or on $\bDsc$ simply  by scaling the values of the distribution.

\subsection{Comparing modular symbols}

We seek to compare $\M(\Zpx \times \Zp)$-valued modular symbols with $\bDsc$-valued modular symbols.  We begin with a map.

\begin{prop}
\label{prop:alpha}
There is a $\tLambda$-linear $\sigop$-map
$$
\alpha: \M(\Zpx \times \Zp) \lra \bD(\bArc)
$$
given by
$$
\mu \mapsto (f \mapsto (\kappa \mapsto \mu_\kappa(f))).
$$
That is, for $\mu \in \D(\Zpx \times \Zp)$,  the moments of the distributions $\mu_\kappa$ vary (rigid) analytically as $\kappa$ varies over weight space.
\end{prop}

\begin{proof}
That the moments of the $\mu_\kappa$ vary analytically is standard.  We leave the details the reader.

To see that $\alpha$ is $\sigop$-equivariant, first note that $\alpha$ commutes with specialization to weight $\kappa$.  That is,
$$
\xymatrix{
  {\M(\Zpx \times \Zp)} \ar@{->}[r]^-{\alpha} \ar[rd]  & {\bD(\bArc)} \ar[d]\\
      & \bDo_\kappa     }
$$
commutes; this follows directly from the definitions of these maps.   Since specialization to weight $\kappa$ is $\sigop$-equivariant with either $\M(\Zpx \times \Zp)$ or $\bD(\bArc)$ as a source, 
for $\mu \in \M(\Zpx \times \Zp)$, we have 
$$
(\alpha(\mu) | \gamma)_\kappa 
= \alpha(\mu)_\kappa \big|_\kappa \gamma 
= \mu_\kappa \big|_\kappa \gamma 
$$
while
$$
(\alpha(\mu | \gamma))_\kappa = (\mu | \gamma)_\kappa = \mu_\kappa \big|_\kappa \gamma.
$$
Thus $\alpha(\mu) | \gamma$ and $\alpha(\mu | \gamma)$ have the same specialization to weight $\kappa$ for all $\kappa \in \Wr$.  Thus, by definition, $\alpha(\mu | \gamma) = \alpha(\mu) | \gamma$ in $\bD(\bArc)$ as desired.

To see $\tLambda$-linearity, we can argue the same way since specialization to weight $\kappa$ (with either source) is $\tLambda$-linear if $\bDo_\kappa$ is acted on by $\tLambda$ with $[a]$ acting by $\kappa(a)$.
\end{proof}



For any $\Phis \in \MSo{\bD(\bArc)}^{\ord}$, we have that $\Phis$ is in the image of $U_p$, and thus as in  Lemma \ref{lemma:Upimprove}, $\Phis$ takes values in $\bDsc$.  Thus, 
$$
\MSo{\bDsc}^{\ord} = \MSo{\bD(\bArc)}^{\ord},
$$
and we have a Hecke-equivariant map
$$
\MSo{\M(\Zpx \times \Zp)}^{\ord} \otimes_{\tLambda} \bArc \to \MSo{\bDsc}^{\ord}.
$$
We aim to show that this map is an isomorphism, and thus the characteristic polynomial of a Hecke operator acting on the target is the same as the restriction to $\bArc$ of the characteristic polynomial of that Hecke operator on the source.  From this, we can deduce that the coefficients of the characteristic polynomials of Hecke operators on the target (which {\it a priori} are in $\bArc$) extend to an open disc of radius 1 in weight space.

\subsection{Control theorems}

Fix a non-negative integer $k \equiv m \mod{p-1}$ and consider the map $\tLambda \to \Zp$ given by evaluation at weight $k$.  Let $\p_k \subseteq \tLambda$ denote the kernel of this map; it is a principal ideal.  We now state several control theorems for spaces of Greenberg--Stevens modular symbols and for our spaces of families of modular symbols.  We note that in the below theorems, $p$ is {\it not} inverted.

Let $\widetilde{\P}_k$ denote the $\Zp$-span of $\binom{z}{j}$, for $j=0, \dots, k$, in the space of $\Zp$-valued continuous functions on $\Zp$, and set $\widetilde{\P}_k^\vee$ equal to the $\Zp$-dual of $\widetilde{\P}_k$.  We note that $\widetilde{\P}_k^\vee$ is isomorphic to a lattice in $\Sym^k(\Qp^2)$.  We have a surjective $\sigop$-equivariant map\footnote{We note that this map would not be surjective if we simply looked at the span of $z^j$ for $j=0,\dots,k$.  See, for instance, \cite[Lemma A.4]{BP}.}
\begin{align*}
\M(\Zpx \times \Zp) &\to \widetilde{\P}_k^\vee \\
\mu &\mapsto \left(f(z) \mapsto \int_{\Zpx \times \Zp} x^{k} f(y/x) ~d\mu\right),
\end{align*}
and thus a Hecke-equivariant map
$$
\MSo{\M(\Zpx \times \Zp)} \lra \MSo{\widetilde{\P}_k^\vee}.
$$

\begin{lemma}
\label{lemma:compGS}
The above map induces a Hecke-equivariant isomorphism
$$
\MSo{\M(\Zpx \times \Zp)}^{\ord} \ot \tLambda / \p_k \stackrel{\sim}{\lra} \MSo{\widetilde{\P}_k^\vee}^{\ord}
$$
on ordinary subspaces.
\end{lemma}

\begin{proof}
This isomorphism is implicitly given in \cite{GS} if we allow $p$ to be inverted.  The above integral version is given in \cite[Corollary A.9]{BP}.
\end{proof}

Now, set $\P_k$ equal to the span of $1$, $z$, $\dots$, $z^k$ in $\bAo_k$; we then have a surjective map $\bDo_k \lra \P_k^\vee$.  We now state a control theorem for families of modular symbols.

\begin{lemma}
\label{lemma:compOMS}
Specialization to weight $k$ and the above map induce Hecke-equivariant isomorphisms:
$$
\MSo{\bDsc}^{\ord} \ot \bArc / \p_k \bArc \cong \MSo{\bDo_k}^{\ord} \cong \MSo{\P_k^\vee}^{\ord}.
$$

\end{lemma}

\begin{proof}
The first isomorphism is given in \cite[Corollary 3.12]{Bellaiche} if $p$ is inverted.  Mimicking the argument there, we simply need to check that 
$H_0(\Gamma_0,\bDsc)^{\ord}$ vanishes.  Since 
$$
H_0(\Gamma_0,\bDsc)^{\ord} \cong \left((\bDsc)_{\Gamma_0}\right)^{\ord} = 
\left((\bDsc)^{\ord}\right)_{\Gamma_0},
$$
it suffices to see that $(\bDsc)^{\ord}$ vanishes.  (Here, we let $U_p$ act on $\bDsc$ by acting by $\sum_{a=0}^{p-1} \psmallmat{1}{a}{0}{p}$.) 

To this end, take $\mus \in (\bDsc)^{\ord}$, and write $\mus = \nus \big| U_p^n$ for some $n$.  Then
\begin{align*}
\mus(z^j) 
&= 
\sum_{a=0}^{p^n-1} (\nus \big|\psmallmat{1}{a}{0}{p^n}) (z^j)  \\
&= 
\sum_{a=0}^{p^n-1} \nus((a+p^nz)^j)\\
&\equiv 
\nus({\bf 1}) \sum_{a=0}^{p^n-1} a^j \mod{p^n}\\
&\equiv 0 \mod{p^{n-1}}.
\end{align*}
by Lemma \ref{lemma:ent} below.
Since this congruence holds for all $n$, we get our desired result.

For the second isomorphism, we note that this is proven in \cite[Theorem 5.4]{PS2} except that $p$ is inverted.  Mimicking the arguments there, but keeping everything integral, we need to show that $H_0(\Gamma_0,K)^{\ord} = 0$ where $K \subseteq \bDo_k$ are the distributions which vanish on $z^j$ for $0 \leq j \leq k$.  Arguing as above, it suffices to see that $K^{\ord} = 0$.  For $\mu \in K^{\ord}$, write $\mu = \nu \big| U_p^n$, and thus
\begin{align*}
\mu(z^j) 
&= 
\sum_{a=0}^{p^n-1} (\nu \big|\psmallmat{1}{a}{0}{p^n}) (z^j)  \\
&= 
\sum_{a=0}^{p^n-1} \nu((a+p^nz)^j)\\
&\equiv 0 \mod{p^{n}}
\end{align*}
as $\nu({\bf 1}) = 0$.
\end{proof}

\begin{lemma}
\label{lemma:ent}
For all $j\geq 0$ and $n \geq 1$, we have
$p^{n-1}$ divides $\displaystyle \sum_{a=0}^{p^n-1} a^j$.
\end{lemma}

\begin{proof}
We proceed by induction on $n$ with $n=1$ being vacuous.  Then, for $n>1$, we have
\begin{align*}
\sum_{a=0}^{p^n-1} a^j 
\equiv p \cdot \sum_{a=0}^{p^{n-1}-1} a^j  \pmod{p^{n-1}}.
\end{align*}
By induction, $p^{n-2}$ divides $\sum_{a=0}^{p^{n-1}-1} a^j$, and thus $p^{n-1}$ divides $\sum_{a=0}^{p^{n}-1} a^j$ as desired.
\end{proof}

\begin{lemma}
\label{lemma:NAK}
Let $X$ and $Y$ be $\bArc$-modules with $Y$ free over $\bArc$.  Assume there is a map $\alpha : X \to Y$ such that the induced map $X/\p_kX \to Y/\p_kY$ is an isomorphism for some $k$.  Then $\alpha$ is an isomorphism.
\end{lemma}

\begin{proof}
Let $Z$ be the $\bArc$-module defined by the exact sequence
$$
X \to Y \to Z \to 0.
$$
Thus we have an exact sequence,
$$
 X / \p_k X \to Y/\p_kY \to Z/\p_k Z \to 0.
$$
Since $X / \p_k X \cong Y/\p_kY$, we have $Z/\p_kZ = 0$.  But then $Z=0$ and $\alpha$ is surjective.  

Now let $W$ be the $\bArc$-module defined by the exact sequence
$$
0 \to W \to X \to Y \to 0.
$$
By the snake lemma, we then have an exact sequence
$$
Y[\p_k] \to W/\p_k W \to  X / \p_k X \to Y/\p_kY \to 0.
$$
Since $Y$ is free, $Y[\p_k] = 0$, and  since $X / \p_k X \cong Y/\p_kY$, we have $W/\p_kW=0$.  Thus, $W=0$ and $\alpha$ is an isomorphism.
\end{proof}

\begin{lemma}
\label{lemma:free}
We have $\MSo{\bDsc}^{\ord}$ is free over $\Lambda$ with finite rank.
\end{lemma}

\begin{proof}
We use the fact that if $Y$ is any $\Lambda$-module such that $Y/\p_kY$ is a free $\Zp$-module of finite rank, then $Y$ is free over $\Lambda$ with finite rank.  Then note that  by Lemma \ref{lemma:compOMS}, 
$$
\MSo{\bDsc}^{\ord} / \p_k \MSo{\bDsc}^{\ord} \cong \MSo{\P_k^\vee}^{\ord}
$$
which is indeed free over $\Zp$ with finite rank (as this last space is a classical space of modular symbols).
\end{proof}

\begin{thm}
\label{thm:comp}
The map $\alpha$ induces a Hecke-equivariant isomorphism
$$
\MSo{\M(\Zpx \times \Zp)}^{\ord} \otimes_{\tLambda} \bArc \to \MSo{\bDsc}^{\ord}.
$$
\end{thm}

\begin{proof}
If we choose $k \equiv m \mod{p-1}$ such that $0 \leq k \leq p-2$, then it is easy to see that $\widetilde{\P}_k = \P_k$.  This  theorem then follows from Lemma \ref{lemma:NAK}, 
 Lemma \ref{lemma:compGS}, and Lemma \ref{lemma:compOMS}.
\end{proof}

\begin{cor}
The characteristic polynomial of any Hecke operator acting on $\MSo{\bDsc}^{\ord}$ has coefficients which converge on all of $W_m$.
\end{cor}

\begin{proof}
This corollary follows immediately from Theorem \ref{thm:comp} as characteristic polynomials of Hecke operators on $\MSo{\M(\Zpx \times \Zp)}^{\ord}$ have this property. 
\end{proof}


\newpage
\section{Some data}
\label{DataAppendix}
\subsection{$q$-expansions}\label{DAq}\mbox{}\\

Full data for Example~\ref{ex:X011_qexp}: $p=11$, tame level $N=1$, branch $m=0$.

\begin{align*}
a_{2} &= 9 + 10p + 10p^{2} + 10p^{3} + 10p^{4} + 10p^{5} + 10p^{6} + 10p^{7} + 10p^{8} + 10p^{9} + 10p^{10}\\&+ \left(2p + p^{2} + 2p^{3} + 8p^{4} + 10p^{5} + 7p^{6} + 9p^{7} + 8p^{8} + 8p^{9} + 4p^{10}\right)k\\&+ \left(10p^{2} + 4p^{3} + 9p^{4} + 9p^{5} + 5p^{6} + 8p^{7} + 2p^{8} + 4p^{9} + 8p^{10}\right)k^{2}\\&+ \left(6p^{3} + 2p^{4} + 3p^{5} + 8p^{6} + 6p^{7} + 9p^{9} + p^{10}\right)k^{3}\\&+ \left(3p^{4} + 7p^{5} + 6p^{6} + 5p^{7} + 7p^{9} + 2p^{10}\right)k^{4}+ \left(8p^{5} + 6p^{6} + 3p^{8} + 3p^{10}\right)k^{5}\\&+ \left(6p^{6} + 10p^{7} + 8p^{8} + 6p^{9} + 8p^{10}\right)k^{6}+ \left(9p^{7} + 5p^{8} + 10p^{9} + 5p^{10}\right)k^{7}+ \left(10p^{9} + 6p^{10}\right)k^{8}\\&+ \left(7p^{9}\right)k^{9}+ \left(7p^{10}\right)k^{10}+ \left(9p^{10}\right)k^{11} + O(p^{11}, k^{12})
\\
\\
a_{3} &= 10 + 10p + 10p^{2} + 10p^{3} + 10p^{4} + 10p^{5} + 10p^{6} + 10p^{7} + 10p^{8} + 10p^{9} + 10p^{10}\\&+ \left(10p + 6p^{2} + p^{3} + 10p^{4} + 7p^{5} + 8p^{6} + 2p^{7} + 9p^{8} + 3p^{9} + 8p^{10}\right)k\\&+ \left(10p^{2} + 4p^{5} + 3p^{6} + 2p^{7} + 7p^{8} + 6p^{9} + 2p^{10}\right)k^{2}\\&+ \left(8p^{3} + 10p^{4} + 5p^{5} + 4p^{6} + 6p^{7} + 8p^{8} + 6p^{9} + 8p^{10}\right)k^{3}\\&+ \left(9p^{4} + 5p^{5} + 10p^{6} + 9p^{7} + 4p^{8} + 8p^{10}\right)k^{4}+ \left(p^{5} + p^{6} + 3p^{7} + 3p^{8} + 5p^{9} + 3p^{10}\right)k^{5}\\&+ \left(4p^{6} + 9p^{7} + 3p^{8} + 3p^{9} + 6p^{10}\right)k^{6}+ \left(5p^{7} + 8p^{8} + 9p^{9} + 3p^{10}\right)k^{7}\\&+ \left(4p^{8} + 10p^{9} + p^{10}\right)k^{8}+ \left(2p^{9} + 4p^{10}\right)k^{9}+ \left(3p^{10}\right)k^{10}+ \left(p^{10}\right)k^{11} + O(p^{11}, k^{12})
\\
\\
a_{5} &= 1+ \left(11 + 10p^{2} + p^{4} + 9p^{5} + 6p^{6} + 3p^{7} + 10p^{8} + 8p^{9} + 6p^{10}\right)k\\&+ \left(5p^{2} + 8p^{4} + 9p^{5} + 6p^{6} + 3p^{7} + 3p^{8} + 2p^{9} + 7p^{10}\right)k^{2}\\&+ \left(7p^{3} + p^{4} + 8p^{5} + 10p^{6} + 7p^{7} + 8p^{8} + p^{9} + 5p^{10}\right)k^{3}\\&+ \left(2p^{4} + 2p^{6} + 6p^{8} + 9p^{10}\right)k^{4}+ \left(9p^{5} + 10p^{6} + 5p^{7} + 10p^{8} + 4p^{9}\right)k^{5}\\&+ \left(5p^{6} + 4p^{7} + 4p^{9} + 6p^{10}\right)k^{6}+ \left(10p^{7} + 3p^{8} + 7p^{9} + 3p^{10}\right)k^{7}\\&+ \left(10p^{8} + 10p^{9} + 7p^{10}\right)k^{8}+ \left(6p^{9} + 9p^{10}\right)k^{9}+ \left(2p^{10}\right)k^{10}+ \left(10p^{10}\right)k^{11} + O(p^{11}, k^{12})
\\
\\
a_{7} &= 9 + 10p + 10p^{2} + 10p^{3} + 10p^{4} + 10p^{5} + 10p^{6} + 10p^{7} + 10p^{8} + 10p^{9} + 10p^{10}\\&+ \left(4p + 4p^{3} + p^{4} + 3p^{5} + 5p^{6} + 5p^{7} + 6p^{9} + 2p^{10}\right)k\\&+ \left(p^{2} + p^{3} + p^{4} + 2p^{5} + 8p^{6} + 9p^{7} + 4p^{8} + 8p^{9} + 7p^{10}\right)k^{2}\\&+ \left(8p^{3} + p^{4} + 6p^{5} + 2p^{6} + 4p^{7} + 3p^{8} + 5p^{9} + 3p^{10}\right)k^{3}\\&+ \left(7p^{4} + 4p^{5} + 6p^{6} + 8p^{7} + p^{8} + 9p^{9} + 7p^{10}\right)k^{4}+ \left(6p^{5} + 4p^{6} + 9p^{8} + 7p^{9} + 8p^{10}\right)k^{5}\\&+ \left(6p^{7} + 2p^{8} + 6p^{9} + 4p^{10}\right)k^{6}+ \left(10p^{7} + 5p^{8} + 6p^{9} + p^{10}\right)k^{7}\\&+ \left(5p^{8} + 9p^{9} + 2p^{10}\right)k^{8}+ \left(9p^{9} + 9p^{10}\right)k^{9}+ \left(2p^{10}\right)k^{10}+ \left(7p^{10}\right)k^{11} + O(p^{11}, k^{12})
\end{align*}
\vfill
\newpage
\begin{align*}
a_{11} &= 1+ \left(8p + 2p^{2} + 7p^{3} + p^{4} + 7p^{5} + p^{6} + 3p^{7} + 9p^{8} + 7p^{9} + p^{10}\right)k\\&+ \left(2p^{2} + 3p^{3} + 2p^{4} + 8p^{6} + 7p^{8}\right)k^{2}+ \left(2p^{4} + 4p^{5} + 6p^{6} + 3p^{7} + p^{8} + 7p^{9} + 4p^{10}\right)k^{3}\\&+ \left(2p^{4} + 4p^{5} + 8p^{6} + 4p^{7} + 2p^{8} + 10p^{9} + 6p^{10}\right)k^{4}+ \left(10p^{5} + p^{6} + 3p^{8} + 3p^{9}\right)k^{5}\\&+ \left(10p^{6} + 3p^{7} + 10p^{8} + 3p^{9} + 9p^{10}\right)k^{6}+ \left(6p^{7} + 2p^{8} + 3p^{9} + 2p^{10}\right)k^{7}\\&+ \left(10p^{8} + 8p^{9} + 5p^{10}\right)k^{8}+ \left(7p^{9} + 6p^{10}\right)k^{9}+ \left(5p^{10}\right)k^{10}+ \left(3p^{10}\right)k^{11} + O(p^{11}, k^{12})
\end{align*}

\medskip
Full data for Example~\ref{ex:15a1_qexp}: $p=5$, tame level $N=3$, branch $m=0$.
\medskip
\begin{align*}
a_{2} &= 4 + 4p + 4p^{2} + 4p^{3} + 4p^{4} + 4p^{5} + 4p^{6} + 4p^{7} + 4p^{8} + 4p^{9} + 4p^{10}\\&+ \left(5 + 3p^{2} + p^{3} + 2p^{4} + 3p^{5} + 3p^{6} + p^{7} + p^{8} + 3p^{10}\right)k\\&+ \left(2p^{2} + p^{4} + p^{5} + 4p^{7} + 4p^{8} + 4p^{9} + p^{10}\right)k^{2}\\&+ \left(3p^{3} + 3p^{4} + 4p^{5} + 2p^{6} + 2p^{7} + 3p^{8} + 2p^{9} + p^{10}\right)k^{3}\\&+ \left(2p^{4} + 3p^{5} + 3p^{7} + 2p^{8} + 2p^{10}\right)k^{4}+ \left(4p^{4} + p^{5} + p^{6} + 2p^{7} + 2p^{10}\right)k^{5}\\&+ \left(p^{5} + 3p^{6} + 4p^{7} + 2p^{8} + 4p^{9} + p^{10}\right)k^{6}+ \left(p^{6} + p^{7} + 3p^{8} + 3p^{9} + 3p^{10}\right)k^{7}\\&+ \left(2p^{7} + 2p^{8} + 4p^{9} + 2p^{10}\right)k^{8}+ \left(p^{8} + 4p^{9} + 2p^{10}\right)k^{9}+ \left(2p^{8} + p^{10}\right)k^{10}\\&+ \left(4p^{9}\right)k^{11} + O(p^{11}, k^{12})
\\
\\
a_{3} &= 4 + 4p + 4p^{2} + 4p^{3} + 4p^{4} + 4p^{5} + 4p^{6} + 4p^{7} + 4p^{8} + 4p^{9} + 4p^{10}\\&+ \left(3p + 4p^{3} + 3p^{6} + p^{7} + 3p^{8} + 4p^{9} + 2p^{10}\right)k\\&+ \left(3p^{2} + p^{3} + 4p^{6} + p^{7} + 2p^{8} + 2p^{10}\right)k^{2}+ \left(2p^{3} + 3p^{4} + p^{6} + 2p^{7} + 3p^{8} + 4p^{9} + 4p^{10}\right)k^{3}\\&+ \left(p^{4} + p^{5} + p^{6} + 3p^{7} + 2p^{8} + 3p^{9} + 3p^{10}\right)k^{4}+ \left(2p^{4} + p^{5} + 2p^{6} + 2p^{8} + 3p^{9} + 3p^{10}\right)k^{5}\\&+ \left(4p^{5} + p^{6} + 3p^{7} + 4p^{8} + 4p^{9} + p^{10}\right)k^{6}+ \left(4p^{6} + 2p^{7} + 4p^{9} + 3p^{10}\right)k^{7}+ \left(p^{7} + 4p^{9}\right)k^{8}\\&+ \left(3p^{8} + p^{9} + 4p^{10}\right)k^{9}+ \left(3p^{8} + 2p^{9} + 2p^{10}\right)k^{10}+ \left(p^{9}\right)k^{11} + O(p^{11}, k^{12})
\\
\\
a_{5} &= 1+ \left(4p + 4p^{2} + p^{3} + 4p^{4} + p^{7} + 4p^{10}\right)k\\&+ \left(p^{2} + 4p^{3} + p^{4} + 2p^{5} + 2p^{6} + 3p^{7} + p^{8} + 3p^{9} + 3p^{10}\right)k^{2}\\&+ \left(p^{4} + 3p^{6} + 4p^{7} + 2p^{8} + 4p^{10}\right)k^{3}+ \left(p^{5} + 3p^{6} + 3p^{7} + 2p^{8} + 2p^{9} + 3p^{10}\right)k^{4}\\&+ \left(p^{4} + 4p^{5} + 2p^{6} + 3p^{7} + 3p^{10}\right)k^{5}+ \left(3p^{5} + p^{6} + 3p^{7} + 3p^{8} + 3p^{10}\right)k^{6}\\&+ \left(p^{8} + 3p^{9} + 3p^{10}\right)k^{7}+ \left(3p^{8} + 3p^{9}\right)k^{8}+ \left(4p^{8} + p^{9} + p^{10}\right)k^{9}\\&+ \left(p^{8} + p^{9} + 3p^{10}\right)k^{10}+ \left(2p^{10}\right)k^{11} + O(p^{11}, k^{12})
\end{align*}
\vfill
\newpage
\begin{align*}
a_{7} &= \left(3p + p^{2} + 4p^{3} + 4p^{4} + 4p^{5} + 4p^{7} + 4p^{8} + 3p^{9} + 4p^{10}\right)k\\&+ \left(2p^{2} + p^{3} + p^{4} + 2p^{5} + 3p^{6} + 3p^{7} + 3p^{8} + 4p^{9} + p^{10}\right)k^{2}\\&+ \left(3p^{4} + 3p^{6} + 3p^{7} + 3p^{8} + p^{9} + 2p^{10}\right)k^{3}+ \left(4p^{5} + 3p^{6} + p^{7} + 4p^{8} + p^{9} + 3p^{10}\right)k^{4}\\&+ \left(2p^{4} + 4p^{5} + p^{6} + 4p^{7} + 2p^{8} + 2p^{9} + 2p^{10}\right)k^{5}+ \left(p^{5} + 4p^{6} + 2p^{7} + 4p^{9} + 4p^{10}\right)k^{6}\\&+ \left(2p^{7} + 3p^{8} + 4p^{9} + 3p^{10}\right)k^{7}+ \left(3p^{8} + 2p^{10}\right)k^{8}+ \left(3p^{8} + 2p^{9} + p^{10}\right)k^{9}\\&+ \left(2p^{8} + 3p^{9} + 4p^{10}\right)k^{10}+ \left(2p^{10}\right)k^{11} + O(p^{11}, k^{12})
\\
\\
a_{11} &= 1 + 4p + 4p^{2} + 4p^{3} + 4p^{4} + 4p^{5} + 4p^{6} + 4p^{7} + 4p^{8} + 4p^{9} + 4p^{10}\\&+ \left(2p + p^{2} + 2p^{5} + 4p^{7} + p^{8} + 2p^{10}\right)k+ \left(2p^{2} + 2p^{4} + 2p^{6} + 3p^{8}\right)k^{2}\\&+ \left(2p^{3} + 4p^{4} + 4p^{6} + p^{8} + 3p^{9} + 2p^{10}\right)k^{3}+ \left(2p^{4} + 3p^{5} + 2p^{8} + 4p^{9} + 3p^{10}\right)k^{4}\\&+ \left(3p^{4} + 2p^{5} + 2p^{7} + 3p^{8} + p^{9} + 4p^{10}\right)k^{5}+ \left(p^{5} + p^{6} + 3p^{7} + p^{8} + 4p^{9}\right)k^{6}\\&+ \left(4p^{6} + p^{7} + 4p^{8} + 2p^{9}\right)k^{7}+ \left(2p^{7} + 3p^{8} + 3p^{10}\right)k^{8}+ \left(2p^{8} + 3p^{9}\right)k^{9}\\&+ \left(2p^{8} + 2p^{9} + 4p^{10}\right)k^{10}+ \left(p^{9} + 4p^{10}\right)k^{11} + O(p^{11}, k^{12})
\end{align*}

\medskip
Full data for Example~\ref{ex:19a1_qexp}: $p=5$, tame level $N=19$, branch $m=0$.
\medskip
\begin{align*}
a_{2} &= \left(3p + 4p^{2} + 3p^{3} + 2p^{6}\right)k+ \left(3p^{2} + p^{3} + 3p^{4} + 4p^{5} + 2p^{6}\right)k^{2}+ \left(p^{3} + p^{4} + 2p^{6}\right)k^{3}\\&+ \left(3p^{4} + 2p^{5} + 3p^{6}\right)k^{4}+ \left(2p^{4} + 3p^{6}\right)k^{5}+ \left(4p^{5}\right)k^{6}+ \left(2p^{6}\right)k^{7} + O(p^{7}, k^{8})
\\
\\
a_{3} &= 3 + 4p + 4p^{2} + 4p^{3} + 4p^{4} + 4p^{5} + 4p^{6}+ \left(2p + 2p^{2} + 2p^{3} + 3p^{4} + 2p^{5}\right)k\\&+ \left(3p^{3} + p^{4} + 4p^{6}\right)k^{2}+ \left(p^{3} + 2p^{4} + 2p^{6}\right)k^{3}+ \left(2p^{4} + 2p^{5} + 2p^{6}\right)k^{4}\\&+ \left(3p^{4} + 4p^{5}\right)k^{5}+ \left(2p^{6}\right)k^{6}+ \left(2p^{6}\right)k^{7} + O(p^{7}, k^{8})
\\
\\
a_{5} &= 3 + 3p + 3p^{2} + p^{3} + 4p^{6}+ \left(5 + 2p^{2} + 2p^{4} + p^{5} + 2p^{6}\right)k\\&+ \left(p^{2} + 2p^{3} + 2p^{4} + 4p^{5}\right)k^{2}+ \left(2p^{4} + 3p^{5} + 3p^{6}\right)k^{3}+ \left(3p^{4} + 3p^{5} + 3p^{6}\right)k^{4}\\&+ \left(4p^{4} + p^{5}\right)k^{5}+ \left(3p^{5} + 2p^{6}\right)k^{6}+ O(p^{7}, k^{8})
\\
\\
a_{7} &= 4 + 4p + 4p^{2} + 4p^{3} + 4p^{4} + 4p^{5} + 4p^{6}+ \left(3p + 4p^{2} + 4p^{3} + p^{5} + p^{6}\right)k\\&+ \left(3p^{3} + 4p^{4} + 3p^{5} + 4p^{6}\right)k^{2}+ \left(3p^{3} + 4p^{6}\right)k^{3}+ \left(p^{4} + p^{5} + 2p^{6}\right)k^{4}\\&+ \left(2p^{4} + p^{5} + 4p^{6}\right)k^{5}+ \left(2p^{6}\right)k^{6}+ \left(p^{6}\right)k^{7} + O(p^{7}, k^{8})
\\
\\
a_{11} &= 3+ \left(2p + 3p^{3} + 4p^{4} + p^{5} + 4p^{6}\right)k+ \left(4p^{2} + p^{3} + 3p^{4} + 2p^{5} + p^{6}\right)k^{2}\\&+ \left(p^{3} + 2p^{4} + 3p^{6}\right)k^{3}+ \left(3p^{4} + p^{6}\right)k^{4}+ \left(3p^{4} + p^{5} + 3p^{6}\right)k^{5}\\&+ \left(2p^{5}\right)k^{6}+ \left(2p^{6}\right)k^{7} + O(p^{7}, k^{8})
\end{align*}
\vfill
\newpage
\begin{align*}
a_{13} &= 1 + 4p + 4p^{2} + 4p^{3} + 4p^{4} + 4p^{5} + 4p^{6}+ \left(3p^{2} + 2p^{3} + 2p^{4} + 2p^{5} + 4p^{6}\right)k\\&+ \left(3p^{2} + 3p^{3} + 4p^{4} + p^{5} + 4p^{6}\right)k^{2}+ \left(4p^{3} + 2p^{4} + 4p^{5} + 3p^{6}\right)k^{3}\\&+ \left(p^{5} + 2p^{6}\right)k^{4}+ \left(4p^{5} + 2p^{6}\right)k^{5}+ \left(4p^{5} + 4p^{6}\right)k^{6}+ \left(3p^{6}\right)k^{7} + O(p^{7}, k^{8})
\\
\\
a_{17} &= 2 + 4p + 4p^{2} + 4p^{3} + 4p^{4} + 4p^{5} + 4p^{6}+ \left(3p + 3p^{2} + 2p^{3} + 2p^{4} + 2p^{5} + 2p^{6}\right)k\\&+ \left(p^{2} + 3p^{4} + p^{5}\right)k^{2}+ \left(2p^{4}\right)k^{3}+ \left(4p^{5} + 2p^{6}\right)k^{4}\\&+ \left(2p^{4} + p^{5} + 2p^{6}\right)k^{5}+ \left(3p^{5}\right)k^{6} + O(p^{7}, k^{8})
\\
\\
a_{19} &= 1+ \left(3p + 2p^{4} + 2p^{5} + 3p^{6}\right)k+ \left(2p^{2} + 3p^{3} + 2p^{4} + 3p^{5} + 4p^{6}\right)k^{2}\\&+ \left(2p^{3} + 3p^{4} + 2p^{5} + p^{6}\right)k^{3}+ \left(4p^{4} + 3p^{5}\right)k^{4}+ \left(2p^{4} + p^{5} + 2p^{6}\right)k^{5}\\&+ \left(p^{5} + 3p^{6}\right)k^{6}+ \left(4p^{6}\right)k^{7} + O(p^{7}, k^{8})
\end{align*}

\medskip
Full data for Example~\ref{ex:N95p5_qexp}: $p=5$, tame level $N=19$, branch $m=0$.
\medskip
\begin{align*}
a_{2} &= 3 + 5 + p^{2} + 4p^{3} + 2p^{4} + p^{5}+ \left(2p + 2p^{2} + 3p^{3} + 3p^{4} + 2p^{5}\right)k+ \left(4p^{2} + 4p^{5}\right)k^{2}\\&+ \left(p^{3} + 3p^{4} + 2p^{5}\right)k^{3}+ \left(3p^{4}\right)k^{4}+ \left(3p^{4} + p^{5}\right)k^{5}+ \left(2p^{5}\right)k^{6} + O(p^{6}, k^{7})
\\
\\
a_{3} &= 4 + 2p + p^{2} + 2p^{3} + 3p^{4} + p^{5}+ \left(3p + 2p^{2} + 3p^{3} + 3p^{4}\right)k+ \left(3p^{2} + 2p^{3} + 4p^{4}\right)k^{2}\\&+ \left(3p^{3} + 4p^{5}\right)k^{3}+ \left(3p^{4} + 2p^{5}\right)k^{4}+ \left(2p^{4} + 2p^{5}\right)k^{5}+ \left(4p^{5}\right)k^{6} + O(p^{6}, k^{7})
\\
\\
a_{5} &= 1+ \left(2p^{2} + 3p^{5}\right)k+ \left(p^{2} + 2p^{4} + p^{5}\right)k^{2}+ \left(2p^{4} + p^{5}\right)k^{4}\\&+ \left(2p^{5}\right)k^{5}+ \left(3p^{5}\right)k^{6} + O(p^{6}, k^{7})
\\
\\
a_{7} &= 3 + 5 + 4p^{2} + p^{3} + 2p^{4} + 3p^{5}+ \left(2p^{2} + 2p^{3} + 3p^{4} + p^{5}\right)k+ \left(p^{2} + 3p^{3} + p^{4} + 3p^{5}\right)k^{2}\\&+ \left(4p^{3} + 4p^{4} + 4p^{5}\right)k^{3}+ \left(2p^{5}\right)k^{4}+ \left(p^{5}\right)k^{5}+ \left(3p^{5}\right)k^{6} + O(p^{6}, k^{7})
\\
\\
a_{11} &= 2 + 5 + 2p^{2} + p^{3} + 4p^{4} + p^{5}+ \left(4p + 4p^{2} + p^{3} + p^{4} + 2p^{5}\right)k\\&+ \left(4p^{2} + 2p^{3} + p^{4} + p^{5}\right)k^{2}+ \left(2p^{3} + 2p^{4} + p^{5}\right)k^{3}+ \left(3p^{4} + 2p^{5}\right)k^{4}+ \left(p^{4} + 4p^{5}\right)k^{5}\\&+ \left(2p^{5}\right)k^{6} + O(p^{6}, k^{7})
\end{align*}

\newpage
Full data for Example~\ref{ex:big_ex_qexp}: $p=11$, tame level $N=31$, branch $m=0$.
\medskip
\begin{align*}
a_{2} &= 4 + 3p + 3p^{3} + 5p^{4} + 6p^{5} + 2p^{6} + p^{7}+ \left(9p^{2} + 9p^{3} + 8p^{4} + 6p^{5} + 5p^{6} + 4p^{7}\right)k\\
&+ \left(5p^{2} + 3p^{3} + 2p^{4} + 10p^{6} + 8p^{7}\right)k^{2}+ \left(5p^{4} + 2p^{5} + 4p^{6} + 7p^{7}\right)k^{3}\\
&+ \left(8p^{4} + 9p^{5} + 3p^{6}\right)k^{4}+ \left(4p^{5} + 3p^{6} + 9p^{7}\right)k^{5}\\
&+ \left(6p^{6} + 7p^{7}\right)k^{6}+ \left(10p^{7}\right)k^{7} + O(p^{8}, k^{8})
\\
\\
a_{3} &= 3 + 4p + 10p^{2} + 4p^{3} + 9p^{5} + 5p^{6} + 8p^{7}+ \left(8p + 7p^{2} + 4p^{3} + 6p^{4} + 8p^{5} + 4p^{6} + 7p^{7}\right)k\\
&+ \left(6p^{2} + 5p^{3} + 7p^{4} + 10p^{5} + 10p^{7}\right)k^{2}+ \left(p^{4} + 10p^{5} + 7p^{6} + 6p^{7}\right)k^{3}\\
&+ \left(8p^{4} + 4p^{5} + 2p^{6} + p^{7}\right)k^{4}+ \left(8p^{5} + 5p^{6} + 3p^{7}\right)k^{5}\\
&+ \left(7p^{7}\right)k^{6}+ \left(6p^{7}\right)k^{7} + O(p^{8}, k^{8})
\\
\\
a_{5} &= 1+ \left(5p + 7p^{2} + p^{3} + 10p^{4} + 6p^{5} + 7p^{6} + 7p^{7}\right)k\\
&+ \left(7p^{2} + 10p^{3} + 9p^{4} + 7p^{5} + 6p^{6} + 8p^{7}\right)k^{2}+ \left(8p^{3} + 8p^{4} + 4p^{5} + 4p^{6} + 6p^{7}\right)k^{3}\\
&+ \left(7p^{4} + 6p^{5}\right)k^{4}+ \left(9p^{5} + 4p^{6} + 2p^{7}\right)k^{5}\\
&+ \left(5p^{6} + 9p^{7}\right)k^{6}+ \left(6p^{7}\right)k^{7} + O(p^{8}, k^{8})
\\
\\
a_{7} &= 5 + 6p + 6p^{3} + 10p^{4} + p^{5} + 5p^{6} + 2p^{7}+ \left(4p + 7p^{2} + 8p^{3} + 6p^{4} + 10p^{5} + 5p^{6} + 5p^{7}\right)k\\
&+ \left(7p^{2} + p^{3} + 3p^{4} + 2p^{5} + 10p^{6} + 6p^{7}\right)k^{2}+ \left(3p^{4} + 4p^{5} + 7p^{6} + 9p^{7}\right)k^{3}\\
&+ \left(5p^{4} + 9p^{6} + p^{7}\right)k^{4}+ \left(4p^{5} + 3p^{6} + 2p^{7}\right)k^{5}\\
&+ \left(3p^{6} + 5p^{7}\right)k^{6}+ \left(9p^{7}\right)k^{7} + O(p^{8}, k^{8})
\\
\\
a_{11} &= 2 + 5p + 9p^{2} + 8p^{3} + 8p^{4} + 6p^{5} + 3p^{6} + 2p^{7}+ \left(10p + 6p^{2} + 8p^{3} + p^{4} + p^{5} + p^{6} + 8p^{7}\right)k\\
&+ \left(2p^{2} + 6p^{3} + 4p^{4} + 2p^{5} + 8p^{7}\right)k^{2}+ \left(10p^{3} + 7p^{4} + p^{5} + 9p^{6} + 9p^{7}\right)k^{3}\\
&+ \left(3p^{4} + 2p^{5} + 4p^{6} + 2p^{7}\right)k^{4}+ \left(10p^{5} + 5p^{6} + 8p^{7}\right)k^{5}\\
&+ \left(9p^{6} + 9p^{7}\right)k^{6}+ \left(10p^{7}\right)k^{7} + O(p^{8}, k^{8})
\end{align*}

\newpage
\subsection{$L$-invariants}\label{DAL}\mbox{}\\

\medskip
Full data for Example~\ref{ex:sym2family_Linvar}:
\medskip
\begin{align*}
		\L_{11}(\ad^0\!\F_{11})&=6p + 5p^{2} + 7p^{3} + 7p^{4} + 7p^{5} + 7p^{6} + 4p^{7} + 3p^{8} + 6p^{9} + 7p^{10}\\&+ \left(10p^{2} + 7p^{3} + 9p^{4} + 9p^{5} + 10p^{6} + p^{8} + 9p^{9}\right)k\\&+ \left(7p^{3} + 7p^{4} + 9p^{5} + p^{6} + 9p^{7} + 5p^{8} + p^{10}\right)k^{2}\\&+ \left(7p^{4} + 9p^{5} + 4p^{7} + 4p^{8} + p^{10}\right)k^{3}+ \left(5p^{5} + 7p^{6} + 6p^{8} + 9p^{10}\right)k^{4}\\&+ \left(10p^{6} + 8p^{7} + 4p^{8} + 9p^{9} + 7p^{10}\right)k^{5}+ \left(4p^{7} + 8p^{8} + 4p^{10}\right)k^{6}\\&+ \left(8p^{8} + 6p^{10}\right)k^{7}+ \left(5p^{9} + 8p^{10}\right)k^{8}+ \left(4p^{10}\right)k^{9} + O(p^{11}, k^{10}).
\\
\\
		\L_{5}(\ad^0\!\F_{15})&=2p + p^{3} + p^{4} + 3p^{5} + 4p^{6} + 2p^{7} + 4p^{8} + 4p^{9} + p^{10}\\&+ \left(3p^{2} + 3p^{3} + 4p^{4} + 2p^{5} + 4p^{6} + 3p^{7} + 2p^{8} + 3p^{9} + 4p^{10}\right)k\\&+ \left(p^{3} + 4p^{4} + 2p^{5} + 4p^{7} + 3p^{8} + 2p^{9} + 4p^{10}\right)k^{2}\\&+ \left(3p^{4} + 3p^{5} + 2p^{6} + 3p^{9} + 2p^{10}\right)k^{3}+ \left(2p^{6} + 4p^{8} + 3p^{9} + 3p^{10} + 2p^{11}\right)k^{4}\\&+ \left(2p^{5} + 4p^{6} + 2p^{8} + 3p^{9} + 4p^{10}\right)k^{5}+ \left(3p^{6} + 2p^{8} + p^{9} + 3p^{10}\right)k^{6}\\&+ \left(p^{7} + 3p^{8} + 4p^{9} + 4p^{10}\right)k^{7}+ \left(p^{9} + 2p^{10}\right)k^{8}+ \left(3p^{9} + 3p^{11}\right)k^{9} + O(p^{11}, k^{10}).
\\
\\
		\L_{5}(\ad^0\!\F_{19})&=p + 4p^{2} + 2p^{3} + 4p^{4} + p^{5} + 4p^{6}+ \left(3p^{4} + 3p^{5} + 3p^{6}\right)k+ \left(3p^{3} + 2p^{4} + p^{6}\right)k^{2}\\&+ \left(p^{4} + 4p^{5} + 4p^{6}\right)k^{3}+ \left(2p^{6}\right)k^{4} + O(p^{7}, k^{5}).
\\
\\
		\L_{5}(\ad^0\!\F_{95})&=p^{2} + 4p^{3} + 4p^{4} + 3p^{5}+ \left(p^{2} + 4p^{3} + 4p^{4}\right)k\\&+ \left(2p^{4} + 2p^{5}\right)k^{2}+ \left(3p^{4} + 4p^{5}\right)k^{3} + O(p^{6}, k^{4}).
	\end{align*}

\vfill
\newpage
\subsection{Two-variable $p$-adic $L$-functions}\label{DApL}\mbox{}\\

\medskip
Full data for Example~\ref{ex:X011_padicL}:
\medskip
\begin{align*}
L_{11}(\F_{11},s,k)&=(p + 5p^2 + 9p^3 + 9p^4 + 9p^5 + 5p^6 + 8p^7 )k\\&
 + (9p + 3p^3 + 2p^4 + 2p^5 + 10p^6 + 4p^7 )s\\&
  + (6p^2 + 10p^3 + 10p^4 + 6p^6 + 2p^7 )k^2 + (10p^2 + 9p^5 + 9p^6 + 5p^7 )ks\\&
   + (2p^3 + 6p^4 + 2p^5 + 2p^6 + 9p^7 )k^3 + (2p^3 + 7p^4 + 3p^5 + 4p^6 + 3p^7 )k^2s\\&
    + (4p^3 + 7p^4 + 6p^5 + 6p^6 )ks^2 + (p^3 + 6p^4 + 6p^5 + 6p^6 + 10p^7 )s^3\\&
     + (6p^4 + 9p^5 + p^6 )k^4 + (3p^5 + 9p^6 + 6p^7 )k^3s\\&
      + (8p^4 + 10p^5 + 4p^6 )k^2s^2 + (2p^4 + 4p^6 + 3p^7 )ks^3\\&
       + (10p^5 + 7p^6 + 9p^7 )k^5 + (4p^5 + 10p^7 )k^4s + (8p^5 + 2p^6 + 9p^7 )k^3s^2\\&
        + (3p^5 + 3p^6 + 9p^7 )k^2s^3 + (4p^5 + 3p^6 + 4p^7 )ks^4 + (5p^5 + 7p^6 + 2p^7 )s^5\\&
         + (9p^6 + 7p^7 )k^6 + (8p^6 + 4p^7 )k^5s + (7p^6 + p^7 )k^4s^2 + (9p^6 + 9p^7 )k^3s^3\\&
          + (7p^6 + 5p^7 )k^2s^4 + (6p^6 + 6p^7 )ks^5\\&
           + (6p^7 )k^7 + (2p^7 )k^6s + (9p^7 )k^5s^2 + (10p^7 )k^4s^3 \\&
           + (7p^7 )k^3s^4 + (10p^7 )k^2s^5 + (7p^7 )ks^6 + (9p^7 )s^7+O(p^8,(k,s)^8)\\&
           =(k-2s)(p + 5p^2 + 9p^3 + 9p^4 + 9p^5 + 5p^6 + 8p^7\\&
             + (6p^2 + 10p^3 + 10p^4 + 6p^6 + 2p^7 )k + (2p^3 + 6p^4 + 2p^5 + 2p^6 + 9p^7 )k^2\\&
              + (6p^3 + 8p^4 + 8p^5 + 8p^6 + 10p^7 )ks + (5p^3 + 2p^4 + 2p^5 + 2p^6 )s^2\\&
               + (6p^4 + 9p^5 + p^6 )k^3 + (p^4 + 2p^6 + 7p^7 )k^2s + (10p^4 + 10p^5 + 8p^6 + 3p^7 )ks^2\\&
                + (10p^5 + 7p^6 + 9p^7 )k^4 + (2p^5 + 5p^6 + 7p^7 )k^3s + (p^5 + 2p^6 + 2p^7 )k^2s^2\\&
                 + (5p^5 + 7p^6 + 2p^7 )ks^3 + (3p^5 + 7p^6 + 9p^7 )s^4\\&
                  + (9p^6 + 7p^7 )k^5 + (4p^6 + 9p^7 )k^4s + (4p^6 + 9p^7 )k^3s^2\\&
                   + (6p^6 + 6p^7 )k^2s^3 + (8p^6 + 7p^7 )ks^4 + (6p^7 )k^6 + (3p^7 )k^5s + (4p^7 )k^4s^2\\&
                   + (7p^7 )k^3s^3 + (10p^7 )k^2s^4 + (8p^7 )ks^5 + (p^7 )s^6+O(p^8,(k,s)^7)
\end{align*}
\vfill
\newpage
Full data for Example~\ref{ex:37a_padicL}:
\medskip
\begin{align*}
L_{5}(37a,s,k)&=(p + 2p^2 + 4p^3 + 2p^4 + p^6 )k + (3p + p^3 + 4p^4 + 3p^5 + 2p^6 )s\\&
 + (3p^2 + 4p^3 + 3p^4 + 3p^5 + p^6 )k^2 + (p^3 + 4p^4 + 3p^5 )ks + (3p^2 + 4p^3 + 2p^5 + p^6 )s^2\\&
  + (p^3 + 2p^6 )k^3 + (3p^3 + p^4 + p^5 + p^6 )k^2s\\&
   + (p^3 + 4p^4 + p^6 )ks^2 + (3p^3 + 3p^4 + 2p^5 + p^6 )s^3\\&
    + (4p^4 + p^5 + 2p^6 )k^4 + (2p^4 + p^5 + 3p^6 )k^3s + (2p^4 + p^5 )k^2s^2\\& + (3p^5 + 3p^6 )ks^3 + (2p^4 + 3p^5 + 2p^6 )s^4 \\&+ (4p^4 + 4p^5 + p^6 )k^5 + (2p^5 + 4p^6 )k^4s + (3p^5 + 3p^6 )k^3s^2\\&
     + (4p^5 + 3p^6 )k^2s^3 + (2p^5 + 2p^6 )ks^4 + (2p^4 + 3p^6 )s^5\\&
      + (4p^5 + 3p^6 )k^6 + (p^6 )k^5s + (2p^6 )k^4s^2 + (2p^6 )k^3s^3\\&
       + (p^6 )k^2s^4 + (p^6 )ks^5 + (4p^5 )s^6+O(p^7,(k,s)^7)\\&
       =(k-2s)(p + 2p^2 + 4p^3 + 2p^4 + p^6 + (3p^2 + 4p^3 + 3p^4 + 3p^5 + p^6)k\\&
        + (p^2 + 2p^4 + p^5 + 4p^6)s + (p^3 + 2p^6)k^2 + (2p^4 + p^5)ks + (p^3 + 3p^4 + 3p^5 + p^6)s^2\\&
         + (4p^4 + p^5 + 2p^6)k^3 + (3p^6)k^2s + (2p^4 + p^5 + p^6)ks^2 + (4p^4 + p^6)s^3\\&
          + (4p^4 + 4p^5 + p^6)k^4 + (3p^4 + p^5 + 3p^6)k^3s + (p^4 + p^5)k^2s^2\\&
           + (2p^4 + p^5 + 4p^6)ks^3 + (4p^4 + 4p^5)s^4\\&
            + (4p^5 + 3p^6)k^5 + (3p^5 + 3p^6)k^4s + (p^5 + 4p^6)k^3s^2\\&
             + (2p^5)k^2s^3 + (4p^5 + p^6)ks^4 + (3p^5 + 4p^6)s^5)+O(p^7,(k,s)^7)
\end{align*}

\medskip
Full data for Example~\ref{ex:91b_padicL}:
\medskip
\begin{align*}
L_{7}(91b1,s,k)&=(p^2 + 6p^3 + 4p^4)k^2 + (2p^2 + p^3)ks + (5p^2 + 5p^3 + 6p^4)s^2\\&
 + (p^3 + 2p^4)k^3 + (4p^4)k^2s + (3p^3 + 6p^4)ks^2 + (5p^3 + 3p^4)s^3\\&
  + (2p^4)k^3s + (6p^4)k^2s^2 + (p^4)ks^3 + (3p^4)s^4+O(p^5,(k,s)^5)
\end{align*}

\end{appendix}

\bibliographystyle{amsalpha}
\bibliography{OMS_families}{}

\providecommand{\bysame}{\leavevmode\hbox to3em{\hrulefill}\thinspace}
\providecommand{\MR}{\relax\ifhmode\unskip\space\fi MR }
\providecommand{\MRhref}[2]{%
  \href{http://www.ams.org/mathscinet-getitem?mr=#1}{#2}
}
\providecommand{\href}[2]{#2}
\begin{thebibliography}{BSDGP96}

\bibitem[AIP15]{AIP}
Fabrizio Andreatta, Adrian Iovita, and Vincent Pilloni, \emph{{$p$}-adic
  families of {S}iegel modular cuspforms}, Ann. of Math. (2) \textbf{181}
  (2015), no.~2, 623--697. \MR{3275848}

\bibitem[AIS14]{AIS}
Fabrizio Andreatta, Adrian Iovita, and Glenn Stevens, \emph{Overconvergent
  modular sheaves and modular forms for {${\bf GL}_{2/F}$}}, Israel J. Math.
  \textbf{201} (2014), no.~1, 299--359. \MR{3265287}

\bibitem[AS86]{AshStevens-Duke}
Avner Ash and Glenn Stevens, \emph{Modular forms in characteristic {$l$} and
  special values of their {$L$}-functions}, Duke Math. J. \textbf{53} (1986),
  no.~3, 849--868. \MR{860675}

\bibitem[AS08]{AshStevens}
\bysame, \emph{{$p$}-adic deformations of arithmetic cohomology}, 2008.

\bibitem[BD15]{BellaicheDasgupta}
Jo{\"e}l Bella{\"{\i}}che and Samit Dasgupta, \emph{The {$p$}-adic
  {$L$}-functions of evil {E}isenstein series}, Compos. Math. \textbf{151}
  (2015), no.~6, 999--1040. \MR{3357177}

\bibitem[Bel12]{Bellaiche}
Jo{\"e}l Bella{\"{\i}}che, \emph{Critical {$p$}-adic {$L$}-functions}, Invent.
  Math. \textbf{189} (2012), no.~1, 1--60. \MR{2929082}

\bibitem[BGR84]{BGR}
S.~Bosch, U.~G{\"u}ntzer, and R.~Remmert, \emph{Non-archimedean analysis},
  Grundlehren der Mathematischen Wissenschaften [Fundamental Principles of
  Mathematical Sciences], vol. 261, Springer-Verlag, Berlin, 1984, A systematic
  approach to rigid analytic geometry. \MR{746961 (86b:32031)}

\bibitem[BP]{BP}
Jo\"el Bella\"iche and Robert Pollack, \emph{On $\mu$-invariants and
  congruences with {E}isenstein series}.

\bibitem[BSDGP96]{STE}
Katia Barr{\'e}-Sirieix, Guy Diaz, Fran{\c{c}}ois Gramain, and Georges
  Philibert, \emph{Une preuve de la conjecture de {M}ahler--{M}anin}, Invent.
  Math. \textbf{124} (1996), no.~1--3, 1--9. \MR{1369409 (96j:11103)}

\bibitem[Buz06]{Buzzardnote}
Kevin Buzzard, \emph{Examples of hida families}, available at
  http://www2.imperial.ac.uk/{\textasciitilde}buzzard/maths/research/notes/,
  2006.

\bibitem[CM98]{ColemanMazur}
Robert Coleman and Barry Mazur, \emph{The eigencurve}, Galois representations
  in arithmetic algebraic geometry ({D}urham, 1996), London Math. Soc. Lecture
  Note Ser., vol. 254, Cambridge Univ. Press, Cambridge, 1998, pp.~1--113.

\bibitem[Col97]{Coleman}
Robert~F. Coleman, \emph{{$p$}-adic {B}anach spaces and families of modular
  forms}, Invent. Math. \textbf{127} (1997), no.~3, 417--479. \MR{1431135
  (98b:11047)}

\bibitem[Das14]{samit}
Samit Dasgupta, \emph{Factorization of $p$-adic {R}ankin ${L}$-series}, to
  appear in Invent.\ Math. Available at
  http://people.ucsc.edu/{\textasciitilde}sdasgup2/, 2014.

\bibitem[Dev13]{sage}
The~Sage Developers, \emph{{S}age {M}athematics {S}oftware ({V}ersion 5.11)},
  2013, {\tt http://www.sagemath.org}.

\bibitem[DP06]{DarmonPollack}
Henri Darmon and Robert Pollack, \emph{Efficient calculation of
  {S}tark--{H}eegner points via overconvergent modular symbols}, Israel J.
  Math. \textbf{153} (2006), 319--354. \MR{2254648 (2007k:11077)}

\bibitem[Eme06]{Emerton}
Matthew Emerton, \emph{On the interpolation of systems of eigenvalues attached
  to automorphic {H}ecke eigenforms}, Invent. Math. \textbf{164} (2006), no.~1,
  1--84. \MR{2207783 (2007k:22018)}

\bibitem[FG79]{FG}
Bruce Ferrero and Ralph Greenberg, \emph{On the behavior of {$p$}-adic
  {$L$}-functions at {$s=0$}}, Invent. Math. \textbf{50} (1978/79), no.~1,
  91--102. \MR{516606 (80f:12016)}

\bibitem[Gre94]{G94}
Ralph Greenberg, \emph{Trivial zeros of {$p$}-adic {$L$}-functions}, {$p$}-adic
  monodromy and the {B}irch and {S}winnerton-{D}yer conjecture ({B}oston, {MA},
  1991), Contemp. Math., vol. 165, Amer. Math. Soc., Providence, RI, 1994,
  pp.~149--174. \MR{1279608 (95h:11063)}

\bibitem[GS93]{GS}
Ralph Greenberg and Glenn Stevens, \emph{{$p$}-adic {$L$}-functions and
  {$p$}-adic periods of modular forms}, Invent. Math. \textbf{111} (1993),
  no.~2, 407--447. \MR{1198816 (93m:11054)}

\bibitem[GS94]{GS2}
\bysame, \emph{On the conjecture of {M}azur, {T}ate, and {T}eitelbaum},
  {$p$}-adic monodromy and the {B}irch and {S}winnerton-{D}yer conjecture
  ({B}oston, {MA}, 1991), Contemp. Math., vol. 165, Amer. Math. Soc.,
  Providence, RI, 1994, pp.~183--211.

\bibitem[Har09]{Harron-thesis}
Robert~William Harron, \emph{L-invariants of low symmetric powers of modular
  forms and {H}ida deformations}, ProQuest LLC, Ann Arbor, MI, 2009, Thesis
  (Ph.D.)--Princeton University. \MR{2713827}

\bibitem[Hid86a]{HidaGalRep}
Haruzo Hida, \emph{Galois representations into {${\rm GL}_2({\bf Z}_p[[X]])$}
  attached to ordinary cusp forms}, Invent. Math. \textbf{85} (1986), no.~3,
  545--613. \MR{848685 (87k:11049)}

\bibitem[Hid86b]{HidaIwasawaModules}
\bysame, \emph{Iwasawa modules attached to congruences of cusp forms}, Ann.
  Sci. \'Ecole Norm. Sup. (4) \textbf{19} (1986), no.~2, 231--273.

\bibitem[Hid04]{Hida-Linv}
\bysame, \emph{Greenberg's {$\mathcal L$}-invariants of adjoint square {G}alois
  representations}, Int. Math. Res. Not. (2004), no.~59, 3177--3189.
  \MR{2097038 (2005f:11108)}

\bibitem[Kit94]{Kitagawa}
Koji Kitagawa, \emph{On standard {$p$}-adic {$L$}-functions of families of
  elliptic cusp forms}, {$p$}-adic monodromy and the {B}irch and
  {S}winnerton-{D}yer conjecture ({B}oston, {MA}, 1991), Contemp. Math., vol.
  165, Amer. Math. Soc., Providence, RI, 1994, pp.~81--110. \MR{1279604
  (95f:11031)}

\bibitem[Mau00]{Yves}
Yves Maurer, \emph{Zeros of $p$-adic ${L}$-functions}, 2000, Undergraduate
  project at Imperial College, available at
  http://wwwf.imperial.ac.uk/{\textasciitilde}buzzard/maths/research/notes/.

\bibitem[MTT86]{MazurTateTeitelbaum}
Barry Mazur, John Tate, and Jeremy Teitelbaum, \emph{On {$p$}-adic analogues of
  the conjectures of {B}irch and {S}winnerton-{D}yer}, Invent. Math.
  \textbf{84} (1986), no.~1, 1--48.

\bibitem[NP00]{NP}
Jan Nekov{\'a}{\v{r}} and Andrew Plater, \emph{On the parity of ranks of
  {S}elmer groups}, Asian J. Math. \textbf{4} (2000), no.~2, 437--497.
  \MR{1797592 (2001k:11078)}

\bibitem[PS11]{PS1}
Robert Pollack and Glenn Stevens, \emph{Overconvergent modular symbols and
  {$p$}-adic {$L$}-functions}, Ann. Sci. \'Ec. Norm. Sup\'er. (4) \textbf{44}
  (2011), no.~1, 1--42. \MR{2760194 (2012m:11074)}

\bibitem[PS13]{PS2}
\bysame, \emph{Critical slope {$p$}-adic {$L$}-functions}, J. Lond. Math. Soc.
  (2) \textbf{87} (2013), no.~2, 428--452. \MR{3046279}

\bibitem[Ste94]{Stevens-Rigid}
Glenn Stevens, \emph{Rigid analytic modular symbols}, 1994, Available at
  http://math.bu.edu/people/ghs/research.html.

\bibitem[Tri06]{Mak}
Mak Trifkovi{\'c}, \emph{Stark--{H}eegner points on elliptic curves defined
  over imaginary quadratic fields}, Duke Math. J. \textbf{135} (2006), no.~3,
  415--453. \MR{2272972 (2008d:11064)}

\bibitem[Urb11]{UrbanEigenvarieties}
Eric Urban, \emph{Eigenvarieties for reductive groups}, Ann. of Math. (2)
  \textbf{174} (2011), no.~3, 1685--1784.

\end{thebibliography}

\end{document}